\documentclass[12pt,a4paper]{book}
\usepackage{amsmath}
\usepackage{amssymb}
\usepackage{amsthm}

\usepackage{verbatim}

\usepackage[utf8]{inputenc}

\usepackage{graphicx}
\usepackage{color}
\usepackage{enumerate}

\usepackage{esint}
\usepackage[all]{xy}
\usepackage{framed}
\usepackage{bbm}
\usepackage{subcaption}
\usepackage{tabularx}

\theoremstyle{plain}
\newtheorem{theorem}{Theorem}[section]
\newtheorem{corollary}[theorem]{Corollary}
\newtheorem{lemma}[theorem]{Lemma}
\newtheorem{claim}{Claim}[section]
\newtheorem{proposition}[theorem]{Proposition}

\theoremstyle{definition}
\newtheorem{definition}{Definition}[chapter]
\newtheorem{remark}{Remark}[chapter]
\newtheorem{example}{Example}[chapter]
\numberwithin{equation}{section}

\newcommand{\bfa}{\mathbf{a}}

\newcommand{\bfp}{\mathbf{p}}
\newcommand{\bfe}{\mathbf{e}}
\newcommand{\Hu}{\mathcal{H}u}
\newcommand{\D}{\mathcal{D}}

\newcommand{\dd}{\mathrm{\,d}}
\newcommand{\hx}{\hat{x}}

\newcommand{\oll}{\overline{\lambda}}
\newcommand{\ola}{\overline{\alpha}}
\newcommand{\ols}{\overline{S}}
\newcommand{\uls}{\underline{S}}
\DeclareMathOperator{\tr}{tr}
\DeclareMathOperator{\diag}{diag}

\DeclareMathOperator{\dom}{dom}
\DeclareMathOperator{\dist}{dist}
\DeclareMathOperator{\con}{conv}
\DeclareMathOperator{\ext}{ext}

\begin{document}
\frontmatter

\author{Karl K. Brustad\\ {\small Norwegian University of Science and Technology}\\[1cm] \centering\includegraphics[width=5cm]{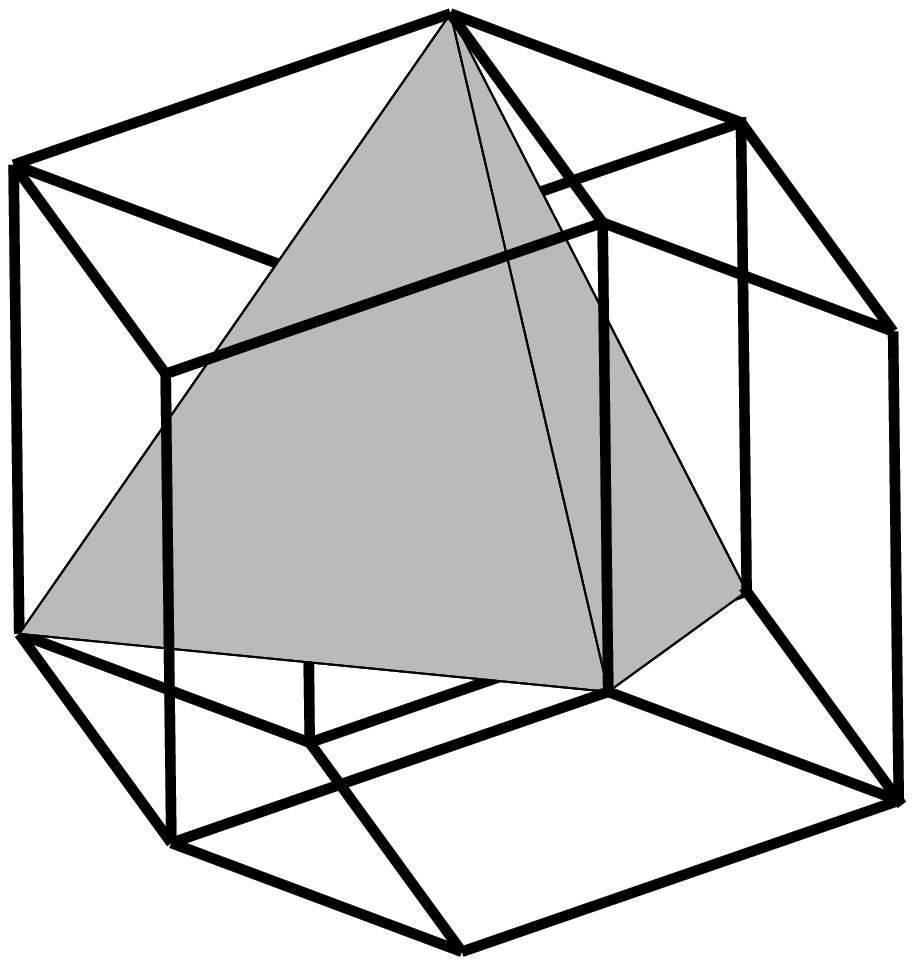}}
\title{Sublinear Elliptic Operators}

\maketitle

\tableofcontents

\chapter{Abstract}
We investigate second order elliptic equations
\[F(\mathcal{H}u) = 0\]
where the function \(F\colon S(n)\to\mathbb{R}\) on the space of symmetric \(n\times n\) matrices is assumed to be \emph{sublinear}.

There is very little to be found in the literature devoted particularly to sublinear elliptic operators. When examples of such operators occur, they are often merely treated as members of the larger class of \emph{convex} operators.
That class has been thoroughly investigated and many of its aspects are well understood.
The main reference is the book \cite{MR1351007} by Caffarelli and Cabré.

There is, however, something to be said about sublinear operators that do not, in general, apply to convex operators.

\chapter{Notation}

\begin{tabularx}{\textwidth}{ X X }
		\(\mathbb{R}^{n\times m}\) & The set of \(n\times m\) matrices with real entries. (\(n\) rows, \(m\) columns.)\\
		\(\mathbb{R}^n := \mathbb{R}^{n\times 1}\) & Euclidean \(n\)-dimensional space. \\
		\(X^T\) & The transposed of a matrix \(X\).\\
		\(S(n) := \{X\in\mathbb{R}^{n\times n}\;|\;X^T=X\}\) & The set of symmetric \(n\times n\) matrices.\\
		\(O(n) := \{Q\in\mathbb{R}^{n\times n}\;|\;Q^TQ=I\}\) & The set of orthogonal \(n\times n\) matrices.\\
		\(\bfe_i\) & The \(i\)'th standard basis vector of \(\mathbb{R}^n\).\\
		\((\bfe_{\pi(1)},\dots,\bfe_{\pi(n)})^T\) & The permutation matrix for a given permutation \(\pi\colon\{1,\dots,n\}\to\{1,\dots,n\}\).\\
		\(\mathcal{P}(n)\subseteq O(n)\) & The set of all \(n\times n\) permutation matrices.\\
		\(Pr(n) := \{P\in S(n)\;|\; PP = P\}\) & The set of \(n\times n\) orthogonal projection matrices.\\
		\(\diag x := \sum_{i=1}^n x_i \bfe_i\bfe_i^T\in S(n)\) & The diagonal matrix consisting of the entries \(x_i := x^T\bfe_i\) of \(x\in\mathbb{R}^n\).\\
		\(\diag X := \sum_{i=1}^n X_{ii}\bfe_i\in\mathbb{R}^n\) & The vector consisting of the diagonal elements \(X_{ii} := \bfe_i^TX\bfe_i\) of \(X\in\mathbb{R}^{n\times n}\).\\
		\(\lambda_1(X)\leq\cdots\leq\lambda_n(X)\) & The eigenvalues of a matrix \(X\in S(n)\). We also write \(\lambda_{\min}(X) :=\lambda_1(X)\), \(\lambda_{\max}(X) :=\lambda_n(X)\).\\
		\(\oll(X) := (\lambda_1(X),\dots,\lambda_n(X))^T\) & The vector in \(\mathbb{R}^n\) containing the eigenvalues of \(X\in S(n)\).\\
		\(\Lambda(X) := \diag\oll(X)\in S(n)\) & A diagonalization of \(X\in S(n)\).\\
		\(\tr X := \sum_{i=1}^n\bfe_i^T X\bfe_i\) & The trace of a matrix \(X\in\mathbb{R}^{n\times n}\).\\ 
		\(\langle Y,X\rangle := \tr (YX)\) & Inner product on \(S(n)\).\\
		\(\|X\| := \sqrt{\langle X,X\rangle}\) & Norm on \(S(n)\).\\
		\(|X| := \sup_{0\neq\xi\in\mathbb{R}^n}\frac{|X\xi|}{|\xi|}\) & The operator norm on \(\mathbb{R}^{n\times n}\).
\end{tabularx}

\begin{tabularx}{\textwidth}{ X X }
		\(X\leq Y\) & Partial ordering in \(S(n)\). It means \(\xi^TX\xi\leq \xi^TY\xi\) for all \(\xi\in\mathbb{R}^n\).\\
		\(X<Y\) & Means \(\xi^TX\xi < \xi^TY\xi\) for all \(\xi\in\mathbb{R}^n\) with \(|\xi|=1\).\\
		\(x^\uparrow\in\mathbb{R}^n\) & The vector obtained by permuting the components of \(x\in\mathbb{R}^n\) so that they are in increasing order.\\
		\(\mathbbm{1} := (1,\dots,1)^T\in\mathbb{R}^n\) &\\
		\(\mathcal{K}(n)\) & The set of convex bodies in \(S(n)\).\\
		\(\kappa(n)\) & The set of convex bodies in \(\mathbb{R}^n\).\\
\end{tabularx}

\mainmatter

%

\chapter{Introduction}

\section{Summary}
The object of study is second order elliptic equations \(F(\mathcal{H}u) = 0\) depending only on the Hessian matrix
\[\mathcal{H}u := \left(\frac{\partial^2 u}{\partial x_i\partial x_j}\right)_{i,j=1}^n.\]
We say that the equation, or operator, is \emph{sublinear} if the function \[F\colon S(n)\to\mathbb{R}\] on the space of symmetric \(n\times n\) matrices is sublinear. That is, for all \(X,Y\in S(n)\) and all numbers \(t\geq 0\), the following two conditions hold:
\begin{align}
F(X+Y) &\leq F(X) + F(Y),&&\text{(subadditive)}\label{subadd}\\
F(tX) &= tF(X),&&\text{(positive homogeneous).}\label{poshom}
\end{align}
One may easily check that
\begin{align*}
\text{sublinear}\qquad&\Rightarrow\qquad \text{convex,}\\
\text{convex + pos.hom.}\qquad&\Rightarrow\qquad \text{sublinear.}
\end{align*}
By Theorem 6.6 in \cite{MR1351007}, this implies in particular that
\begin{framed}\noindent
A viscosity solution to a sublinear and uniformly elliptic equation is \(C^{2,\alpha}\).
\end{framed}

Though \eqref{subadd} and \eqref{poshom} are restrictive, the family of sublinear operators is still numerous: The Laplacian, \(F(\mathcal{H}u) = \Delta u\), as well as any other homogeneous linear equation with constant coefficients, \(F(\mathcal{H}u) = \tr\left(A\mathcal{H}u\right)\), are obviously members. Nonlinear examples includes the well-known \emph{Pucci} operator
\begin{equation}
\D_{\lambda,\Lambda}u := P_{\lambda,\Lambda}(\mathcal{H}u) := \max\left\{\tr\left(A\mathcal{H}u\right)\;|\; \lambda I\leq A\leq \Lambda I\right\},\qquad 0<\lambda\leq \Lambda,
\label{puccidef}
\end{equation}
and the \emph{dominative \(p\)-Laplacian}
\begin{equation}
\D_p u := F_p(\mathcal{H}u) :=
\begin{cases}
(p-2)\lambda_{\min}(\mathcal{H}u) + \Delta u,\qquad &1\leq p\leq 2,\\
(p-2)\lambda_{\max}(\mathcal{H}u) + \Delta u, &2\leq p< \infty,\\
\lambda_{\max}(\mathcal{H}u), &p = \infty
\end{cases}
\label{Dpdefex}
\end{equation}
which was introduced in \cite{bru18}.
In fact, we shall see that there are as many sublinear operators as there are non-empty compact and convex subsets of \(S(n)\). That is, every sublinear operator is associated with a unique \emph{convex body} in \(S(n)\) (Theorem \ref{repsupthm}).
This is the expedient observation on which the theory is built.

In this survey, various properties of the sub- and supersolutions to sublinear elliptic equations \(F(\mathcal{H}u) = 0\) will be derived.
It is important to bear in mind that \emph{the subsolutions and the supersolutions are two qualitatively different classes of functions.} For example, if \(u\) is a supersolution, then \(-u\) is a subsolution (Theorem \ref{subsupsub}) but the converse is not true in general.
This is in contrast to linear operators -- and even the \(p\)-Laplacian --
where the sub- and supersolutions only differ by the sign.

We have chosen to divide the exposition into three main chapters. Chapter \ref{ch:deg} contains the most general theory for sublinear elliptic operators, while some natural special cases are investigated in Chapter \ref{ch:rot} and Chapter \ref{chp:unell}.
We start off, in Section \ref{sec:example_bodies}, by studying the convex bodies of the two nonlinear examples above; the Pucci operator and the dominative \(p\)-Laplacian.
In Section \ref{sec:compprin}, an optimal \emph{non-degeneracy} condition is provided in order to achieve the \emph{comparison principle}. In Section \ref{sec:suppos} we are concerned with properties of sums and differences of sub- and supersolutions. Any such property will be called a \emph{superposition principle}. In particular, we show that a sum of supersolutions is again a supersolution. The result is an important generalization of part 2 of Proposition 5 in \cite{bru18} as it completes the picture of the superposition principle in the \(p\)-Laplace equation.

Chapter \ref{ch:rot} is devoted to operators possessing a \emph{rotational invariance}. This symmetry condition enables an identification of sublinear operators with convex bodies in \(\mathbb{R}^n\), rather than in \(S(n)\). For example, the Pucci operator \(\D_{\lambda,\Lambda}\) corresponds to the cube \(\left\{(x_1,\dots,x_n)^T\;|\;\lambda\leq x_i\leq\Lambda\right\}\) in \(\mathbb{R}^n\).
Considering its definition \eqref{puccidef}, that is no surprise. It is less obvious, perhaps, that the dominative operator corresponds to a certain \((n-1)\)-\emph{simplex}. Some work is required in order to prove that the desired properties are preserved by this correspondence between convex bodies of \(S(n)\) and \(\mathbb{R}^n\).

With these preliminary results at hand, it is demonstrated that every rotationally invariant sublinear elliptic operator is bounded \emph{below} by the dominative \(p\)-Laplacian (for some \(p\in[2,\infty]\) that can be explicitly calculated.) This shows that \(\D_p\) plays the special rôle as the \emph{minimal operator} of its class. The nesting property (Proposition 5 (4) \cite{bru18}) yields an immediate but surprising corollary: Every supersolution to a rotationally invariant sublinear elliptic equation is \emph{superharmonic}.

The stronger notion of \emph{uniform} ellipticity is assumed in Chapter \ref{chp:unell}.
It turns out that a function is a supersolution to a uniformly elliptic sublinear equation if and only if every member of a class of functions -- obtained by a precomposition from a certain set of linear deformations of the domain -- is superharmonic. This again leads to the \emph{mean value} property in Section \ref{sec:meanvalue} and a result regarding convolutions in Section \ref{sec:convpos}. The latter is in some ways the continuous generalization of the superposition principle. It states that a convolution of a supersolution with a non-negative Radon measure is still a supersolution. Of course, some convergence conditions have to be imposed but we have made an effort in order to find the best ones possible.

Section \ref{sec:strongcomp} is based on the results of \cite{MR1351007}. We derive the \emph{strong comparison principle} in uniformly elliptic equations. Although this property is not exclusive to sublinear equations, it is included for the sake of completeness. Another issue is that \cite{MR1351007} \emph{defines} sub- and supersolutions to be \emph{continuous}. In contrast, we only require semi-continuity and our results are thus, in this sense, slightly generalized versions.

Finally, we list some notational conventions used in this survey. By \(\Omega\) we always mean a nonempty open subset of \(\mathbb{R}^n\) where \(n\geq 2\). Vectors are \emph{column} vectors, except gradients which are to be read as row vectors. e.g. \(\xi^T\xi\) is a number while \(\xi\xi^T\) is an \(n\times n\) matrix.
We refer to Appendix \ref{ch:visc_ell} for the definitions of viscosity solutions and the two kinds of ellipticity. Please be aware of that we mean \emph{degenerate} elliptic when we talk about ``elliptic'' equations. The stronger \emph{uniform} ellipticity will always be explicitly stated.
Also, the word ``viscosity'' will often be omitted since it is our sole notion of a (sub/super)-solution to an equation.

%
%

\section{The convex body}\label{sec:conv_body}
By introducing the \emph{inner product}
\begin{equation}
\langle X,Y\rangle := \tr (XY),
\label{inprod}
\end{equation}
we turn \(S(n)\) into an \(N := \frac{n(n+1)}{2}\)-dimensional Hilbert space.
The induced norm is
\[\|X\| = \sqrt{\langle X,X\rangle} = \sqrt{\tr X^2} = \sqrt{\lambda_1^2 + \cdots + \lambda_n^2}\]
where \(\lambda_1,\dots,\lambda_n\) are the eigenvalues of \(X\).

Let \(\mathcal{A}\subseteq S(n)\) be a nonempty set of symmetric \(n\times n\) matrices. The \textbf{support function} of \(\mathcal{A}\) is the function \( F_{\mathcal{A}}\colon S(n)\to(-\infty,\infty]\) defined by
\begin{equation}
 F_{\mathcal{A}}(X) := \sup_{Y\in\mathcal{A}}\langle Y, X\rangle.
\label{suppdef}
\end{equation}
Usually, we will only consider \emph{bounded} subsets \(\mathcal{A}\). That ensures that \( F_{\mathcal{A}}\) does not take infinite values for some finite \(X\). Similarly, the nonemptyness makes \( F_{\mathcal{A}}>-\infty\). Note that it is immediate from the definition that \( F_{\mathcal{A}}\) is a sublinear function on \(S(n)\).

To gain some intuition on how support functions work, we observe first that the positive homogenicity uniquely determines \( F_{\mathcal{A}}\) by its values on the unit sphere in \(S(n)\): \( F_{\mathcal{A}}(X) = \|X\| F_{\mathcal{A}}\left(\frac{X}{\|X\|}\right)\). Consider therefore a unit length matrix \(X\). Imagine the line \(\{tX\;|\;t\in\mathbb{R}\}\) and the hyperplanes in \(S(n)\) which intersect it perpendicularly. They are naturally parametrized by \(t\). Now, let a hyperplane come in from \(t=\infty\) and move it towards \(t=-\infty\) until it hits \(\mathcal{A}\). The value of \( F_{\mathcal{A}}\) at \(X\) is thus the corresponding \(t\) -- the signed distance from the origin to the hyperplane.

Our indulgence to these matters is justified by the following classical result: \emph{Every sublinear function is a support function}.
See \cite{MR3155183}, Theorem 1.7.1 and the foregoing discussion.\footnote{This theorem in \cite{MR3155183}, as well as the other results we refer to from this reference, is stated in terms of subsets (and functions) in \(\mathbb{R}^n\). This is, however, insignificant since the inner product \eqref{inprod} makes \(S(n)\cong\mathbb{R}^N\).}

\begin{theorem}\label{repsupthm}
Let \(F\colon S(n)\to\mathbb{R}\) be sublinear. Then there exists a compact convex set \(\emptyset\neq\mathcal{K}\subseteq S(n)\) such that
\[F(X) = \sup_{Y\in\mathcal{K}}\langle Y,X\rangle = \max_{Y\in\mathcal{K}}\langle Y,X\rangle.\]
Moreover, the convex set is unique and is given by
\[\mathcal{K} = \left\{Z\in S(n)\;|\; \langle Z,X\rangle  \leq F(X)\;\forall X\in S(n)\right\}.\]
\end{theorem}

We shall call the set \(\mathcal{K}\) from Theorem \ref{repsupthm} above for the \emph{the associated convex body} to the operator \(F\), and sometimes refer to it as \(\mathcal{K}_F\). In general, we say that a subset \(\mathcal{K}\subseteq S(n)\) is a \emph{convex body} if it is nonempty, convex and compact. Observe that a sublinear function is \emph{linear} if and only if \(\mathcal{K}\) is a singleton.

Since we are only interested in elliptic equations, we have to know what restrictions this inflicts on \(\mathcal{K}\):

\begin{proposition}\label{definite}
A sublinear operator \(F\colon S(n)\to\mathbb{R}\) is
\begin{itemize}
	\item degenerate elliptic if and only if each element in \(\mathcal{K}_F\) is positive semidefinite.
	\item uniformly elliptic if and only if each element in \(\mathcal{K}_F\) is positive definite.
\end{itemize}
\end{proposition}

See Appendix \ref{ch:visc_ell} for the proof.

\section{Basic convexity}
The \textbf{convex hull} of a subset \(\mathcal{A}\subseteq S(n)\) is denoted \(\con\mathcal{A}\). It is the smallest convex set containing \(\mathcal{A}\). Equivalently, it is the set of \emph{convex combinations} of points in \(\mathcal{A}\):
\begin{equation}
\con\mathcal{A} := \left\{\sum_{i=1}^M\alpha_i Y_i\;\middle|\; M\in\mathbb{N},\,\alpha_i\geq 0,\,\sum_{i=1}^M\alpha_i = 1,\,Y_i\in \mathcal{A}\right\}.
\label{hulldef}
\end{equation}

As, perhaps, is intuitively clear, a support function does not distinguish a subset from its convex hull:
\begin{proposition}\label{convprop}
Let \(\mathcal{A}\subseteq S(n)\) be nonempty. Then
\[ F_{\mathcal{A}} =  F_{\con\mathcal{A}}.\]
\end{proposition}

\begin{proof}
Define the set
\[\mathcal{A}' := \left\{Z\in S(n)\;|\; \langle Z,X\rangle \leq  F_{\mathcal{A}}(X)\quad\forall X\in S(n)\right\}.\]
Let \(Z = \sum_{i=1}^M\alpha_i Z_i \in \con\mathcal{A}\) and let \(X\in S(n)\). Then
\begin{align*}
\langle Z,X\rangle &= \sum_{i=1}^M\alpha_i\langle Z_i,X\rangle\\
                   &\leq \sum_{i=1}^M\alpha_i\sup_{Y\in\mathcal{A}}\langle Y,X\rangle\\
									 &= \sum_{i=1}^M\alpha_i  F_{\mathcal{A}}(X) =  F_{\mathcal{A}}(X)
\end{align*}
and \(\con \mathcal{A}\subseteq \mathcal{A}'\). Thus
\begin{align*}
 F_{\mathcal{A}}(X) = \sup_{Y\in\mathcal{A}}\langle Y,X\rangle \leq \sup_{Y\in\con\mathcal{A}}&\langle Y,X\rangle\leq \sup_{Y\in\mathcal{A}'}\langle Y,X\rangle \leq  F_{\mathcal{A}}(X)\\
	||&\\
 F_{\con\mathcal{A}}&(X).
\end{align*}
\end{proof}

In light of Proposition \ref{convprop} one may wonder if, given a sublinear operator \(F(X) = \max_{Y\in\mathcal{K}}\langle Y,X\rangle\), there exists a \emph{smallest} set \(\mathcal{A}\subseteq\mathcal{K}\) such that we still have \(\max_{Y\in\mathcal{A}}\langle Y,X\rangle = F(X)\) for all \(X\in S(n)\).

\begin{definition}\label{extremedef}
Let \(\mathcal{K}\) be a convex body. A matrix \(E\in\mathcal{K}\) is an \textbf{extreme point} of \(\mathcal{K}\) if \(\mathcal{K}\setminus\{E\}\) is convex. The set of extreme points of \(\mathcal{K}\) is denoted \(\ext\mathcal{K}\).
\end{definition}

If we define the set
\begin{equation}
\mathcal{K}^\circ := \left\{tY + (1-t)Z\;|\; 0<t<1,\, Y,Z\in\mathcal{K},\, Y\neq Z\right\},
\label{Kcircdef}
\end{equation}
we also have the alternative definition
\[\ext\mathcal{K} = \mathcal{K}\setminus\mathcal{K}^\circ.\]
See Section 1.4 in \cite{MR3155183}.

Part 2 and 3 from the list below of elementary results in convex geometry, shows that the extreme points indeed is this smallest subset of \(\mathcal{K}\) that still determines the operator. Knowing the extreme points is important because they constitute a minimal set of data needed to evaluate \(F\) at a given Hessian.


\begin{proposition}\label{extconvbody}
Let \(\mathcal{K}\) be a convex body and let \(\mathcal{A}\) be compact. Then
\begin{enumerate}
	\item \(\con\mathcal{A}\) is compact.
	\item \(\con\ext\mathcal{K} = \mathcal{K}\)\;(Minkowski's theorem)
	\item If \(\con\mathcal{A} = \mathcal{K}\), then \(\ext\mathcal{K}\subseteq\mathcal{A}\).
	\item (Separating hyperplane) If \(X\notin\mathcal{K}\), then there exists a unique point \(Y_0\) on the boundary of \(\mathcal{K}\) such that \(\dist(X,\mathcal{K}) = \|X-Y_0\|\) and
	\[\langle Y-Y_0, X-Y_0\rangle \leq 0\qquad \forall Y\in\mathcal{K}.\]
\end{enumerate}
\end{proposition}

\begin{proof}
Part 1 and part 2 is Theorem 1.1.11 and Corollary 1.4.5 in \cite{MR3155183}, respectively.
For part 3, let \(\con\mathcal{A} = \mathcal{K}\) but assume that there is a matrix \(E\in\ext\mathcal{K}\) not in \(\mathcal{A}\). Then \(\mathcal{A}\subseteq\mathcal{K}\setminus\{E\}\) and we get the contradiction \(\con\mathcal{A}\subseteq\mathcal{K}\setminus\{E\}\) since \(\mathcal{K}\setminus\{E\}\) is convex.

Part 4 is standard theory: The point \(Y_0\) is the \emph{metric projection} of \(\mathcal{K}\) at \(X\). The inequality then states that the convex body lies entirely in the half-space defined by the hyperplane through \(Y_0\) orthogonal to \(X-Y_0\).
\end{proof}

We now present a way to make new sublinear operators from existing ones.

Let \(A\) and \(B\) be subsets of a vector space. The \textbf{Minkowski sum} of \(A\) and \(B\) is the set
\[A+B := \left\{a+b \;|\; a\in A,\, b\in B\right\}.\]
A set may also be scaled in an obvious way:
\[cA := \left\{ca \;|\; a\in A\right\},\qquad c\in\mathbb{R}.\]

\begin{proposition}\label{mincsum}
Let \(F\) and \(G\) be sublinear operators and let \(\alpha\) and \(\beta\) be non-negative constants. Then \(H := \alpha F + \beta G\) is a sublinear operator with convex body
\[\mathcal{K}_H = \alpha\mathcal{K}_F + \beta\mathcal{K}_G.\]
Furthermore, the function \(H(X) := F(-X)\) is sublinear with convex body
\[\mathcal{K}_H = -\mathcal{K}_F := -1\cdot\mathcal{K}_F.\]
\end{proposition}

\begin{proof}
It is easily checked that the functions \(\alpha F\) and \(\beta G\) are sublinear with convex bodies \(\alpha\mathcal{K}_F\), \(\beta\mathcal{K}_G\) when \(\alpha,\beta\geq 0\). The first claim then follows from Theorem 1.7.5 (a) in \cite{MR3155183}. Also, \(-\mathcal{K}_F\) is a convex body and
\[H(X) = F(-X) = \max_{Y\in\mathcal{K}_F}\langle Y,-X\rangle = \max_{Y\in\mathcal{K}_F}\langle -Y,X\rangle = \max_{Y\in-\mathcal{K}_F}\langle Y,X\rangle.\]
\end{proof}

The extreme points are invariant under scaling and shifts.

\begin{proposition}\label{extscaleshiftinv}
Let \(\mathcal{K}\) be a convex body and let \(A\in S(n)\) and \(\alpha\in\mathbb{R}\). Then
\[\ext\Big(\{A\} + \alpha\mathcal{K}\Big) = \{A\} + \alpha\ext\mathcal{K}.\]
\end{proposition}

\begin{proof}
The claim is obviously true for \(\alpha=0\) since \(0\cdot\mathcal{K} = 0\cdot\ext\mathcal{K} = \{0\}\) and \(\{A\} + \{0\} = \{A\}\).
The set \(\mathcal{K}' := \{A\}+\alpha\mathcal{K}\) is a convex body by Proposition \ref{mincsum} and
\begin{align*}
(\mathcal{K}')^\circ
	&= \left\{tY + (1-t)Z\;|\; t\in(0,1),\, Y\neq Z\in\mathcal{K}'\right\}\\
	&= \left\{t(A + \alpha Y) + (1-t)(A + \alpha Z)\;|\; t\in(0,1),\, Y\neq Z\in\mathcal{K}\right\}\\
	&= \left\{A + \alpha(tY + (1-t)Z)\;|\; t\in(0,1),\, Y\neq Z\in\mathcal{K}\right\}\\
	&= \{A\} + \alpha\mathcal{K}^\circ.
\end{align*}
It follows that
\begin{align*}
\ext\mathcal{K}' &= \mathcal{K}'\setminus(\mathcal{K}')^\circ\\
                 &= \left(\{A\} + \alpha\mathcal{K}\right)\setminus\left(\{A\} + \alpha\mathcal{K}^\circ\right)\\
								 &= \left\{A+\alpha Y\;|\; Y\in\mathcal{K}\right\}\setminus\left\{A+\alpha Z\;|\; Z\in\mathcal{K}^\circ\right\}\\
								 &= \left\{A+\alpha Y\;|\; Y\in\mathcal{K},\, A+\alpha Y \neq A+\alpha Z\,\forall Z\in\mathcal{K}^\circ\right\}\\
								 &= \left\{A+\alpha Y\;|\; Y\in\mathcal{K},\, Y \notin\mathcal{K}^\circ\right\}\\
								 &= \{A\} + \alpha\left(\mathcal{K}\setminus\mathcal{K}^\circ\right) = \{A\} + \alpha\ext\mathcal{K}.
\end{align*}

\end{proof}

\chapter{Degenerate ellipticity}\label{ch:deg}

\section{The bodies of Pucci and dominative p-Laplace}\label{sec:example_bodies}
For constants \(0<\lambda\leq\Lambda\), define
\begin{equation}
\mathcal{K}_{\lambda,\Lambda} := \left\{Y\in S(n)\;|\; \lambda I\leq Y\leq \Lambda I\right\}.
\label{ex_pucci_body}
\end{equation}
The Pucci operator is then \(\D_{\lambda,\Lambda}u := P_{\lambda,\Lambda}(\mathcal{H}u)\) where
\[P_{\lambda,\Lambda}(X) = \max_{Y\in\mathcal{K}_{\lambda,\Lambda}}\langle Y,X\rangle.\]
As \(\lambda_{\min}(A+B) \geq \lambda_{\min}(A) + \lambda_{\min}(B)\) and \(\lambda_{\max}(A+B) \leq \lambda_{\max}(A) + \lambda_{\max}(B)\), we get that \(\mathcal{K}_{\lambda,\Lambda}\) is a convex subset of \(S(n)\). Since it obviously also is nonempty and compact, it follows that \(\mathcal{K}_{\lambda,\Lambda}\) is the associated convex body of \(P_{\lambda,\Lambda}\). A short survey of the Pucci operator is found in Appendix \ref{ch:visc_ell}.

Now let \(p\in[1,\infty]\) and let \(\mathcal{K}_p\) denote the associated convex body to the dominative \(p\)-Laplacian, \(F_p(\mathcal{H}u) = \D_pu\). If we define the compact, but non-convex, subset of \(S(n)\),
\begin{align*}
\mathcal{E}_p &:= \left\{I + (p-2)\xi\xi^T\;|\; \xi\in\mathbb{R}^n,\; |\xi| = 1\right\},\qquad 1\leq p<\infty,\\
\mathcal{E}_\infty &:= \left\{\xi\xi^T\;|\; \xi\in\mathbb{R}^n,\; |\xi| = 1\right\},
\end{align*}
we know that\footnote{See Remark 1 in \cite{bru18}.}
\[F_p(X) = \max_{Y\in\mathcal{E}_p}\langle Y,X\rangle\]
and thus \(\mathcal{K}_p = \con\mathcal{E}_p\) by Proposition \ref{extconvbody} part 1, Proposition \ref{convprop}, and the uniqueness of the convex body. We shall see later (Remark \ref{remarkgenerated}) that \(\mathcal{E}_p\) is, in fact, the extreme points of \(\mathcal{K}_p\).

We claim that
\[\mathcal{K}_\infty = \left\{Y\in S(n)\;|\; 0\leq Y\leq I,\; \tr Y = 1\right\}.\]
Indeed, let \(Y\in\con\mathcal{E}_\infty\). That is, \(Y\) is a convex combination on the form \(Y = \sum_{i=1}^M\alpha_i\xi_i\xi_i^T\) where \(|\xi_i|=1\). Then
\[\tr Y = \tr\left(\sum_{i=1}^M\alpha_i\xi_i\xi_i^T\right) = \sum_{i=1}^M\alpha_i\tr\left(\xi_i\xi_i^T\right) = \sum_{i=1}^M\alpha_i = 1,\]
and if \(|z|=1\), then
\[z^TYz = \sum_{i=1}^M\alpha_i(z^T\xi_i)^2
\begin{cases}
\leq 1,\\
\geq 0.
\end{cases}\]
For the other inclusion, assume that \(Y\in S(n)\) satisfies \(0\leq Y\leq I\) and \(\tr Y = 1\). The symmetric matrix \(Y\) can be written in terms of its eigenvalues and unit length eigenvectors as \(Y=\sum_{i=1}^n\lambda_i\xi_i\xi_i^T\). This means that \(0\leq\lambda_i\leq 1\), \(\sum_{i=1}^n\lambda_i = 1\), and \(Y\) is surely a convex combination of points in \(\mathcal{E}_\infty\).

We now compute \(\mathcal{K}_p\) for \(1\leq p<\infty\). Assume first that \(p\geq 2\). Then
\[F_p(X) = \tr X + (p-2)\lambda_{\max}(X) = F_2(X) + (p-2)F_\infty(X)\]
and
\[\mathcal{K}_p = \mathcal{K}_2 + (p-2)\mathcal{K}_\infty\]
by Propostion \ref{mincsum}.
If \(1\leq p \leq 2\), then
\[F_p(X) = \tr X + (p-2)\lambda_{\min}(X) = F_2(X) + (2-p)F_\infty(-X)\]
and
\(\mathcal{K}_p = \mathcal{K}_2 + (2-p)(-\mathcal{K}_\infty) = \mathcal{K}_2 + (p-2)\mathcal{K}_\infty\) aslo in this case.

Since \(\mathcal{K}_2\) is just the singleton \(\{I\}\) it follows that \(\mathcal{K}_p\) is nothing but a scaling and a shift of \(\mathcal{K}_\infty\):
\begin{equation}
\mathcal{K}_p = \{I\} + (p-2)\mathcal{K}_\infty, \qquad 1\leq p<\infty.
\label{Kpscaleshift}
\end{equation}

Doing the addition, we find that the associated convex body to the dominative \(p\)-Laplace operator is
\begin{equation}
\mathcal{K}_p =
\begin{cases}
\left\{Y\;|\; (p-1)I\leq Y\leq I,\; \tr Y = n+p-2\right\},\qquad &\text{if \(1\leq p\leq 2\),}\\
\left\{Y\;|\; I\leq Y\leq (p-1)I,\; \tr Y = n+p-2\right\},&\text{if \(2\leq p<\infty\),}\\
\left\{Y\;|\; 0\leq Y\leq I,\; \tr Y = 1\right\},&\text{if \(p=\infty\).}
\end{cases}
\label{dombody}
\end{equation}
We know that \(\lim_{p\to\infty}\frac{1}{p}F_p = F_\infty\), and, indeed, we see that \(\lim_{p\to\infty}\frac{1}{p}\mathcal{K}_p = \mathcal{K}_\infty\) (in any reasonable sense).

It is also useful to define a sort of \(\infty\)-Pucci operator in the same manner: Allowing \(\lambda\) to go all the way down to zero in \eqref{ex_pucci_body}, we get
\begin{equation}
\lim_{\Lambda\to\infty}\frac{1}{\Lambda}P_{\lambda,\Lambda} = P_{0,1}
\label{infinitypucci}
\end{equation}
since \(c\mathcal{K}_{\lambda,\Lambda} = \mathcal{K}_{c\lambda,c\Lambda}\) for non-negative constants \(c\).
It is defined by the convex body
\[\mathcal{K}_{0,1} = \left\{Y\in S(n)\;|\; 0\leq Y\leq I\right\}\]
and the operator is only degenerate elliptic. Like in the dominative case, the Pucci body can now be decomposed into a scaling and a shift of its \(\infty\)-body: Using that \(\mathcal{K}_{a,b} + \mathcal{K}_{\lambda,\Lambda} = \mathcal{K}_{a+\lambda,b+\Lambda}\) and \(\mathcal{K}_{1,1} = \{I\}\), yields
\begin{equation}
\mathcal{K}_{\lambda,\Lambda} = \lambda\{I\} + (\Lambda-\lambda)\mathcal{K}_{0,1},\qquad 0\leq\lambda\leq\Lambda.
\label{pucscaleshift}
\end{equation}

Notice that \(\mathcal{K}_p\) is the intersection between a Pucci body and a hyperplane in \(S(n)\). For example, when \(2\leq p<\infty\),
\[\mathcal{K}_p = \mathcal{K}_{1,p-1}\cap \left\{\tr Y = n+p-2\right\} = \{I\} + (p-2)\Big(\mathcal{K}_{0,1}\cap \left\{\tr Y = 1\right\}\Big)\]
and the operators may be written in a way that make them appear to be very similar:
\begin{align*}
\D_{1,p-1}u &= \Delta u + (p-2)\max_{Y\in\mathcal{K}_{0,1}}\tr\left(Y\mathcal{H}u\right),\\
\D_p u &= \Delta u + (p-2)\max_{\substack{Y\in\mathcal{K}_{0,1}\\ \tr Y = 1}}\tr\left(Y\mathcal{H}u\right),\qquad 2\leq p<\infty.
\end{align*}

\section{Body cones and solution cones}
Assume that \(F\) and \(G\) are two sublinear operators with associated convex bodies \(\mathcal{K}_F\) and \(\mathcal{K}_G\). It is obvious from the definition that \(\mathcal{K}_F\subseteq\mathcal{K}_G\) implies \(F(X)\leq G(X)\) and thus a supersolution of \(G(\mathcal{H}u)=0\) is also a supersolution of \(F(\mathcal{H}u)=0\). But as we shall see, inclusion of the convex bodies is not a \emph{necessary} condition to ensure this nesting property of supersolutions: It is the ``width'' of the convex body as seen from the origin that matters.

\begin{definition}\label{bodycone}
A subset \(C\) of a vector space is a \textbf{convex cone} if \(X,Y\in C\) and \(\alpha,\beta\geq 0\) implies
\[\alpha X + \beta Y \in C.\]
The \textbf{body cone} of a convex body \(\mathcal{K}\) is the closed convex cone
\[C_{\mathcal{K}} := \left\{tZ\;|\; Z\in\mathcal{K},\;t\geq 0\right\}.\]
Given a convex cone \(C\), we define the \textbf{dual cone} as
\[C^\circ := \{X\in S(n)\;|\; \langle Y,X\rangle\leq 0\quad\forall Y\in C\}.\]
\end{definition}
Note that for a sublinear operator, \(F(X) = \max_{Y\in\mathcal{K}}\langle Y,X\rangle\), we have
\begin{equation}
F(X)\leq 0\qquad\text{if and only if}\qquad X\in C_{\mathcal{K}}^\circ.
\label{equivsolcone}
\end{equation}
We shall therefore call \(C_{\mathcal{K}}^\circ\) for the \textbf{solution cone} of the operator \(F\). See Figure \ref{fig:bodycones}

\begin{figure}[ht]%
\center
\includegraphics{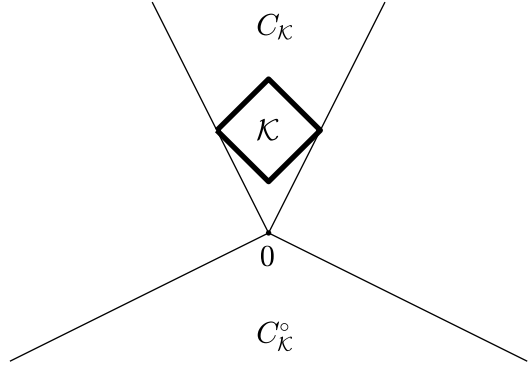}%
\caption{The body cone \(C_{\mathcal{K}}\) and the solution cone \(C_{\mathcal{K}}^\circ\) to the operator \(F\) with associated convex body \(\mathcal{K}\). On the boundary of \(C_{\mathcal{K}}^\circ\) we have \(F=0\).}%
\label{fig:bodycones}%
\end{figure}

\begin{lemma}\label{duallem}
Let \(C_1, C_2\) be closed convex cones. Then
\[C_1\subseteq C_2\qquad\iff\qquad C_2^\circ\subseteq C_1^\circ.\]
\end{lemma}

\begin{proof}
Assume that \(C_1\subseteq C_2\) and that \(X\in C_2^\circ\). Then \(\langle Y,X\rangle\leq 0\) for all \(Y\in C_2\). Thus \(\langle Y,X\rangle\leq 0\) for all \(Y\in C_1\) and \(X\in C_1^\circ\).

The leftward implication then follows from the fact that \(C^{\circ\circ} = C\) for closed convex cones \(C\). See section 1.6 in \cite{MR3155183}.
\end{proof}

\begin{theorem}[Nesting property of sub- and supersolutions]\label{nesting}
Consider two sublinear and degenerate elliptic equations \(F(\mathcal{H}w)=0\) and \(G(\mathcal{H}w)=0\) in \(\Omega\subseteq\mathbb{R}^n\) and let \(\mathcal{K}_F\), \(\mathcal{K}_G\) be the associated convex bodies. Then
\[
\begin{gathered}
C_{\mathcal{K}_F}\subseteq C_{\mathcal{K}_G}\\
\iff\\
G(\mathcal{H}u)\leq 0 \qquad\text{implies}\qquad F(\mathcal{H}u)\leq 0\\
\text{in the viscosity sense in \(\Omega\).}\\
\iff\\
F(\mathcal{H}v)\geq 0 \qquad\text{implies}\qquad G(\mathcal{H}v)\geq 0\\
\text{in the viscosity sense in \(\Omega\).}
\end{gathered}
\]
\end{theorem}

\begin{proof}
We only prove the equivalence of the first and second claim. The proof of the equivalence of the first and third claim is analogous.

Assume that \(C_{\mathcal{K}_F}\subseteq C_{\mathcal{K}_G}\) and that \(u\) is a supersolution to \(G(\mathcal{H}w)=0\) in \(\Omega\). Let \(\phi\) be a test function to \(u\) from below at some point \(x_0\in\Omega\) and write \(X:=\mathcal{H}\phi(x_0)\). Then \(G(X) \leq 0\). That is, \(X\in C_{\mathcal{K}_G}^\circ\subseteq C_{\mathcal{K}_F}^\circ\) by \eqref{equivsolcone} and Lemma \ref{duallem}. Thus \(F(X)\leq 0\) and \(u\) is a viscosity supersolution to \(F(\mathcal{H}w) = 0\) in \(\Omega\).

Now assume that \(C_{\mathcal{K}_F}\) is not contained in \(C_{\mathcal{K}_G}\). We have to find a supersolution to \(G(\mathcal{H}w) = 0\) that is not a supersolution to \(F(\mathcal{H}w)=0\). By Lemma \ref{duallem}, \(C_{\mathcal{K}_G}^\circ\) is not contained in \(C_{\mathcal{K}_F}^\circ\) and we may therefore find a matrix \(X_0\) such that
\[C_{\mathcal{K}_G}^\circ\ni X_0\notin C_{\mathcal{K}_F}^\circ.\]
Let \(\phi\) be the quadratic \(\phi(x) := \frac{1}{2}x^T X_0 x\) in \(\Omega\). Then \(\mathcal{H}\phi = X_0\) and it follows from \eqref{equivsolcone} that
\[G(X_0)\leq 0 < F(X_0).\]
\end{proof}

The associated convex body to the Laplace operator is the singleton \(\{I\}\). The next result is therefore an immediate corollary.
\begin{corollary}\label{supharmcor}
Every supersolution to a sublinear and degenerate elliptic equation is superharmonic if and only if the associated convex body contains a positively scaled identity matrix.
\end{corollary}

\begin{example}[Nesting property in the dominative \(p\)-Laplace equation]
If \(2\leq p\leq q<\infty\), one may confirm from \eqref{dombody} that
\[\frac{1}{n}\mathcal{K}_2\subseteq \frac{1}{n+p-2}\mathcal{K}_p\subseteq \frac{1}{n+q-2}\mathcal{K}_q\subseteq \mathcal{K}_\infty.\]
Thus
\[C_{\mathcal{K}_2}\subseteq C_{\mathcal{K}_p}\subseteq C_{\mathcal{K}_q}\subseteq C_{\mathcal{K}_\infty}\]
and Proposition 5 (4) in \cite{bru18} follows from Theorem \ref{nesting}.
\end{example}

\section{Comparison principles}\label{sec:compprin}
In order to have any hope to prove a \emph{comparison principle} -- which, for example, is needed to show uniqueness in the Dirichlet problem -- some kind of \emph{non-degeneracy condition} of the operator is necessary. (One may think of the sublinear and degenerate elliptic operator \(F\equiv 0\) in which \emph{every} continuous function is a solution to the equation \(F(\mathcal{H}u) = 0\).)

Uniform ellipticity will do, but it turns out that a much less restrictive requirement on \(F\) is sufficient:
\begin{equation}
\exists X\in S(n)\qquad\text{such that}\qquad F(X) < 0.
\label{nondegcond}
\end{equation}
Since \(F(0) = 0\), it is clear from ellipticity that such an \(X = \sum_{i=1}^n\lambda_i\xi_i\xi_i^T\) is not positive semidefinite. That is, \(\lambda_1<0\). As \(\lambda_1 I\leq X\), positive homogenicity and ellipticity gives \(-\lambda_1 F(-I) = F(\lambda_1 I) \leq F(X) < 0\). Thus \eqref{nondegcond} is equivalent to
\begin{equation}
F(-I) < 0.
\label{Fneg}
\end{equation}

Furthermore, if we consider a negative definite matrix \(X = \sum_{i=1}^n\lambda_i\xi_i\xi_i^T\), \(\lambda_n<0\). Then \(X\leq\lambda_nI\) and \(F(X) \leq F(\lambda_nI) = -\lambda_nF(-I)<0\) and \eqref{nondegcond} is actually equivalent to
\begin{equation}
X<0\qquad\Rightarrow\qquad F(X)<0.
\label{negdefcond}
\end{equation}

In terms of the associated convex body of \(F\), we shall prove that \(F\geq 0\) if and only if the body contains the zero-matrix. That is,
\begin{equation}
\eqref{nondegcond}\qquad\iff\qquad 0\notin\mathcal{K}.
\label{zerocond}
\end{equation}
\begin{proof}
If \(0\in\mathcal{K}\) then
\[F(X) = \max_{Y\in\mathcal{K}}\langle Y,X\rangle \geq \langle 0,X\rangle = 0\]
for every \(X\in S(n)\).

Assume \(0\notin\mathcal{K}\). Since \(\mathcal{K}\) is closed, the distance from 0 to \(\mathcal{K}\) is positive. Moreover, by part 4 of Proposition \ref{extconvbody}, there is a matrix \(Y_0\in\mathcal{K}\) so that
\[0<\dist(0,\mathcal{K}) = \|0-Y_0 \|\]
and
\[\langle Y-Y_0,0-Y_0\rangle \leq 0\]
for all \(Y\in \mathcal{K}\). Writing the left-hand side as \(\|Y_0\|^2 - \langle Y,Y_0\rangle\), it follows that
\[F(-Y_0) = \max_{Y\in\mathcal{K}}\Big(-\langle Y,Y_0\rangle\Big) = -\|Y_0\|^2 < 0.\]
\end{proof}

From \eqref{zerocond}, we see that the \(\infty\)-Pucci operator \eqref{infinitypucci} is the only operator we have encountered so far that does not satisfy the non-degeneracy condition.

The condition \eqref{nondegcond} is also optimal in the sense that if it does \emph{not} hold, then no comparison principle is possible: If \eqref{nondegcond} is not true, then every \(C^2\) function will be a subsolution. One may therefore compare the subsolution \(v(x) = 1-|x|^2\) to the supersolution \(u = 1/2\) in the unit ball to reach a counter example.

In the proof of the comparison principle, we regularize the subsolutions and the supersolutions by sup- and inf-convolutions, respectively. See Definition \ref{recal}. Unfortunately, this will \emph{not} produce \(C^2\)-functions, but functions that are merely \emph{semiconvex} and \emph{semiconcave}. We then employ the classical result of Alexandrov, stating that such functions are \emph{twice differentiable} almost everywhere.

\begin{definition}\label{semicondef}
A function is said to be \textbf{semiconcave} (with concavity constant \(\mu\geq 0\)) if it at each point in its domain it can be touched above by a paraboloid \(\phi\) with \(\mathcal{H}\phi = \mu I\). A function is semiconvex if its negative is semiconcave.
\end{definition}

\begin{definition}
A function \(u\) in \(\Omega\) is \textbf{twice differentiable at} \(x_0\in\Omega\) if there exists a \(q\in\mathbb{R}^n\) and a \(X\in S(n)\) such that
\begin{equation}
u(x) = u(x_0) + q^T(x-x_0) + \frac{1}{2}(x-x_0)^TX(x-x_0) + o(|x-x_0|^2)
\label{alexder}
\end{equation}
as \(x\to x_0\).
We shall call \(q\) and \(X\) for the Alexandrov derivatives of \(u\) at \(x_0\) and even write \(q^T=\nabla u(x_0)\), \(X = \mathcal{H}u(x_0)\).
\end{definition}

\begin{theorem}[Alexandrov]\label{alexandrov}
A semiconcave function is twice differentiable at almost every point in open domains.
\end{theorem}

A very important fact about the Alexandrov derivatives, is that they behave correctly as arguments in an elliptic operator (see Proposition 2.6 in \cite{MR2084272}):
\begin{lemma}\label{twicedifflem}
Suppose that \(u\) is a supersolution to an elliptic equation \(F(x,w,\nabla w,\mathcal{H}w)=0\). If \(u\) is twice differentiable at a point \(x_0\), that is, there is a \(q\in\mathbb{R}^n\) and an \(X\in S(n)\) satisfying \eqref{alexder}, then
\[F(x_0,u(x_0),q,X)\leq 0.\]
The analogous result also holds for subsolutions.
\end{lemma}

We shall make use of \emph{Jensen's Lemma} which again relies on the \emph{convex envelope} and the \emph{contact set} as defined below.

\begin{definition}
Let \(D\subseteq\mathbb{R}^n\) be bounded and let \(w\colon D\to\mathbb{R}\). The \textbf{convex envelope} of \(w\) is the largest convex function \(\Gamma\) defined in the convex hull of \(D\) that is smaller or equal to \(w\) in \(D\):
\[\Gamma(x) := \sup\{c + q^Tx\;|\; c\in\mathbb{R},\; q\in\mathbb{R}^n,\; c + q^Ty \leq w(y)\;\forall y\in D\}.\]
\end{definition}

\begin{definition}
The \textbf{contact set} \(\mathcal{A} = \mathcal{A}(w)\) is the set of points where a function \(w\) agrees with its convex envelope \(\Gamma\):
\[\mathcal{A}(w) := \{x\in\dom(w)\;|\; \Gamma(x) = w(x)\}.\]
\end{definition}

\begin{lemma}[Jensen]\label{jensen}
Let \(D\subseteq\mathbb{R}^n\) be open and bounded and let \(w\colon\overline{D}\to\mathbb{R}\) be semiconcave.
If the minimum of \(w\) is \emph{not} on the boundary, i.e.
\[\min_{\overline{D}} w < \min_{\partial D} w,\]
then the contact set \(\mathcal{A}(w)\) has positive measure.
\end{lemma}

See Appendix in \cite{MR2084272}.

\begin{theorem}[Comparison Principle I]\label{compp}
Let \(F\colon S(n)\to\mathbb{R}\) be degenerate elliptic and sublinear and assume that \eqref{nondegcond} holds. Assume that \(u\) and \(v\) are viscosity super- and subsolutions to the PDE
\[F(\mathcal{H}w) = 0\qquad\text{in \(\Omega\)},\]
respectively. Let \(D\subset\subset\Omega\). Then
\[v|_{\partial D}\leq u|_{\partial D}\qquad\Rightarrow\qquad v \leq u\quad\text{in \(D\)}.\]
\end{theorem}

\begin{proof}
Let \(v\) be a subsolution and let \(u\) be a supersolution. Fix \(D\subset\subset\Omega\) and assume \((u-v)|_{\partial D}\geq 0\). For the sake of contradiction, we shall assume that \(u-v\) is negative at some point in \(D\). Since the function \(u-v\) is l.s.c. and non-negative on the boundary, it obtain its minimum in \(D\):
\[m := \min_{\overline{D}} (u-v) = u(\hx) - v(\hx) < 0,\qquad \hx\in D.\]
The function \(w := u-v-m\) in \(\overline{D}\) then satisfies
\[w\geq 0,\qquad w(\hx) = 0,\qquad w|_{\partial D} \geq - m > 0.\]

Let
\[m_u := \min_{\overline{D}} u,\qquad M_v := \max_{\overline{D}}v.\]
Observe that
\[M_v - m_u \geq -m > 0.\]
The cut-off functions
\[\hat{u}(x) := \min\{u(x),M_v\},\qquad \hat{v}(x) := \max\{v(x),m_u\}\]
are again super- and subsolutions, respectively (Lemma \ref{minlem}), but now they are bounded in \(\overline{D}\):
\[m_u \leq \hat{u},\hat{v} \leq M_v.\]
Since \(u\) is less than \(v\) at \(\hx\), we must have \(u(\hx)<M_v\) and \(v(\hx)>m_u\). Thus
\[\hat{u}(\hx) - \hat{v}(\hx) = u(\hx) - v(\hx) = m.\]
Going through all the cases, we find that
\[\hat{u} - \hat{v} = \min\{u,M_v\} - \max\{v,m_u\} =
\begin{cases}
u - v \geq m\\
u - m_u \geq 0\\
M_v - v \geq 0\\
M_v - m_u \geq 0
\end{cases}
\]
and is non-negative on the boundary. Like \(w\),
the function \(\hat{w} := \hat{u} - \hat{v} - m\) then satisfies
\[\hat{w}\geq 0,\qquad \hat{w}(\hx) = 0,\qquad \hat{w}|_{\partial D} \geq - m > 0.\]

For \(\epsilon>0\), the inf-convolution
\[\hat{u}_\epsilon(x) := \inf_{y\in D}\left\{\hat{u}(y) + \frac{|x-y|^2}{2\epsilon}\right\}\]
and the sup-convolution
\[\hat{v}^\epsilon(x) := \sup_{y\in D}\left\{\hat{v}(y) - \frac{|x-y|^2}{2\epsilon}\right\}\]
are super- and subsolutions in the slightly smaller domain
\[D_\epsilon := \left\{x\in D\;|\; \dist(x,\partial D) > \sqrt{2\epsilon(M_v-m_u)}\right\}.\]
They satisfy
\[\hat{u}_\epsilon \nearrow \hat{u}\qquad\text{and}\qquad \hat{v}^\epsilon \searrow \hat{v}\]
pointwise as \(\epsilon\to 0\). Moreover, they are semiconcave and semiconvex, respectively, and thus twice differentiable almost everywhere in \(D\). This is Proposition \ref{infthm}, \ref{supconvthm} and Theorem \ref{alexandrov}.

By compactness of the boundary and semicontinuity, we can fix \(\epsilon>0\) so small so that
\[(\hat{u}_\epsilon - \hat{v}^\epsilon)|_{\partial D_\epsilon} \geq \frac{m}{2}.\]
Define \(\hat{w}_\epsilon := \hat{u}_\epsilon - \hat{v}^\epsilon - m\).\footnote{\(\hat{w}_\epsilon\) is \emph{not} the inf-convolution of \(\hat{w}\).} in \(D_\epsilon\). Then
\[\hat{w}_\epsilon(\hx) \leq \hat{w}(\hx) = 0 \qquad\text{and}\qquad \hat{w}_\epsilon|_{\partial D_\epsilon} \geq -\frac{m}{2} > 0.\]

Let \(\psi\) be the paraboloid
\[\psi(x) := \frac{\alpha}{2}|x-\hx|^2\]
where
\[\alpha := -\frac{m}{2d^2} > 0,\qquad d := \operatorname{diam}(D).\]
We see that \(\psi|_{\partial D_\epsilon} \leq -\frac{m}{4}\).
Finally, define the function \(f\colon D_\epsilon\to\mathbb{R}\) as
\[f(x) := \hat{w}_\epsilon(x) - \psi(x).\]
Again, \(f\) is semiconcave and
\[f(\hx) \leq 0 \qquad\text{and}\qquad f|_{\partial D_\epsilon} \geq -\frac{m}{2} + \frac{m}{4} = -\frac{m}{4} > 0.\]

By Jensen's lemma (Lemma \ref{jensen}) the contact set \(\mathcal{A}(f)\) has positive measure. It must therefore be a point \(x_1\in\mathcal{A}\) where the Alexandrov Hessian matrices
\[X := \mathcal{H}\hat{u}_\epsilon(x_1),\qquad Y := \mathcal{H}\hat{v}^\epsilon(x_1)\]
both exist. Also, \(f\) is touched below by a plane at \(x_1\). This means that
\begin{equation}
\begin{aligned}
0 &\leq \mathcal{H}f(x_1)\\
  &= \mathcal{H}\hat{w}_\epsilon(x_1) - \mathcal{H}\psi\\
	&= X - Y - \alpha I
\end{aligned}
\label{XYI}
\end{equation}
and the contradiction follows:
\begin{align*}
0 &\leq F(Y) &&\text{since \(\hat{v}^\epsilon\) is a subsolution,}\\
  &\leq F(X-\alpha I)&&\text{by \eqref{XYI} and ellipticity,}\\
	&\leq F(X) + \alpha F(-I)&&\text{by sublinearity,}\\
	&\leq \alpha F(-I)&&\text{since \(\hat{u}_\epsilon\) is a supersolution,}\\
	&< 0&&\text{by \eqref{Fneg}}.
\end{align*}
\end{proof}

\begin{theorem}[Comparison Principle II]\label{comppII}
Let \(F\colon S(n)\to\mathbb{R}\) be degenerate elliptic and sublinear and assume that the non-degeneracy condition \eqref{nondegcond} holds. Consider the equation
\begin{equation}
F(\mathcal{H}u) = 0\qquad\text{in \(\Omega\)}.
\label{pdeII}
\end{equation}
Let \(u\colon\Omega\to (-\infty,\infty]\) be l.s.c and finite on a dense subset. If
\begin{equation}
D\subset\subset\Omega,\quad \text{\(h\in C(\overline{D})\) is a solution to \eqref{pdeII}},\footnote{The proof shows that the Theorem can be sharpened. We only need to consider balls \(D\) and quadratics \(h\).}\quad h|_{\partial D}\leq u|_{\partial D}
\label{IIcond}
\end{equation}
implies \(h\leq u\) in \(D\). Then
\[\text{\(u\) is a supersolution to \eqref{pdeII} in \(\Omega\).}\]
\end{theorem}

\begin{proof}
We shall assume that \(u\) is \emph{not} a supersolution. The Theorem is proved if we can find a \(D\) and a \(h\) as in \eqref{IIcond} so that \(h > u\) at some point in \(D\).

As \(u\) is not a supersolution, there is a test functions \(\phi\) touching \(u\) from below at some \(x_0\in\Omega\),
\[\phi(x_0) = u(x_0),\qquad \phi(x) \leq u(x)\quad\text{in, say, \(B_r := B(x_0,r)\subset\subset\Omega\),}\]
with
\[F(\mathcal{H}\phi(x_0))>0.\]
We may assume that \(\phi\) is a quadratic (Lemma \ref{quadstrict}). That is, \(\mathcal{H}\phi\) is constant.

Define the quadratic \(h_\alpha\) with parameter \(\alpha\geq 0\) as
\[h_\alpha(x) := \phi(x) - \frac{\alpha}{2}|x-x_0|^2.\]
Now, \(\mathcal{H}h_\alpha = \mathcal{H}\phi - \alpha I\) and \(F(\mathcal{H}h_\alpha)>0\) if \(\alpha = 0\). On the other hand,
\[F(\mathcal{H}h_\alpha) \leq F(\mathcal{H}\phi) + \alpha F(-I) < 0\]
for large \(\alpha\) by sublinearity and \eqref{Fneg}. Therefore, by continuity, there must be an \(\alpha > 0\) so that \(F(\mathcal{H}h_\alpha) = 0\). It follows that
\[h_\alpha(x_0) = \phi(x_0) = u(x_0),\qquad h_\alpha(x) < \phi(x) \leq u(x)\quad\text{at \(\partial B_r\)}\]
and it is possible to shift \(h_\alpha\) upwards by a small amount so that \(h_\alpha(x_0) > u(x_0)\) even though \eqref{IIcond} still holds.
\end{proof}

\section{Superposition principles}\label{sec:suppos}
By subadditivity, the result below is obviously true for \(C^2\)-functions. But the proof in the viscosity setting seems to require the same machinery as was used in the proof of the comparison principle. There is, however, a shortcut when the equation is uniformly elliptic. See Corollary \ref{unifsuppos}.

\begin{theorem}\label{sublin_superpos}
Let
\begin{equation}
F(\mathcal{H}w) = 0
\label{pdesublin}
\end{equation}
be a degenerate elliptic and sublinear equation.
If \(u\) and \(v\) are supersolutions
in \(\Omega\), then the sum \[u+v\] is also a supersolution to \eqref{pdesublin} in \(\Omega\).
\end{theorem}
That sub- and supersolutions also are preserved by a non-negative scaling of the function, can easily be verified by using the positive homogenicity of the operator.

One may observe that the theorem is analogous to Theorem 5.8 in \cite{MR1351007}. However, \cite{MR1351007} only deals with \emph{uniformly} elliptic equations and \emph{continuous} sub- and supersolutions.

\begin{proof}

Assume first that \(u\) and \(v\) are locally bounded.

It is clear that it suffices to show that every point in \(\Omega\) has a neighbourhood where \(u+v\) is a supersolution.
To this end, let \(x_0\in\Omega\).
By Proposition \ref{infthm} part 4, we may choose domains \(D,D'\) and an \(\epsilon'>0\) so that
\(x_0\in D\subseteq D'\subseteq \Omega\) and so that (the restricted) infimal convolutions \(u_\epsilon\), \(v_\epsilon\) are supersolutions in \(D\) for every \(0<\epsilon<\epsilon'\). Fix \(\epsilon<\epsilon'\). We show that the sum \(u_\epsilon + v_\epsilon\) is a supersolution in \(D\).

Let \(\phi\) be a test function touching \(u_\epsilon + v_\epsilon\) from below at some point \(x_1\in D\). We may assume that \(\phi\) is a quadratic and that the touching is strict (Lemma \ref{quadstrict}). Choose \(R>0\) so small so that \(B_R := B_R(x_1) \subset\subset D\) and so that
\[\phi < u_\epsilon + v_\epsilon\qquad\text{in \(\overline{B}_R\setminus\{x_1\}\)}.\]
By part 1 of Proposition \ref{infthm}, the function
\[w := u_\epsilon + v_\epsilon - \phi\]
is semiconcave in \(B_R\). Since it has a strict minimum at \(x_1\), it follows from Lemma \ref{jensen} that the contact set
\[\mathcal{A}(w) = \{x\in B_R\;|\; \Gamma(x) = w(x)\}\]
with its convex envelope \(\Gamma\) has positive measure.

Also, since \(u_\epsilon\) and \(v_\epsilon\) are semiconcave, they are twice differentiable a.e. in \(B_R\) (Theorem \ref{alexandrov}). In particular, there must be a point \(x_2\in\mathcal{A}(w)\subseteq B_R\) where both the Alexandrov Hessians \(X:=\Hu_\epsilon(x_2)\) and \(Y:=\mathcal{H}v_\epsilon(x_2)\) exist.
Therefore,
\[F(X) \leq 0\qquad\text{and}\qquad F(Y)\leq 0\]
by Lemma \ref{twicedifflem}.
Furthermore, in \(\mathcal{A}(w)\), \(w\) is touched below by hyperplanes. This means that \(X + Y - \mathcal{H}\phi =: \mathcal{H}w(x_2)\geq 0\) and it follows that
\[\mathcal{H}\phi \leq X + Y.\]
Thus
\begin{align*}
F(\mathcal{H}\phi) &\leq F(X + Y)\\
                   &\leq F(X) + F(Y)\\
									&\leq 0
\end{align*}
by degenerate ellipticity and sublinearity, and \(u_\epsilon + v_\epsilon\) is a supersolution in \(D\).

We are now able to construct the increasing sequence of supersolutions \((u_{\epsilon/k} + v_{\epsilon/k})\) in \(D\). It converges pointwise to \(u+v\) (Proposition \ref{infthm} part 2) and it follows that \(u+v\) is a supersolution in \(D\) by Lemma \ref{inclemma}.

The initial assumption on local boundednes, can be removed in the same manner by applying Lemma \ref{inclemma} to the increasing sequence
\[(w_k)\qquad\text{where}\qquad w_k := \min\{u(x),k\} + \min\{v(x),k\}.\]
The sequence converges pointwise to \(u+v\) and we have shown that each \(w_k\) is a supersolution in \(\Omega\) since it is locally bounded by semi-continuity, and since each of the two terms in \(w_k\) are supersolutions by Lemma \ref{minlem}.
\end{proof}

An equation \(F(\mathcal{H}w) = 0\) is uniformly elliptic with ellipticity constants \(0<\lambda\leq\Lambda\) if and only if
\begin{equation}
F(X+Y) \leq F(X) + P_{\lambda,\Lambda}(Y)
\label{altunidef}
\end{equation}
for all \(X,Y\in S(n)\). See Proposition \ref{altelldef} in Appendix \ref{ch:visc_ell}.

Here, \(P_{\lambda,\Lambda}\) denotes the Pucci operator. If \(F(X)\leq 0\) and \(F(Y)\geq 0\), then \eqref{altunidef} implies
\begin{equation}
P_{\lambda,\Lambda}(Y-X) \geq F(Y) - F(X)\geq 0.
\label{subineq}
\end{equation}

Our second superposition principle extends the inequality \eqref{subineq} to the viscosity sense. One may also replace the Pucci operator with an arbitrary sublinear and degenerate elliptic operator \(G\). It is a generalization of Theorem 5.3 in \cite{MR1351007}.

\begin{theorem}\label{subsupsub}
Let \(G\colon S(n)\to\mathbb{R}\) be sublinear and degenerate elliptic, and assume that an operator \(\D u = F(\mathcal{H}u)\)
 satisfies the ellipticity condition
\begin{equation}
F(X+Y) \leq F(X) + G(Y)
\label{altunidefII}
\end{equation}
for all \(X,Y\in S(n)\).

Then \(F\) is degenerate elliptic. Moreover,
if \(u\) is a supersolution and \(v\) is a subsolution to \(F(\mathcal{H}w) = 0\) in \(\Omega\), then \(v-u\) is a subsolution to \(G(\mathcal{H}w) = 0\) in \(\Omega\).
\end{theorem}

\begin{remark}
Note the special case obtained by setting \(F\) equal to \(G\).
\end{remark}
\begin{remark}
If \(G\) is uniformly elliptic, then \(G\leq P_{\lambda,\Lambda}\) for some \(0<\lambda\leq \Lambda\) and \(F\) is uniformly elliptic by \eqref{altunidef} as well.
\end{remark}

\begin{proof}
Observe first that if \(X\leq Y\) then
\[F(X) = F(Y+X-Y) \leq F(Y) + G(X-Y) \leq F(Y)\]
by \eqref{altunidefII} and since \(X-Y\leq 0\) and \(G(X-Y)\leq G(0)=0\). Thus \(F\) is degenerate elliptic.

The rest of the proof is almost identical to the proof of Theorem \ref{sublin_superpos} and may be omitted. 
The differences lies in that we use the inequality
\[G(Y-X) \geq F(Y) - F(X)\geq 0\]
instead of the sublinearity, and that we also employ the corresponding results for subsolutions in the various lemmas.

\end{proof}

\chapter{Rotational invariance}\label{ch:rot}

In this chapter, we shall be concerned with operators that are invariant under isometric mappings of the domain:

\begin{definition}\label{rotinvdef}
An operator \(\D u = F(u,\nabla u,\mathcal{H}u)\) is \textbf{rotationally invariant} if
\[\D[u\circ T] = (\D u)\circ T\]
for every isometry \(T\colon\mathbb{R}^n\to \mathbb{R}^n\).
\end{definition}
By the equivalent definition below, one may verify that the \(p\)-Laplacian, as well as the dominative \(p\)-Laplacian, and the Pucci operator are all rotationally invariant.

\begin{proposition}\label{Frot}
An operator \(\D u = F(u,\nabla u,\mathcal{H}u)\) is rotationally invariant if and only if the function
\[F\colon \mathbb{R}\times \mathbb{R}^n\times S(n)\to \mathbb{R}\]
satisfies
\begin{equation}
F(s,q,X) = F(s,Qq,QXQ^T)
\label{rotinvF}
\end{equation}
for all \((s,q,X)\in \mathbb{R}\times \mathbb{R}^n\times S(n)\) and all \(Q\in O(n)\).
\end{proposition}
Here, \(O(n)\) denotes the set of \(n\times n\) orthogonal matrices.

\begin{proof}
Let \(F\) be rotationally invariant and let \((s,q,X)\in \mathbb{R}\times \mathbb{R}^n\times S(n)\) and \(Q\in O(n)\).
Define \(u(x) := s + q^Tx + \frac{1}{2}x^TXx\). Then, at \(x=0\), \(u = s, \nabla u^T = q\) and \(\mathcal{H}u = X\).
If we set \(v(x) := u(Q^Tx)\), then \(v = u\circ T\) for the isometry \(T(x) = Q^Tx\) and, at \(x=0\), \(T(0) = 0\) and \(v = s, \nabla v^T = Qq\) and \(\mathcal{H}v = QXQ^T\). It follows that
\begin{align*}
F(s,q,X) &= (\D u)(T(0))\\
				 &= \D[u\circ T](0),\qquad\text{by assumption,}\\
				 &= \D v(0)\\
				 &= F(s,Qq,Q^TXQ).
\end{align*}

Now assume that \eqref{rotinvF} holds and let \(T\colon\mathbb{R}^n\to \mathbb{R}^n\) be an isometry. It is in the form \(T(x) = Q^T(x-x_0)\) for some \(Q\in O(n)\) and some \(x_0\in\mathbb{R}^n\).

Let \(u\) be \(C^2\) and define \(v := u\circ T\). Then \(v(x) = u(T(x))\), \(\nabla v(x) = \nabla u(T(x))Q^T\) and
\(\mathcal{H}v(x) = Q\mathcal{H}u(T(x))Q^T\). Thus
\begin{align*}
\D[u\circ T](x) &= \D v(x)\\
                &= F(v(x), \nabla v^T(x),\mathcal{H}v(x))\\
								&= F(u(T(x)), Q\nabla u(T(x)), Q\mathcal{H}u(T(x))Q^T)\\
								&= F(u(T(x)), \nabla u(T(x)), \mathcal{H}u(T(x))),\qquad\text{by assumption,}\\
								&= (\D u)(T(x)).
\end{align*}
\end{proof}

We now return to sublinear operators. First, we show that a symmetry in \(F\) is reflected in a corresponding symmetry in the associated convex body:

If \(T\) is a linear function from \(S(n)\) to \(S(n)\), its \emph{adjoint} \(T^*\) is defined by the formula \(\langle TY,X\rangle = \langle Y, T^*X\rangle\). Of course, \(T\) can be identified with an \(N\times N\) matrix \((N:=\frac{n(n+1)}{2})\) and \(T^*\) is then the transposed.

Denote the image set of \(T\) on \(\mathcal{K}\) by \(T(\mathcal{K}) := \{TY\;|\; Y\in\mathcal{K}\}\).

\begin{lemma}\label{Kinv}
Let \(F(X) = \max_{Y\in\mathcal{K}}\langle Y,X\rangle\) be a sublinear operator and let \(T\) be a linear and invertible function from \(S(n)\) to \(S(n)\). Then
\[F\circ T^* = F\qquad\iff\qquad T(\mathcal{K}) = \mathcal{K}.\]
\end{lemma}

\begin{proof}
\begin{align*}
F(T^*X) &= \max_{Y\in\mathcal{K}}\langle Y,T^*X\rangle\\
        &= \max_{Y\in\mathcal{K}}\langle TY,X\rangle\\
				&= \max_{T^{-1}Z\in\mathcal{K}}\langle Z,X\rangle\\
				&= \max_{Z\in T(\mathcal{K})}\langle Z,X\rangle
\end{align*}
which equals the function \(F(X)\) if and only if \(T(\mathcal{K}) = \mathcal{K}\). This follows from the fact that \(T(\mathcal{K})\) is convex and the uniqueness of the convex body.
\end{proof}

\begin{proposition}\label{rotsymequiv}
A sublinear operator is rotationally invariant if and only if the associated convex body \(\mathcal{K}\) satisfies
\begin{equation}
\mathcal{K} = \left\{Q^TYQ\;|\; Y\in\mathcal{K}\right\}\qquad\text{for all \(Q\in O(n)\).}
\label{QisY2}
\end{equation}
\end{proposition}

\begin{proof}
For \(Q\in O(n)\) define \(T_Q : S(n)\to S(n)\) by 
\[T_Q Y := Q^TYQ.\]
Then \(T_Q\) is a linear isometry with adjoint (and inverse) \(T_Q^*X = QXQ^T\). From Proposition \ref{Frot}, \(F\) is rotationally invariant if and only if \(F\circ T_Q^* = F\;\forall Q\in O(n)\) which again is equivalent to \eqref{QisY2} by Lemma \ref{Kinv}.
\end{proof}

\section{Symmetric bodies}

Let \(\oll\colon S(n)\to \mathbb{R}^n\) and \(\Lambda\colon S(n)\to S(n)\) be the functions
\[\oll(Y) := (\lambda_1(Y),\dots,\lambda_n(Y))^T,\qquad \Lambda(Y) := \diag\oll(Y),\]
where \(\lambda_1(Y)\leq\cdots\leq\lambda_n(Y)\) are the eigenvalues of \(Y\) labeled in increasing order.
Notice that \(\Lambda(Y)\) is a diagonalization of \(Y\). That is,
\[\Lambda(Y) = Q^TYQ\]
for some \(Q\in O(n)\).

Since
\[F(X) = F(\Lambda(X)),\]
a rotationally invariant operator only depends on the eigenvalues of the argument.
\(F\) is essentially a function in \(\mathbb{R}^n\), and it is reasonable to expect that the dimension of the associated convex body also can be reduced from \(\frac{(n+1)n}{2}\) to \(n\) without any loss of information.
We shall see that the mapping
\begin{equation}
\Phi(\mathcal{E}) := \left\{P\oll(Y)\;|\; P\in\mathcal{P}(n),\, Y\in\mathcal{E}\right\},
\label{Phidef}
\end{equation}
from subsets of \(S(n)\) to subsets of \(\mathbb{R}^n\) provides this identification.
Here \(\mathcal{P}(n)\subseteq O(n)\) is the set of permutation matrices.


\begin{definition}\label{symbodies}
A subset \(\mathcal{E}\subseteq S(n)\) is \textbf{symmetric} if
\[\mathcal{E} = \left\{Q^TYQ\;|\; Y\in\mathcal{E}\right\}\qquad\forall Q\in O(n).\]
A subset \(\varepsilon\subseteq \mathbb{R}^n\) is \textbf{symmetric} if
\[\varepsilon = \left\{Py\;|\; y\in\varepsilon\right\}\qquad\forall P\in \mathcal{P}(n).\]
We let \(\mathcal{E}(n)\) and \(\varepsilon(n)\) denote the set of symmetric subsets of \(S(n)\) and \(\mathbb{R}^n\), respectively.
\end{definition}

We introduce another diagonal operator as well. The function \(\diag\colon S(n)\to\mathbb{R}^n\) gives the vector consisting of the diagonal elements of \(Y\):
\[\diag Y := \sum_{i= 1}^n Y_{ii}\cdot \bfe_i,\qquad Y_{ii} := \bfe_i^TY\bfe_i.\]
If \(x\in\mathbb{R}^n\) and \(Y\in S(n)\), notice that \(\diag\diag x = x\), and that
\[\tr(Y\diag x) = \tr\left(Y\sum_{i=1}^nx_i\bfe_i\bfe_i^T\right) = \sum_{i=1}^nx_i\tr\left(Y\bfe_i\bfe_i^T\right) = \sum_{i=1}^nx_i \bfe_i^TY\bfe_i.\]
That is,
\[\langle Y,\diag x\rangle_{S(n)} = \langle\diag Y,x\rangle_{\mathbb{R}^n}.\]

If \(F\colon S(n)\to\mathbb{R}\) is sublinear, we can define a function \(f\colon\mathbb{R}^n\to\mathbb{R}\) as
\begin{equation}
f(x) := F(\diag x).
\label{fFdef}
\end{equation}
It is easy to check that \(f\) inherits the sublinearity from \(F\) and \(f\) is therefore in the form
\begin{equation}
f(x) = \max_{y\in\kappa}y^Tx
\label{fdef}
\end{equation}
for some unique convex body \(\kappa\subseteq\mathbb{R}^n\) by Theorem \ref{repsupthm}. In fact, as
\begin{align*}
f(x) &= F(\diag x)\\
     &= \max_{Y\in\mathcal{K}}\langle Y,\diag x\rangle\\
		 &= \max_{Y\in\mathcal{K}}(\diag Y)^Tx,
\end{align*}
we get
\begin{equation}
\kappa = \left\{\diag Y\;|\; Y\in\mathcal{K}\right\}
\label{kappafbody}
\end{equation}
since the right-hand side is compact, and convex by the linearity of \(\diag\).

A symmetry of \(\mathcal{K}\) is equivalent to \(F\) being rotationally invariant by Proposition \ref{rotsymequiv}. And in those cases, as we might suspect, we actually have \(\kappa = \Phi(\mathcal{K})\). The proof is non-trivial and requires the introduction of some new tools.

Let \(\mathcal{K}(n)\) and \(\kappa(n)\) denote the sets of convex bodies in \(S(n)\) and \(\mathbb{R}^n\), respectively.

\begin{proposition}\label{Phibijective}
The mapping
\[\Phi(\mathcal{E}) := \left\{P\oll(Y)\;|\; P\in\mathcal{P}(n),\, Y\in\mathcal{E}\right\}\]
is a one-to-one correspondence between the symmetric subsets of \(S(n)\) and \(\mathbb{R}^n\). Its inverse is
\begin{equation}
\Phi^{-1}(\varepsilon) = \left\{Q(\diag y)Q^T\;|\; Q\in O(n),\, y\in\varepsilon\right\}.
\label{Phiinvdef}
\end{equation}
Moreover, the restriction of \(\Phi\) to \(\mathcal{E}(n)\cap\mathcal{K}(n)\) is given by
\[\Phi(\mathcal{K}) = \left\{\diag Y\;|\;Y\in\mathcal{K}\right\}\]
and is a one-to-one correspondence between the symmetric convex bodies of \(S(n)\) and \(\mathbb{R}^n\)
\end{proposition}
Notice that \(\Phi\) and \(\Phi^{-1}\) clearly maps sets into symmetric sets.

The main properties of \(\Phi\) are gathered in Theorem \ref{Phithm} below. Recall the \emph{Minkowski} sum and scaling of subsets.

\begin{theorem}\label{Phithm}
Let \(\mathcal{E},\mathcal{E}'\in\mathcal{E}(n)\), \(\mathcal{K},\mathcal{K}'\in\mathcal{E}(n)\cap\mathcal{K}(n)\), and let \(\alpha,\beta\in\mathbb{R}\). Then
\begin{enumerate}
	\item \label{partextsym}\(\ext\mathcal{K}\in\mathcal{E}(n)\).
	\item \label{partextext}\(\ext\Phi(\mathcal{K}) = \Phi(\ext\mathcal{K})\).
	\item \label{suminvsym}\[\alpha\mathcal{E} + \beta\mathcal{E}'\in\mathcal{E}(n)\qquad\text{and}\qquad \alpha\mathcal{K} + \beta\mathcal{K}'\in\mathcal{E}(n)\cap\mathcal{K}(n).\]
	\item \label{suminvsymII}\[\Phi\big(\alpha\{I\}+\beta\mathcal{E}\big) = \alpha\{\mathbbm{1}\} + \beta\Phi(\mathcal{E})\quad\text{and}\quad \Phi\left(\alpha\mathcal{K}+\beta\mathcal{K}'\right) = \alpha\Phi(\mathcal{K}) + \beta\Phi(\mathcal{K}').\]
	\item \label{part_subset}\[\mathcal{E}'\subseteq\mathcal{E}\qquad\Rightarrow\qquad \Phi(\mathcal{E}')\subseteq\Phi(\mathcal{E}).\]
\end{enumerate}
Furthermore, every corresponding claim for \(\varepsilon,\varepsilon'\in\varepsilon(n)\), \(\kappa,\kappa'\in\varepsilon(n)\cap\kappa(n)\) and \(\Phi^{-1}\) is also true.
\end{theorem}
Here, \(\mathbbm{1}\) is the vector \((1,\dots,1)^T\) in \(\mathbb{R}^n\).

The various claims of Proposition \ref{Phibijective} and Theorem \ref{Phithm} shall be proved in the order of increasing difficulty. We start with part \ref{suminvsym} and part \ref{suminvsymII} of the theorem. No proof of part \ref{part_subset} is needed as it follows directly from the definition.
Also, once the proposition is established, the last assertion of the theorem is achieved by a straightforward use of the bijectivity of \(\Phi\). Except, perhaps, the corresponding part 3 which proof is most easily done by a direct calculation.

\begin{proof}[Proof of part \ref{suminvsym} of Theorem \ref{Phithm}]
Let \(\mathcal{E}, \mathcal{E}'\in\mathcal{E}(n)\) and let \(\alpha,\beta\in\mathbb{R}\). Then for any \(Q\in O(n)\),
\begin{align*}
\left\{Q^TYQ\;|\; Y\in \alpha\mathcal{E}+\beta\mathcal{E}'\right\}
	&= \left\{Q^T(\alpha E + \beta E')Q\;|\; E\in \mathcal{E}, E'\in\mathcal{E}'\right\}\\
	&= \left\{\alpha Q^TEQ + \beta Q^TE'Q\;|\; E\in \mathcal{E}, E'\in\mathcal{E}'\right\}\\
	&= \left\{\alpha E + \beta E'\;|\; E\in \mathcal{E}, E'\in\mathcal{E}'\right\}\\
	&= \alpha\mathcal{E}+\beta\mathcal{E}'
\end{align*}
and \(\alpha\mathcal{E}+\beta\mathcal{E}'\in\mathcal{E}(n)\). The proof regarding symmetric convex bodies follows from this and Proposition \ref{mincsum}.
\end{proof}

\begin{proof}[Proof of part \ref{suminvsymII} of Theorem \ref{Phithm}]
Let \(\mathcal{E}\in\mathcal{E}(n)\) and let \(\alpha,\beta\in\mathbb{R}\). Clearly \(\{I\}\in\mathcal{E}(n)\), and part \ref{suminvsym} shows that
\(\mathcal{E}' := \alpha\{I\} + \beta\mathcal{E}\in\mathcal{E}(n)\). Since \(\lambda_i(\alpha I+Y) = \alpha + \lambda_i(Y)\), \(P\oll(I) = P\mathbbm{1} = \mathbbm{1}\), and \(\oll(\beta E) = \beta P'\oll(E)\) where \(P'\in\mathcal{P}(n)\) is \(I\) if \(\beta\geq 0\) and reverses the order if \(\beta<0\), we get
\begin{align*}
\Phi(\mathcal{E}') &= \left\{P\oll(Y)\;|\; P\in\mathcal{P}(n),\, Y\in \mathcal{E}'\right\}\\
                   &= \left\{P\oll\left(\alpha I + \beta E\right)\;|\; P\in\mathcal{P}(n),\, E\in \mathcal{E}\right\}\\
									 &= \left\{P\left(\alpha\oll(I) + \oll(\beta E)\right)\;|\; P\in\mathcal{P}(n),\, E\in \mathcal{E}\right\}\\
									 &= \left\{\alpha \mathbbm{1} + \beta P\oll(E)\;|\; P\in\mathcal{P}(n),\, E\in \mathcal{E}\right\}\\
									 &= \alpha\{\mathbbm{1}\} + \beta\Phi(\mathcal{E}).
\end{align*}

Now let \(\mathcal{K}\) and \(\mathcal{K}'\) be symmetric convex bodies. Then \(\alpha\mathcal{K} + \beta\mathcal{K}'\) is again a symmetric convex body by part \ref{suminvsym}. From Proposition \ref{Phibijective} (which is not yet proved, but its proof does not depend on this result) we immediately get
\begin{align*}
\Phi\left(\alpha\mathcal{K}+\beta\mathcal{K}'\right)
	&= \left\{\diag Y\;|\; Y\in \alpha\mathcal{K}+\beta\mathcal{K}'\right\}\\
	&= \left\{\diag \left(\alpha Y+\beta Y'\right)\;|\; Y\in \mathcal{K}, Y'\in\mathcal{K}'\right\}\\
	&= \left\{ \alpha\diag Y+\beta \diag Y'\;|\; Y\in \mathcal{K}, Y'\in\mathcal{K}'\right\}\\
	&= \alpha\Phi(\mathcal{K})+\beta\Phi(\mathcal{K}')
\end{align*}
\end{proof}

\begin{lemma}\label{1_1lemma}
Let \(\mathcal{E},\mathcal{E}'\in\mathcal{E}(n)\) and let \(P\in\mathcal{P}(n)\). Then
\[Y\in\mathcal{E}\setminus\mathcal{E}'\qquad\iff\qquad P\oll(Y)\in\Phi(\mathcal{E})\setminus\Phi(\mathcal{E}').\]
\end{lemma}

\begin{proof}[Proof of Lemma]
If \(Y\in\mathcal{E}\), then \(P\oll(Y) \in\Phi(\mathcal{E})\) by definition.

If \(P\oll(Y) \in\Phi(\mathcal{E})\), then \(\oll(Y) \in\Phi(\mathcal{E})\) by symmetry and \(\oll(Y) = \oll(Z)\) for some
\(Z\in \mathcal{E}\) (no permutation since the components of \(\oll(Z)\) already are in increasing order). This means that
\[Q^TYQ = \Lambda(Y) = \Lambda(Z) = Q'^TZQ'\]
for some diagonalizing orthogonal matrices \(Q,Q'\) and \(Y = Q''^TZQ''\in\mathcal{E}\).
\end{proof}

For a convex body \(\mathcal{K}\), recall the definition \eqref{Kcircdef} of \(\mathcal{K}^\circ\).
Since \(\ext\mathcal{K} = \mathcal{K}\setminus\mathcal{K}^\circ\), part \ref{partextsym} of Theorem \ref{Phithm} is proved by the next Lemma.


\begin{lemma}\label{syminv}
Let \(\mathcal{K}, \mathcal{E},\mathcal{E}'\in\mathcal{E}(n)\) where \(\mathcal{K}\) is a convex body. Then the sets
\[\mathcal{K}^\circ,\quad\mathcal{E}\setminus\mathcal{E}',\quad \con\mathcal{E}\]
are again symmetric in \(S(n)\). Moreover, \(\Phi(\mathcal{E}\setminus\mathcal{E}') = \Phi(\mathcal{E})\setminus\Phi(\mathcal{E}')\).
\end{lemma}

\begin{proof}
Let \(Q\in O(n)\). Then

\begin{align*}
\left\{Q^TXQ\;|\; X\in\mathcal{K}^\circ\right\}
	&= \left\{Q^T(tY + (1-t)Z)Q\;|\; t\in(0,1),\,Y\neq Z\in\mathcal{K}\right\}\\
	&= \left\{tQ^TYQ + (1-t)Q^TZQ\;|\; t\in(0,1),\,Y\neq Z\in\mathcal{K}\right\}\\
	&= \left\{tY + (1-t)Z\;|\; t\in(0,1),\,Y\neq Z\in\mathcal{K}\right\}\\
	&= \mathcal{K}^\circ
\end{align*}
and thus \(\mathcal{K}^\circ\in\mathcal{E}(n)\). Next, 

\begin{align*}
\left\{Q^TYQ\;|\;Y\in\mathcal{E}\setminus\mathcal{E}'\right\}
	&= \left\{Z\;|\;Q^TZQ\in\mathcal{E}\setminus\mathcal{E}'\right\}\\
	&= \left\{Z\;|\;Z\in\mathcal{E}\setminus\mathcal{E}'\right\}\\
	&= \mathcal{E}\setminus\mathcal{E}'
\end{align*}
and \(\mathcal{E}\setminus\mathcal{E}'\in\mathcal{E}(n)\). Also,

\begin{align*}
\left\{Q^TEQ\;|\; E\in\con\mathcal{E}\right\}
	&= \left\{Q^T\sum_i\alpha_i E_iQ\;|\; E_i\in\mathcal{E}\right\}\\
	&= \left\{\sum_i\alpha_iQ^TE_iQ\;|\; E_i\in\mathcal{E}\right\}\\
	&= \left\{\sum_i\alpha_iE_i\;|\; E_i\in\mathcal{E}\right\}\\
	&= \con\mathcal{E}
\end{align*}
and \(\con\mathcal{E}\in\mathcal{E}(n)\).

Finally, by Lemma \ref{1_1lemma}, we get
\begin{align*}
\Phi(\mathcal{E}\setminus\mathcal{E}')
	&= \left\{P\oll(Y)\;|\; P\in \mathcal{P}(n),\, Y\in\mathcal{E}\setminus\mathcal{E}'\right\}\\
	&= \left\{P\oll(Y)\;|\; P\oll(Y)\in\Phi(\mathcal{E})\setminus\Phi(\mathcal{E}')\right\}\\
	&= \Phi(\mathcal{E})\setminus\Phi(\mathcal{E}').
\end{align*}
\end{proof}

To show that \(\Phi\colon\mathcal{E}(n)\to\varepsilon(n)\) is a bijection, is also relatively straight forward:
As for the injectivity, let \(\mathcal{E}\) and \(\mathcal{E}'\) be two different symmetric subsets of \(S(n)\) where, say, \(\mathcal{E}\setminus\mathcal{E}'\) is nonempty. Then \(\Phi(\mathcal{E})\setminus\Phi(\mathcal{E}')\) is 
also nonempty by Lemma \ref{1_1lemma} and thus \(\Phi(\mathcal{E})\neq\Phi(\mathcal{E}')\).

Now let \(\varepsilon\) be a symmetric subset of \(\mathbb{R}^n\). Define \(\Psi\colon\varepsilon(n)\to \mathcal{E}(n)\) by \eqref{Phiinvdef} and set
\[\mathcal{E} := \Psi(\varepsilon).\]
Then
\begin{align*}
\Phi(\mathcal{E}) &= \Phi(\Psi(\varepsilon))\\
	&= \left\{P\oll(Y)\;|\; P\in\mathcal{P}(n),\, Y\in\Psi(\varepsilon)\right\}\\
	&= \left\{P\oll\left(Q(\diag y)Q^T\right)\;|\; P\in\mathcal{P}(n),\, Q\in O(n),\,y\in\varepsilon\right\}\\
	&= \left\{P\oll\left(\diag y\right)\;|\; P\in\mathcal{P}(n),\,y\in\varepsilon\right\}\\
	&= \left\{Py\;|\; P\in\mathcal{P}(n),\,y\in\varepsilon\right\}\\
	&= \varepsilon.
\end{align*}
and we have proven that \(\Phi\) is also surjective. The inverse therefore exists and is given by \(\Phi^{-1}(\varepsilon) = \Phi^{-1}\left(\Phi(\Psi(\varepsilon))\right) = \Psi(\varepsilon)\).

\begin{definition}
Let \(x,y\in\mathbb{R}^n\) and let \(x^\uparrow, y^\uparrow\in\mathbb{R}^n\) be the vectors with the same components as \(x\) and \(y\), respectively, but permuted so that they are in increasing order. We say that \(y\) \textbf{majorizes} \(x\), written \(x\prec y\), 
if
\begin{align}
\sum_{i=k}^n x^\uparrow_i &\leq \sum_{i=k}^n y^\uparrow_i\qquad \forall k = 2,\dots, n,\quad\text{and}\label{majk}\\
\sum_{i=1}^n x^\uparrow_i &= \sum_{i=1}^n y^\uparrow_i.\label{maj}
\end{align}
Notice that, by using \eqref{maj}, one can replace condition \eqref{majk} with \(\sum_{i=1}^k x^\uparrow_i \geq \sum_{i=1}^k y^\uparrow_i\) for all \(k = 1,\dots n-1\).
\end{definition}

Ordering the components maximizes the inner product (Corollary II.4 \cite{MR1477662}):
\begin{equation}
x^Ty \leq (x^\uparrow)^Ty^\uparrow
\label{ordermax}
\end{equation}

\begin{lemma}[Schur's Theorem]\label{schurthm}
Let \(X\in S(n)\). Then
\[\diag X\prec \oll(X).\]
\end{lemma}

We also have a sublinearity of the function \(\oll\) with respect to \(\prec\).
\begin{lemma}\label{maysublin}
Let \(Y,Z\in S(n)\) and let \(t\geq 0\). Then
\begin{align*}
\oll(tY) &= t\oll(Y),\\
\oll(Y+Z) &\prec \oll(Y) + \oll(Z).
\end{align*}
\end{lemma}

\begin{proof}
The positive homogenicity is immediate. The second part is a special case of Corollary III.4.2 in \cite{MR1477662}.
\end{proof}

The following observation is also needed:

\begin{lemma}\label{perpdiag}
Let \(P\in\mathcal{P}(n)\) and let \(x\in\mathbb{R}^n\). Then
\[\diag(Px) = P(\diag x)P^T.\]
\end{lemma}

\begin{proof}
A permutation matrix is in the form \(P = (\bfe_{\pi(1)}, \dots ,\bfe_{\pi(n)})^T\) for some permutation \(\pi\colon\{1,\dots,n\}\to\{1,\dots,n\}\). If \(x = (x_i)\in\mathbb{R}^n\), then \(Px = (x_{\pi(i)})\) and \(\diag Px = (x_{\pi(i)}\delta_{ij})\). But the \(ij\)-component of \(P(\diag x)P^T\) is also
\begin{align*}
\bfe_i^TP(\diag x)P^T\bfe_j &= \bfe_{\pi(i)}^T(\diag x)\bfe_{\pi(j)}\\
	&= x_{\pi(i)}\bfe_{\pi(i)}^T\bfe_{\pi(j)}\\
	&= x_{\pi(i)}\delta_{ij}.
\end{align*}

\end{proof}

\begin{definition}
A function \(f:\mathbb{R}^n\to\mathbb{R}\) is \textbf{Schur-convex} if \(x\prec y\) implies \(f(x)\leq f(y)\).
\end{definition}

We shall use the following facts. (Theorem II1.10 and Remark II.3.7 in \cite{MR1477662})
\begin{itemize}
	\item \begin{equation}
x\prec y\qquad\iff\qquad x\in \con\{Py\;|\;P\in\mathcal{P}(n)\}.
\label{maylem}
\end{equation}
	\item \begin{equation}
	\text{A convex and symmetric function is Schur-convex.}
	\label{consym}
	\end{equation}
\end{itemize}

By a \emph{symmetric} function \(f:\mathbb{R}^n\to\mathbb{R}\), we mean a function that does not depend on the order of the components. The point is that \(f\) given by \eqref{fFdef} is Schur-convex whenever \(F\) is rotationally invariant; it is convex by sublinearity, and for any permutation \(P\) we have
\[f(Px) = F(\diag Px) = F(P(\diag x)P^T) = F(\diag x) = f(x)\]
by Lemma \ref{perpdiag} and since \(P\in O(n)\).

We are now ready to complete the proof of Proposition \ref{Phibijective}. The following claim suffices.

\begin{claim}\label{Kdiagprop}
Let \(\mathcal{E}\in\mathcal{E}(n)\). Then
\[\Phi(\mathcal{E})\subseteq\left\{\diag Y\;|\;Y\in\mathcal{E}\right\}.\]
If \(\mathcal{K}\in\mathcal{E}(n)\cap\mathcal{K}(n)\), then
\[\Phi(\mathcal{K}) = \left\{\diag Y\;|\;Y\in\mathcal{K}\right\}.\]
Moreover, if \(\kappa\in\varepsilon(n)\cap\kappa(n)\), then \(\Phi^{-1}(\kappa)\in\mathcal{E}(n)\cap\mathcal{K}(n)\).
\end{claim}

\begin{proof}
Write \(\kappa := \left\{\diag Y\;|\;Y\in\mathcal{K}\right\}\). Assume that \(y\in\Phi(\mathcal{K})\). i.e., \(y = P\oll(Y)\) for some permutation \(P\) and some \(Y\in\mathcal{K}\). By symmetry, \(\diag\oll(Y) = \Lambda(Y)\in\mathcal{K}\) and also \(\diag P\oll(Y) = P(\diag\oll(Y)P^T\in\mathcal{K}\). Therefore
\[y = P\oll(Y) = \diag\diag P\oll(Y)\in\kappa,\]
and we have not used the convexity.

Now assume \(y\in\kappa\). Define \(F := F_\mathcal{K}\) to be the sublinear operator with convex body \(\mathcal{K}\) and let \(f\) be given by \(f(x) = F(\diag x)\). Then \(f(x) = \max_{y\in\kappa}y^Tx\) by \eqref{kappafbody} and,
by Theorem \ref{repsupthm}, the convex bodies \(\kappa\) and \(\mathcal{K}\) are
\begin{align}
\kappa &= \left\{z\in \mathbb{R}^n\;|\; z^Tx  \leq f(x)\quad\forall x\in \mathbb{R}^n\right\},\label{conbI}\\
\mathcal{K} &= \left\{Z\in S(n)\;|\; \langle Z,X\rangle  \leq F(X)\quad\forall X\in S(n)\right\}.\label{conbII}
\end{align}
Let \(X\in S(n)\). Then
\begin{align*}
\langle \diag y, X\rangle &= y^T\diag X\\
													&\leq f(\diag X),\qquad &&\text{by \eqref{conbI}},\\
													&\leq f(\oll(X)), &&\text{by Lemma \ref{schurthm} and since \(f\) is Schur-convex},\\
													&= F(\diag\oll(X))\\
													&= F(\Lambda(X)) = F(X)
\end{align*}
and \(\diag y\in\mathcal{K}\) by \eqref{conbII}.

Since \(y = P\oll(\diag y)\) for some permutation \(P\), it follows that \(y\in\Phi(\mathcal{K})\).
We have proved that 
\[\left\{\diag Y\;|\;Y\in\mathcal{K}\right\} =: \kappa = \Phi(\mathcal{K}) := \left\{P\oll(Y)\;|\;P\in\mathcal{P}(n),\, Y\in\mathcal{K}\right\}.\]

Finally, let \(\kappa\) be a symmetric convex body in \(\mathbb{R}^n\). We only have to prove that
\[\Phi^{-1}(\kappa) = \left\{Q(\diag y)Q^T\;|\; Q\in O(n),\, y\in\kappa\right\}\]
is convex. To this end, let \(Y,Y'\in\Phi^{-1}(\kappa)\), let \(t\in[0,1]\), and set \(Z := tY + (1-t)Y'\). Then there are orthogonal matrices \(Q,Q'\) and points \(y,y'\) in \(\kappa\) so that
\[Z = tQ(\diag y)Q^T + (1-t)Q'(\diag y)Q'^T.\]
Taking \(\oll\) on both sides, Lemma \ref{maysublin} yields
\begin{align*}
\oll(Z) &= \oll\left(tQ(\diag y)Q^T + (1-t)Q'(\diag y')Q'^T\right)\\
        &\prec t\oll\left(Q(\diag y)Q^T\right) + (1-t)\oll\left(Q'(\diag y')Q'^T\right)\\
				&= ty^\uparrow + (1-t)y'^\uparrow.
\end{align*}
By \eqref{maylem}, we may write \(\oll(Z)\) as
\[\oll(Z) = \sum_{i=1}^M\alpha_iP_i\left(ty^\uparrow + (1-t)y'^\uparrow\right) = t\sum_{i=1}^M\alpha_iP_iy^\uparrow + (1-t)\sum_{i=1}^M\alpha_iP_iy'^\uparrow\]
and \(\oll(Z)\in\kappa\) since \(\kappa\) is symmetric and convex.

Since
\[Z = Q''\Lambda(Z)Q''^T = Q''(\diag\oll(Z))Q''^T\]
for some \(Q''\in O(n)\), it follows that \(Z\in \Phi^{-1}(\kappa)\).
\end{proof}

\begin{lemma}\label{generatesing}
Let \(E\in S(n)\) and \(e\in\mathbb{R}^n\). Then the orbits
\[\mathcal{E} := \left\{Q^TEQ\;|\; Q\in O(n)\right\},\qquad \varepsilon := \left\{Pe\;|\;P\in\mathcal{P}(n)\right\},\]
are the extreme points of their convex hull.
\end{lemma}

\begin{proof} 
The following argument apply to both claims: Set \(\mathcal{K} := \con\mathcal{E}\). Then \(\ext\mathcal{K}\subseteq\mathcal{E}\) by Proposition \ref{extconvbody}. But \(\ext\mathcal{K} = \ext\con\mathcal{E}\) is symmetric by Lemma \ref{syminv} and part \ref{partextsym} of Theorem \ref{Phithm}, so \(\mathcal{E}\subseteq\ext\mathcal{K}\) since \(\mathcal{E}\) is symmetric and obviously is the smallest symmetric set containing any of its subsets.
\end{proof}

\begin{remark}\label{remarkgenerated}
This simple result enable us to find the extreme points of the dominative body \(\mathcal{K}_p\) without computation. From section \ref{sec:example_bodies} we know that \(\mathcal{K}_p\) is the convex hull of the set
\[\mathcal{E}_p := \left\{I + (p-2)\xi\xi^T\colon |\xi|=1\right\},\qquad \mathcal{E}_\infty := \left\{\xi\xi^T\colon |\xi|=1\right\}.\]
Clearly, we can write \(\mathcal{E}_\infty = \left\{Q^T\bfe_n\bfe_n^TQ\;|\;Q\in O(n)\right\}\) and
\[\mathcal{E}_p = \left\{Q^T\left(I + (p-2)\bfe_n\bfe_n^T\right)Q\;|\;Q\in O(n)\right\},\]
so it follows that \(\mathcal{E}_p = \ext\mathcal{K}_p\) by Lemma \ref{generatesing}.

Intuitively, this is partially the reason why \(\D_p\) is the minimal rotationally invariant operator: Its associated convex body is ``generated'' by a singleton. Namely, by \(E_p := I + (p-2)\bfe_n\bfe_n^T = \diag(1,\dots,1,p-1)\).

In \(\mathbb{R}^n\) we get
\[\varepsilon_p := \Phi(\mathcal{E}_p) = \left\{P\oll(E)\;|\; P\in\mathcal{P}(n),\, E\in\mathcal{E}_p\right\} = \left\{P\bfp\;|\; P\in\mathcal{P}(n)\right\}\]
where \(\bfp = (1,\dots,p-1)^T\) for \(1\leq p<\infty\) and \(\bfp = \bfe_n\) if \(p=\infty\).
\end{remark}

It only remains to prove part \ref{partextext} of Theorem \ref{Phithm}. Namely that
\[\Phi(\ext\mathcal{K}) = \ext\Phi(\mathcal{K}).\]


\begin{proof}[Proof of part \ref{partextext} of Theorem \ref{Phithm}]
Let \(e\in\Phi(\ext\mathcal{K})\). That is, \(e = P\oll(E)\) for some permutation \(P\) and some \(E\in\ext\mathcal{K}\).
We want to show that \(e\in\ext\Phi(\mathcal{K})\). To this end, let \(y,z\in\Phi(\mathcal{K})\setminus\{e\}\), \(0<t<1\), and assume for the sake of contradiction that \(e = ty + (1-t)z\). This implies that
\[P\oll(E) = e = tP'\oll(Y) + (1-t)P''\oll(Z)\]
for some \(P',P''\in\mathcal{P}(n)\) and \(Y,Z\in\mathcal{K}\). Taking \(\diag\) yields
\[E = tQ'^TYQ' + (1-t)Q''^TZQ''\]
for some \(Q',Q''\in O(n)\). This is a convex combination of points in \(\mathcal{K}\). Since \(E\) is an extreme point it follows that 
\[E = Q'^TYQ' = Q''^TZQ''\]
because if \(Q'^TYQ' \neq Q''^TZQ''\), then \(E\in\mathcal{K}^\circ = \mathcal{K}\setminus\ext\mathcal{K}\).

We now have \(P^Te = \oll(E) = \oll(Y) = \oll(Z)\) which means that \(y=P'\oll(Y)\) and \(z = P''\oll(Z)\) are members of the orbit
\[\varepsilon := \left\{Pe\;|\;P\in\mathcal{P}(n)\right\}.\]
Set \(\kappa := \con\varepsilon\). If \(y\neq z\) we get the contradiction
\[e = ty+(1-t)z\in \kappa^\circ = \kappa\setminus\ext\kappa = \kappa\setminus\varepsilon\]
by Lemma \ref{generatesing}. It follows that \(y = z = e\) and we have proved that \(\Phi(\mathcal{K})\setminus\{e\}\) is convex and that
\[\Phi(\ext\mathcal{K}) \subseteq \ext\Phi(\mathcal{K}).\]

%
%
%
%

On the other hand,
\begin{align*}
\Phi(\mathcal{K}^\circ)
	&= \Phi(\mathcal{K}\setminus\ext\mathcal{K})\\
	&= \Phi(\mathcal{K})\setminus\Phi(\ext\mathcal{K}), &&\text{Lemma \ref{syminv},}\\
	&\supseteq \Phi(\mathcal{K})\setminus\left\{\diag E\;|\;E\in\ext\mathcal{K}\right\}, &&\text{Claim \ref{Kdiagprop},}\\
	&= \left\{\diag Y\;|\;Y\in\mathcal{K}\setminus\ext\mathcal{K}\right\}, &&\text{Claim \ref{Kdiagprop},}\\
	&= \left\{\diag Y\;|\;Y\in\mathcal{K}^\circ\right\}\\
	&= \left\{\diag (tY+(1-t)Z)\;|\;t\in(0,1),\,Y\neq Z\in\mathcal{K}\right\}
\end{align*}
and
\begin{align*}
\Phi(\mathcal{K})^\circ
	&= \left\{ty+(1-t)z)\;|\;t\in(0,1),\,y\neq z\in\Phi(\mathcal{K})\right\}\\
	&= \left\{t\diag Y+(1-t)\diag Z\;|\;t\in(0,1),\,\diag Y\neq \diag Z,\, Y,Z\in\mathcal{K}\right\}\\
	&\subseteq \left\{\diag(tY+(1-t)Z)\;|\;t\in(0,1),\, Y\neq Z\in\mathcal{K}\right\}.
\end{align*}
Thus \(\Phi(\mathcal{K})^\circ\subseteq \Phi(\mathcal{K}^\circ)\) and
\begin{align*}
\Phi(\ext\mathcal{K}) &= \Phi(\mathcal{K}\setminus\mathcal{K}^\circ)\\
                      &= \Phi(\mathcal{K})\setminus\Phi(\mathcal{K}^\circ)\\
											&\supseteq \Phi(\mathcal{K})\setminus\Phi(\mathcal{K})^\circ\\
											&= \ext\Phi(\mathcal{K}).
\end{align*}
\end{proof}

\begin{example}[The bodies of Pucci and dominative \(p\)-Laplace revisited]
The extreme points in \(S(n)\) of the dominative body \(\mathcal{K}_p\) and the Pucci body \(\mathcal{K}_{\lambda,\Lambda}\) are
\begin{align*}
\mathcal{E}_p &= \left\{I + (p-2)\xi\xi^T\;|\; \xi\in\mathbb{R}^n,\,|\xi|=1\right\} &&\\
              &= \{I\} + (p-2)\mathcal{E}_\infty, &&1\leq p<\infty,\\
\mathcal{E}_{\lambda,\Lambda} &= \left\{\lambda I + (\Lambda-\lambda)P\;|\; P\in Pr(n)\right\} &&\\
              &= \lambda\{I\} + (\Lambda-\lambda)\mathcal{E}_{0,1}, &&0<\lambda\leq\Lambda,
\end{align*}
respectively. Here,
\[\mathcal{E}_\infty = \left\{\xi\xi^T\;|\; \xi\in\mathbb{R}^n,\,|\xi|=1\right\},\]
and
\[\mathcal{E}_{0,1} = Pr(n)\]
is the set of symmetric \(n\times n\) projection matrices. See Remark \ref{remarkgenerated} and Appendix \ref{ch:visc_ell}.

These are all symmetric sets, and by Proposition \ref{Phibijective} and Theorem \ref{Phithm} we find that the corresponding extreme points in \(\mathbb{R}^n\) are
\begin{align*}
\varepsilon_p &:= \Phi(\mathcal{E}_p) = \{\mathbbm{1}\} + (p-2)\Phi(\mathcal{E}_\infty),\\
\varepsilon_{\lambda,\Lambda} &:= \Phi(\mathcal{E}_{\lambda,\Lambda}) = \lambda\{\mathbbm{1}\} + (\Lambda-\lambda)\Phi(\mathcal{E}_{0,1}).
\end{align*}
We can then compute the corresponding convex bodies in \(\mathbb{R}^n\) as
\begin{align*}
\kappa_p &:= \{\mathbbm{1}\} + (p-2)\kappa_\infty,\\
\kappa_{\lambda,\Lambda} &:= \lambda\{\mathbbm{1}\} + (\Lambda-\lambda)\kappa_{0,1},
\end{align*}
where \(\kappa_\infty := \con\Phi(\mathcal{E}_\infty)\) and \(\kappa_{0,1} := \con\Phi(\mathcal{E}_{0,1})\).

First, we find that
\begin{align*}
\varepsilon_\infty &:= \Phi(\mathcal{E}_\infty)\\
	&= \left\{P\oll(Y)\;|\; P\in\mathcal{P}(n),\, Y\in\mathcal{E}_\infty\right\}\\
	&= \left\{P\bfe_n\;|\; P\in\mathcal{P}(n)\right\}\\
	&= \left\{\bfe_1,\dots,\bfe_n\right\}
\end{align*}
which should be viewed as the vertices of the standard  \((n-1)\)-\emph{simplex} \(\Delta^{n-1}\) in \(\mathbb{R}^n\). Next, we have
\[\varepsilon_{0,1} := \Phi(\mathcal{E}_{0,1}) = \left\{(\delta_1,\dots,\delta_n)^T\in\mathbb{R}^n\;|\; \delta_i = \text{0 or 1}\right\}\]
which are the vertices of the unit cube \([0,1]^n\) in \(\mathbb{R}^n\). See Appendix \ref{ch:visc_ell}.

This yields \(\kappa_\infty = \Delta^{n-1}\), \(\kappa_{0,1} = [0,1]^n\) and by a scaling and a translation along the \(\mathbbm{1}\)-axis, we arrive at the following identifications between our rotationally invariant sublinear operators and their symmetric convex bodies in \(\mathbb{R}^n\).
\begin{align*}
&\text{Dominative operator \(\D_p\)} &&\leftrightarrow && \{\mathbbm{1}\} + (p-2)\Delta^{n-1}.\\
&\text{Pucci operator \(\D_{\lambda,\Lambda}\)} &&\leftrightarrow && \lambda\{\mathbbm{1}\} + (\Lambda-\lambda)[0,1]^n.
\end{align*}
See Figure \ref{fig:bodyproj}.

\begin{figure}[ht]
    \centering
    \begin{subfigure}[b]{0.45\textwidth}
        \includegraphics[width=\textwidth]{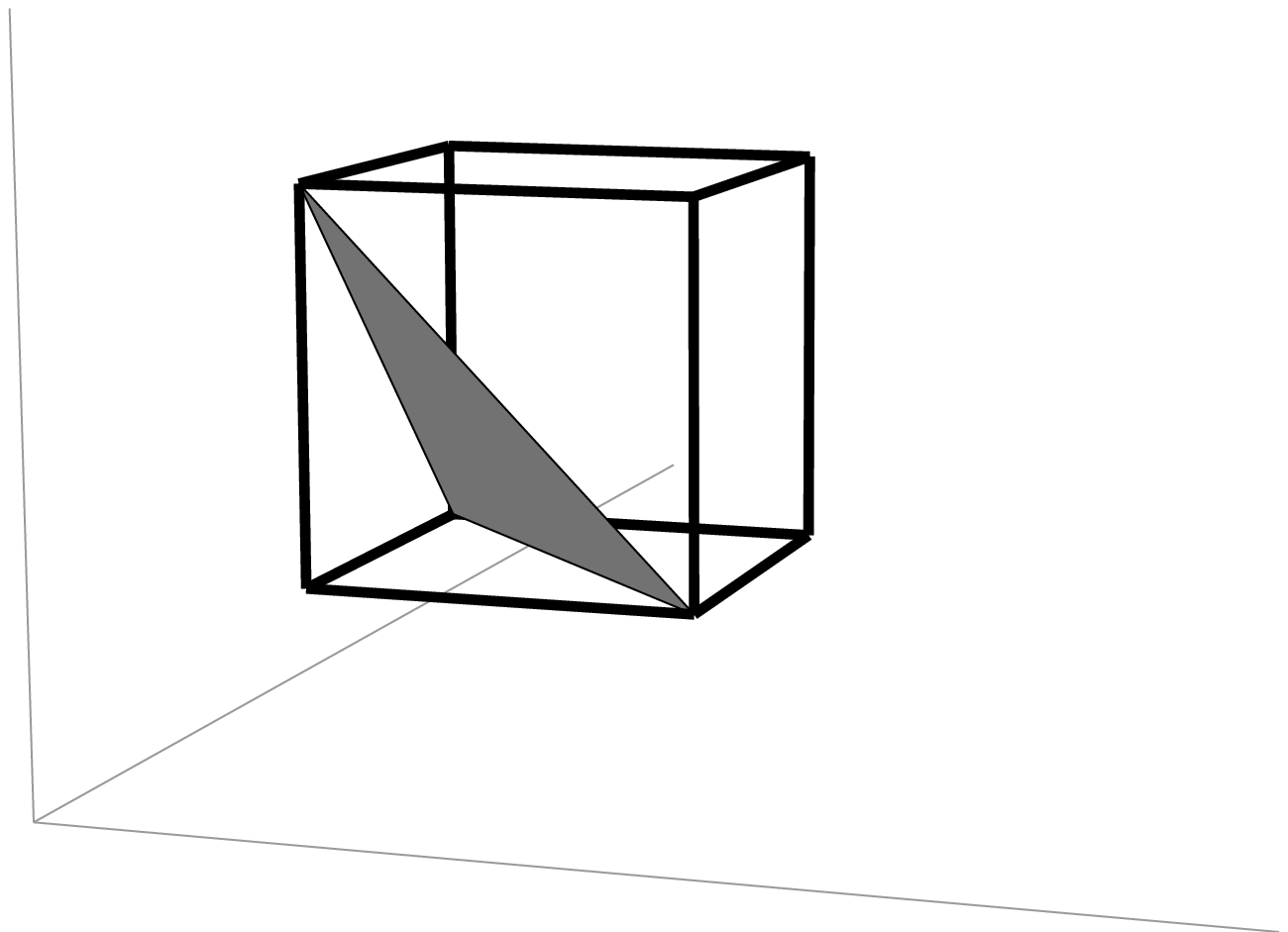}
        \caption{}
        \label{fig:gull}
    \end{subfigure}
    ~ 
    \begin{subfigure}[b]{0.45\textwidth}
        \includegraphics[width=\textwidth]{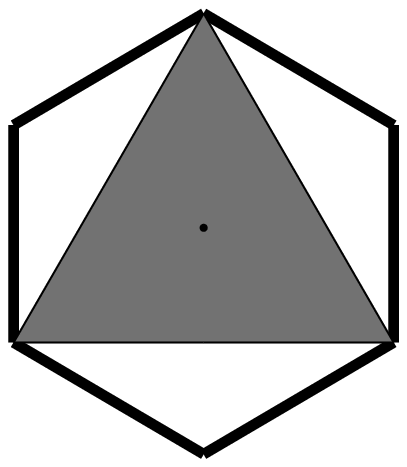}
        \caption{}
        \label{fig:proj}
    \end{subfigure}
    \caption{The corresponding convex bodies of \(\D_p\) (gray) and \(\D_{1,p-1}\) (outlined) in \(\mathbb{R}^3\). Figure (a) reveals their position in the first octant. In (b) the bodies are projected onto the hyperplane given by the simplex, i.e., it shows a cross section of the body cones. The dot in the middle is the Laplacian!}\label{fig:animals}
		\label{fig:bodyproj}
\end{figure}

\end{example}

\section{The body cone aperture}

Given a rotationally invariant elliptic operator with associated convex body \(\mathcal{K}\), we define
\begin{equation}
\alpha = \alpha(\mathcal{K}) := \min_{Z\in\mathcal{K}}\frac{\tr Z}{\lambda_{\max}(Z)}.
\label{aperture}
\end{equation}
For reasons to be explained later, the number \(\alpha\) shall be called the \textbf{solution cone aperture} of the body.

Note first that \(\alpha\) is well-defined whenever \(\mathcal{K}\neq\{0\}\).
Indeed, as the operator is degenerate elliptic, the associated convex body contains only non-negative matrices. Therefore \(\lambda_{\max}(Z)>0\) when \(0\neq Z\in\mathcal{K}\).
Also, \(\alpha\) lies in the interval \([1,n]\), because if \(Z^*\) is a matrix in \(\mathcal{K}\) that minimizes the expression in \eqref{aperture},
then
\[\alpha\]
\[||\]
\[1 = \frac{\lambda_n(Z^*)}{\lambda_n(Z^*)} \leq \frac{\lambda_1(Z^*) + \cdots + \lambda_n(Z^*)}{\lambda_n(Z^*)}\leq \frac{n\lambda_n(Z^*)}{\lambda_n(Z^*)} = n.\]

We may also observe that the fraction in \eqref{aperture} is unchanged by a positive scaling of the body. Thus \(\alpha\) only depends on the \emph{body cone} \(C_\mathcal{K}\) (Definition \ref{bodycone}) and not explicitly of \(\mathcal{K}\) itself. In fact, if we let
\begin{equation}
|Y| := \max_{|\xi|=1}|Y\xi| = \max\left\{|\lambda_1(Y)|, |\lambda_n(Y)|\right\}
\label{opnorm}
\end{equation}
denote the \emph{operator norm} on \(S(n)\), then
\begin{equation}
\alpha = \min_{\substack{Z\in C_{\mathcal{K}}\\ |Z| = 1}}\langle I,Z\rangle.
\label{opening}
\end{equation}

Using Proposition \ref{Kdiagprop} we find, after some calculations, that we also have
\[\alpha = \min_{y\in\kappa}\frac{y^T\mathbbm{1}}{y^\uparrow_n}\]
where \(\kappa = \Phi(\mathcal{K})\) is the corresponding convex body of the operator in \(\mathbb{R}^n\).

\begin{definition}
The \textbf{body cone aperture} of a symmetric convex body \(\mathcal{K}\) is
\begin{equation}
p =p(\mathcal{K}) :=
\begin{cases}
\frac{n+\alpha -2}{\alpha-1}, &\text{if \(1<\alpha \leq n\),}\\
\infty, & \text{if \(\alpha = 1\),}
\end{cases}
\label{pdef}
\end{equation}
where \(\alpha\) is the solution cone aperture given by \eqref{aperture}.
\end{definition}
Notice that \(p\) ranges from 2 to \(\infty\) as \(\alpha\) goes from \(n\) down to 1. They are duals in the sense that
\[(\alpha-1)(p-1) = n-1.\]

If we imagine the ray \(\{tI,t\geq 0\}\) to be the ``axis'' of the cone -- which, as we have seen, is not entirely unjustified --
one may regard the number \(p\) as a measure of the maximal \emph{opening} of the body cone. A small \(p\) implies a narrow body cone. In the extreme case \(p=2\), \((\alpha = n)\), we find that \(\lambda_1(Z) = \cdots = \lambda_n(Z)\) for every \(Z\in \mathcal{K}\) and \(C_\mathcal{K} = C_{\mathcal{K}_2} = \{tI,t\geq 0\}\) is reduced to the axis. At the other end, \(p=\infty\) \((\alpha = 1)\), we get \(\lambda_1(Z^*) = \cdots = \lambda_{n-1}(Z^*) = 0\). This means that \(Z^*\) is a 1-rank matrix and the body cone is as wide as possible for an elliptic operator.

Since the solution cone and the body cone are duals, one shrinks as the other opens and \(\alpha\) is therefore a measure of the aperture of the former.

The choice of using the letter p for the body cone aperture is explained by the fact that \(p(\mathcal{K}_p) = p\) for \(p\geq 2\).

\section{The minimal operator}

The dominative \(p\)-Laplacian \(\D_p u = F_p(\mathcal{H}u)\) satisfies\footnote{for \(2\leq p\leq\infty\) in the pointwise \(C^2\)-sense, and if one reads \(\infty-2\) as 2.}
\[\Delta_p u \leq |\nabla u|^{p-2}\D_p u.\]
Hence the name.

On the other hand, we now show that it provides a \emph{lower} bound amongst the rotationally invariant sublinear elliptic operators.

\begin{theorem}\label{invhol}
Let \(F\colon S(n)\to\mathbb{R}\) be a non-trivial\footnote{We exclude the trivial operator \(F\equiv 0\). i.e. the case \(\mathcal{K}_F = \{0\}\).} rotationally invariant sublinear elliptic operator and
let \(p\in[2,\infty]\) be its body cone aperture \eqref{pdef}.
Then there exists a constant \(c>0\) such that
\[cF_p(X) \leq F(X)\qquad\forall X\in S(n).\]
\end{theorem}

\begin{lemma}\label{ppreclemma}
For \(0\neq z\in\mathbb{R}^n\), where \(z_i\geq 0\), set \(\alpha := z^T\mathbbm{1}/z^\uparrow_n\) and define the vector \(\bfp\) in \(\mathbb{R}^n\) as
\begin{equation}
\bfp := 
\begin{cases}
(1,\dots,1,p-1)^T,\qquad &\text{if \(\alpha\neq 1\)},\\
\bfe_n,&\text{if \(\alpha= 1\)},
\end{cases}
\label{pvecdef}
\end{equation}
where \(p\) is the dual of \(\alpha\). i.e., \((\alpha-1)(p-1) = n-1\).

Write \(c := z^\uparrow_n>0\) if \(\alpha= 1\) and
\[c := \frac{z^T\mathbbm{1}}{\bfp^T\mathbbm{1}}>0\]
otherwise.
Then
\[c\bfp\prec z.\]
\end{lemma}

\begin{proof}
Assume first that \(\alpha\neq 1\). Then
\[c\bfp^T\mathbbm{1} = z^T\mathbbm{1}.\]
Since
\[c = \frac{z^T\mathbbm{1}}{\bfp^T\mathbbm{1}} = \frac{\alpha z^\uparrow_n}{n+p-2} = \frac{z^\uparrow_n}{p-1},\]
we get \(c\bfp_n = z^\uparrow_n\) and thus also
\[\sum_{i=1}^{n-1}z^\uparrow_i = \sum_{i=1}^{n-1}c\bfp_i = c(n-1).\]
We need to show that \(\sum_{i=1}^k z^\uparrow_i \leq c\sum_{i=1}^k\bfp_i = ck\) for all \(k=1,\dots,n-2\). Assume to the contrary that there is a \(k=1,\dots,n-2\) so that \(\sum_{i=1}^k z^\uparrow_i > ck\). Then, \(z^\uparrow_k > c\) and
\begin{align*}
c(n-1) &= \sum_{i=1}^{n-1} z^\uparrow_i\\
       &= \sum_{i=1}^{k}z^\uparrow_i + \sum_{i=k+1}^{n-1}z^\uparrow_i\\
		   &> ck + \sum_{i=k+1}^{n-1}z^\uparrow_i.
\end{align*}
Thus \(\sum_{i=k+1}^{n-1}z^\uparrow_i < c(n-k-1)\) and \(z^\uparrow_{k+1}<c\). But this is a contradiction since \(z^\uparrow_k>c\).

If \(\alpha = 1\) then \(z^\uparrow = (0,\dots,0,z^\uparrow_n)^T\) and \(c\bfp = c\bfe_n = z\).
\end{proof}

We next observe that \eqref{maylem} and the transitivity of \(\prec\) implies the equivalence
\begin{equation}
x\prec y\qquad\iff\qquad \con\left\{Px\;|\;P\in\mathcal{P}(n)\right\}\subseteq\con\left\{Py\;|\;P\in\mathcal{P}(n)\right\}.
\label{transequiv}
\end{equation}

\begin{proof}[Proof of Theorem \ref{invhol}]
We intend to show that \(c\mathcal{K}_p \subseteq \mathcal{K}\) for some scaling \(c>0\), where \(\mathcal{K}_p\) is the dominative body \eqref{dombody} and where \(\mathcal{K}\) is the associated convex body to \(F\). This will imply that
\[cF_p(X) = \max_{Y\in c\mathcal{K}_p}\langle Y,X\rangle \leq \max_{Y\in \mathcal{K}}\langle Y,X\rangle = F(X).\]

Write
\[\alpha = \min_{Y\in\mathcal{K}}\frac{\tr Y}{\lambda_{\max}(Y)} = \min_{y\in\kappa}\frac{y^T\mathbbm{1}}{y^\uparrow_n}.\]
Then \(\alpha = z^T\mathbbm{1}/z^\uparrow_n\) for some \(z\in\kappa\) where \(\kappa = \Phi(\mathcal{K})\) is the corresponding convex body of the operator in \(\mathbb{R}^n\). From Lemma \ref{ppreclemma} there is a \(c>0\) so that \(c\bfp\prec z\), and it follows that
\begin{align*}
c\kappa_p &= c\con\left\{P\bfp\;|\;P\in\mathcal{P}(n)\right\}, && \text{Remark \ref{remarkgenerated},}\\
          &= \con\left\{P(c\bfp)\;|\;P\in\mathcal{P}(n)\right\}\\
	        &\subseteq \con\left\{Pz\;|\;P\in\mathcal{P}(n)\right\}, && \text{by Lemma \ref{ppreclemma} and \eqref{transequiv}},\\
	        &\subseteq \con\left\{Py\;|\;P\in\mathcal{P}(n),\,y\in\kappa\right\}, &&\text{since \(z\in\kappa\),}\\
	        &= \con\kappa = \kappa, &&\text{by symmetry and convexity.}
\end{align*}
That is,
\[c\mathcal{K}_p = \Phi^{-1}(c\kappa_p) \subseteq \Phi^{-1}(\kappa) = \mathcal{K}\]
by Theorem \ref{Phithm}.

\end{proof}

The inclusion \(c\mathcal{K}_p \subseteq \mathcal{K}\) implies \(C_{\mathcal{K}_p} \subseteq C_{\mathcal{K}}\) and Theorem \ref{nesting} is thus at disposal.
By using the Domination and the Nesting property of the dominative operator (Proposition 5 in \cite{bru18}), we get

\begin{corollary}\label{minopcor}
Assume that \(u\) is a supersolution to a non-trivial rotationally invariant sublinear elliptic equation \(F(\mathcal{H}u) = 0\). Let \(p = p(\mathcal{K}_F)\in[2,\infty]\) be the the body cone aperture \eqref{pdef}.
Then the following holds.
\begin{enumerate}
	\item \(u\) is dominative \(p\)-superharmonic.
	\item \(u\) is \(p\)-superharmonic.
	\item \label{partsupminopcpr}\(u\) is superharmonic.
\end{enumerate}
In particular,
\begin{enumerate}\setcounter{enumi}{3}
	\item if \(p>n\), then \(u\) is continuous. \hfill (Thm. 16 \cite{MR2953377})
	\item If \(p=\infty\), then \(u\) is locally concave.
\end{enumerate}
\end{corollary}

The initial result that eventually became generalized to Theorem \ref{invhol} was part \ref{partsupminopcpr} of Corollary \ref{minopcor}.
I thank Fredrik Arbo Høeg who found the key ingredient to that proof. Namely, permutations.

\section{Fundamental solutions}
Not only is \(cF_p \leq F\), \(p = p(\mathcal{K}_F)\), for any rotationally invariant sublinear elliptic \(F\). The following shows that there are rays in \(S(n)\) at which the two operators are both equal to zero. The points on these rays are Hessians of a most important function.

Recall the \emph{fundamental solution} to the (dominative) \(p\)-Laplace equation.

\begin{equation}
w_{n,p}(x) :=
\begin{cases}
-\frac{p-1}{p-n}|x|^\frac{p-n}{p-1}, & p\neq n,\\
-\ln|x|, & p = n,\\
-|x|, & p=\infty.
\end{cases}
\label{sublin_fundsol}
\end{equation}

\begin{theorem}[Existence of fundamental solutions]\label{alpha_fundsol}
Let \(F(\mathcal{H}u) = 0\) be a non-trivial rotationally invariant sublinear elliptic equation and
let \(p = p(\mathcal{K}_F)\in[2,\infty]\) be the the body cone aperture.

Then \(w_{n,p}\)
is a solution in \(\mathbb{R}^n\setminus\{0\}\) and a supersolution in \(\mathbb{R}^n\).
\end{theorem}

\begin{proof}
First we note that \(-\infty<w_{n,p}\leq\infty\) where \(x=0\) is the only possible pole. In any case, there is no test function touching from below at the origin. Therefore, if \(w_{n,p}\) is a solution in \(\mathbb{R}^n\setminus\{0\}\), then it is automatically a supersolution in \(\mathbb{R}^n\).

Writing \(w_{n,p}(x) = W_{n,p}(|x|)\) yields
\begin{align*}
\mathcal{H}w_{n,p}(x) &= W_{n,p}''(r)\hx\hx^T + \frac{W_{n,p}'(r)}{r}\left(I - \hx\hx^T\right)\\
	&= r^{-\alpha}\left[(\alpha-1)\hx\hx^T -\left(I - \hx\hx^T\right)\right],\qquad r:=|x|>0,
\end{align*}
where \(\alpha = \alpha(\mathcal{K}_F)\in[1,n]\) is the solution cone aperture \eqref{aperture}.
Let \(\Lambda_\alpha := \diag\ola\) where
\[\ola := (-1,\dots,-1,\alpha-1)^T = \alpha\bfe_n - \mathbbm{1}.\]
Then \(|x|^{-\alpha}\Lambda_\alpha = \Lambda(\mathcal{H}w_{n,p}(x))\) is a diagonal matrix consisting of the eigenvalues of \(\mathcal{H}w_{n,p}\).  By the rotational invariance and the positive homogenicity,
\[F(\mathcal{H}w_{n,p}(x)) = |x|^{-\alpha}F(\Lambda_\alpha),\qquad x\neq 0,\]
and vanishes independently of \(x\) if \(F(\Lambda_\alpha) = 0\).

Now let \(\kappa := \Phi(\mathcal{K})\) be the corresponding convex body in \(\mathbb{R}^n\).
From Theorem \ref{invhol} and using the fact that \(F_p(\Lambda_\alpha) = 0\),
we have
\begin{equation}
\begin{aligned}
0 &\leq F(\Lambda_\alpha)\\
  &= \max_{Y\in\mathcal{K}}\langle Y,\Lambda_\alpha\rangle\\
	&= \max_{y\in\kappa}y^T\ola\\
	&= \max_{y\in\kappa}\alpha y^\uparrow_n - y^T\mathbbm{1}, &&\text{Lemma \ref{ordermax},}\\
	&= \max_{y\in\kappa}\left(\min_{z\in\kappa}\frac{z^T\mathbbm{1}}{z^\uparrow_n}\, y^\uparrow_n - y^T\mathbbm{1}\right)\\
	&\leq \max_{y\in\kappa}\left(\frac{y^T\mathbbm{1}}{y^\uparrow_n}\, y^\uparrow_n - y^T\mathbbm{1}\right) = 0.
\end{aligned}
\label{lamalpleq}
\end{equation}

It follows that
\(F(\Lambda_\alpha)=0\) and \(w_{n,p}\) is a smooth solution away from the origin.
\end{proof}

\begin{remark}
Theorem \ref{invhol} and Theorem \ref{alpha_fundsol}, together with Theorem \ref{sublin_superpos}, shows that we may have used \emph{any} rotationally invariant sublinear elliptic operator with body cone aperture \(p\) to prove the superposition principle for \(p\)-superharmonic functions in \cite{bru18}. In particular, we could have used the Pucci operator \(\D_{1,p-1}\).
\end{remark}

\begin{remark}[Extremal operators]
Given a rotationally invariant sublinear elliptic operator \(\D u\), let \(\lambda\) and \(\Lambda\) be the smallest and largest eigenvalue found among the matrices in the associated convex body \(\mathcal{K}\), respectively. Then, obviously, \(\mathcal{K}\subseteq\mathcal{K}_{\lambda,\Lambda}\) and by Theorem \ref{invhol}, there is a \(c>0\) and a \(p\in[2,\infty]\) so that
\[c\D_pu \leq \D u \leq \D_{\lambda,\Lambda}u.\]
Moreover, \(w_{n,p}\) is a fundamental solution to \(\D u = 0\).

To illustrate with an example, let \(0\leq\delta\leq 1\) and consider the rotationally invariant sublinear elliptic operator \(\D u = F(\mathcal{H}u)\) defined by the associated convex body
\[\mathcal{K} := \left\{Y\in S(n)\,|\; \|Y - I\|\leq \delta\right\}.\]
Its corresponding convex body in \(\mathbb{R}^3\) is the ball \(\Phi(\mathcal{K}) = \overline{B}(\mathbbm{1},\delta)\)
shown in Figure \ref{fig:spherebody}. The sphere is inscribed in the Pucci cube \(\Phi(\mathcal{K}_{1-\delta,1+\delta})\) and it circumscribes the dominative triangle \(\Phi(c\mathcal{K}_p)\) in such a way that the line passing through the origin and any of the three vertices, is tangent to the sphere. Here \(c>0\) is the constant from Theorem \ref{invhol}
and \(p\) is the body cone aperture of \(\mathcal{K}\). By a little exercise in geometry, we find that
\[c = \frac{1}{3}\left(3 - \delta^2 - \delta\sqrt{\frac{3-\delta^2}{2}}\right),\qquad p = \frac{3 - \delta^2 + 2\delta\sqrt{\frac{3-\delta^2}{2}}}{3 - \delta^2 - \delta\sqrt{\frac{3-\delta^2}{2}}} + 1.\]

In order to give the explicit representation of \(\D u\), we note that 
\[\kappa := \Phi(\mathcal{K}) = \overline{B}(\mathbbm{1},\delta) = \mathbbm{1} + \delta\overline{B}(0,1)\]
and that \(\max_{|y|=1}y^Tx = \frac{x^T}{|x|}x = |x|\) for \(x\in\mathbb{R}^n\). This yields
\begin{align*}
F(X) &= \max_{y\in\ext\kappa}y^T\oll(X)\\
     &= \mathbbm{1}^T\oll(X) + \delta\max_{|y|=1}y^T\oll(X)\\
		 &= \tr X + \delta |\oll(X)|,
\end{align*}
or
\[\D u = \Delta u + \delta\sqrt{\lambda_1^2(\mathcal{H}u) + \lambda_2^2(\mathcal{H}u) + \lambda_3^2(\mathcal{H}u)}.\]

\begin{figure}[ht]%
\centering\includegraphics{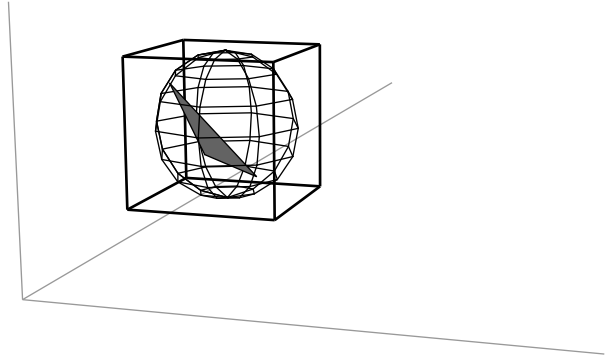}%
\caption{\(c\mathcal{K}_p\subseteq \mathcal{K}\subseteq\mathcal{K}_{\lambda,\Lambda}\)}%
\label{fig:spherebody}%
\end{figure}
\end{remark}

\begin{example}[Singleton-generated sublinear operators]
As noted in Remark \ref{remarkgenerated}, the dominative body is generated by a single matrix \(E_p = \diag\bfp\). Since \(E_p\) only has two different eigenvalues, \(1\) and \(p-1\), it is the simplest generator that is not just a scaling of the identity matrix. Of course,
\[F_p(X) = \bfp^T\oll(X).\]
An analogous formula also hold for arbitrary singleton-generated operators:

\begin{proposition}
Let \(A\in S(n)\) and set
\[\mathcal{E} := \left\{Q^TAQ\;|\;Q\in O(n)\right\}\in\mathcal{E}(n),\qquad \mathcal{K} := \con\mathcal{E}\in \mathcal{E}(n)\cap\mathcal{K}(n).\]
Then the support function of \(\mathcal{K}\) is the rotationally invariant sublinear operator \(F\colon S(n)\to\mathbb{R}\) given by
\[F(X) = \oll(A)^T\oll(X).\]
\end{proposition}

\begin{proof}
We have \(\ext\mathcal{K} = \mathcal{E}\) by Lemma \ref{generatesing} and \(\Phi(\mathcal{E}) = \left\{P\oll(A)\;|\; P\in\mathcal{P}(n)\right\}\). It follows that
\begin{align*}
F(X) &= F(\Lambda(X))\\
     &= \max_{Y\in\mathcal{K}}\langle Y,\Lambda(X)\rangle\\
		 &= \max_{Y\in\mathcal{K}}(\diag Y)^T\oll(X)\\
		 &= \max_{y\in\Phi(\mathcal{K})}y^T\oll(X), && \text{by Proposition \ref{Phibijective},}\\
		 &= \max_{y\in\Phi(\mathcal{E})}y^T\oll(X), && \text{by Prop. \ref{convprop} and Thm. \ref{Phithm} (2),}\\
		 &= \max_{P\in\mathcal{P}(n)}(P\oll(A))^T\oll(X),\\
		 &= \oll(A)^T\oll(X), && \text{by \eqref{ordermax}.}\\
\end{align*}
\end{proof}

\begin{corollary}[von Neumann's trace inequality]\label{vonNeu}
Let \(A,B\in S(n)\). Then
\[\tr(AB) \leq \oll(A)^T\oll(B).\]
\end{corollary}

\begin{proof}
Define \(\mathcal{K}\) and \(F\) as above. Then
\[\langle A,B\rangle \leq \max_{Y\in\mathcal{K}}\langle Y,B\rangle = F(B) = \oll(A)^T\oll(B).\] 
\end{proof}

In advance, it was not immediately obvious that operators in the form
\[F_\bfa(X) := \bfa^T\oll(X),\qquad \bfa = (a_1,\dots,a_n)^T\in\mathbb{R}^n,\, a_1\leq\cdots\leq a_n,\]
are sublinear. It is, however, crucial that \(\bfa\) is ordered, i.e., \(\bfa = \bfa^\uparrow\).
Of course, \(F\) is elliptic if and only if \(a_1\geq 0\).

The associated convex body \(\mathcal{K}_\bfa\) of a singleton-generated operator \(F_\bfa\), \(\bfa = \oll(A)\), is \emph{flat} in the sense that it is a subset of the hyperplane \(\{Y\;|\; \tr Y = \tr A\}\) in \(S(n)\). This is also a consequence of the fact that \(F_\bfa(I) = -F_\bfa(-I)\).

The solution cone aperture to \(F_\bfa\) is
\[\alpha = \min_{Y\in\mathcal{K}_\bfa}\frac{\tr Y}{\lambda_{\max}(Y)} = \frac{a_1 + \cdots + a_n}{a_n}\]
and the body cone aperture is then
\[p = (n-1)\frac{a_n}{a_1 + \cdots + a_{n-1}} + 1.\]
The convex body \(\mathcal{K}_\bfa\) contains the scaled dominative body \(c\mathcal{K}_p\) where
\[c = \frac{a_1 + \cdots + a_{n-1}}{n-1}.\]

\begin{figure}[ht]
    \centering
    \begin{subfigure}[b]{0.45\textwidth}
        \includegraphics[width=\textwidth]{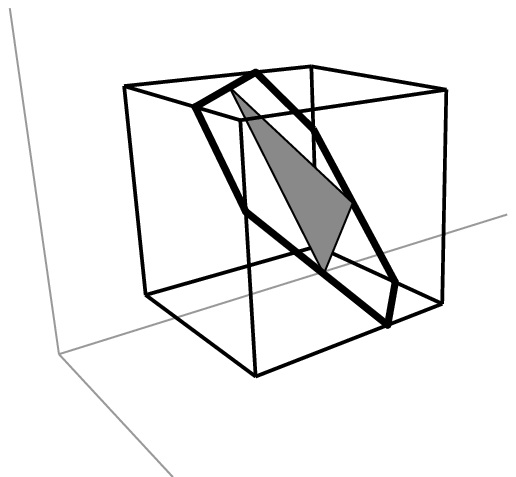}
        \caption{}
        \label{fig:flatbodya}
    \end{subfigure}
    ~ 
    \begin{subfigure}[b]{0.45\textwidth}
        \includegraphics[width=\textwidth]{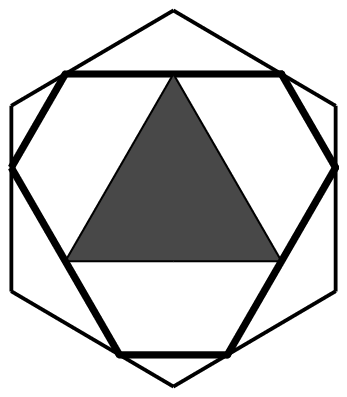}
        \caption{}
        \label{fig:flatbodyb}
    \end{subfigure}
    \caption{\(2\mathcal{K}_3\subseteq \mathcal{K}_{(1,3,4)}\subseteq\mathcal{K}_{1,4}\)}
		\label{fig:flatbody}
\end{figure}

The non-regular hexagon in Figure \ref{fig:flatbody} is the corresponding convex body in \(\mathbb{R}^3\) to the operator
\[F_{(1,3,4)}(X) = \lambda_1(X) + 3\lambda_2(X) + 4\lambda_3(X).\]
It is inscribed in the Pucci cube \(\Phi(\mathcal{K}_{1,4})\) and it circumscribes the dominative triangle \(\Phi(c\mathcal{K}_p)\). Here, \(p = 2\cdot\frac{4}{4}+1 = 3\) and \(c = \frac{1+3}{2} = 2\). It follows, for example,  that \(-\ln|x|\) is a solution to the equation \(F_{(1,3,4)}(\mathcal{H}u) = 0\)
in \(\mathbb{R}^3\setminus\{0\}\).
\end{example}

A natural next thing to consider is \emph{finitely} generated rotationally invariant sublinear operators. They will have a convex body on the form
\[\con\bigcup_{i=1}^M\mathcal{K}_{\bfa_i} = \con\left\{Q(\diag\bfa_i)Q^T\;|\; Q\in O(n),\, i = 1,\dots,M\right\}.\]
In \(\mathbb{R}^3\) the corresponding body will be the convex hull of \(M\) hexagons (possibly non-regular, or degenerate like triangles and points) of different sizes situated along the \(\mathbbm{1}\)-axis. We observe that the Pucci cube \(\Phi(\mathcal{K}_{1,4})\) in Figure \ref{fig:flatbody} is generated by two points and two triangles:
\[(1,1,1),\quad (1,1,4),\quad (1,4,4),\quad (4,4,4).\]


\chapter{Uniform ellipticity}\label{chp:unell}
We now consider sublinear and \emph{uniformly} elliptic equations
\begin{equation}
F(\mathcal{H}u) = 0.
\label{unifeq}
\end{equation}
According to Theorem \ref{repsupthm} and Proposition \ref{definite}, the function \(F\colon S(n)\to\mathbb{R}\) is on the form
\[F(X) = \max_{Y\in\mathcal{K}}\langle Y,X\rangle = \max_{Y\in\mathcal{K}}\tr(YX)\]
for some nonempty compact and convex subset \(\mathcal{K}\subseteq S(n)\) containing only positive definite matrices -- the \emph{associated convex body} to \(F\).

With this stronger notion of ellipticity, more can be said about the operators. This is largely due to the fact that every \(Z\in\mathcal{K}\) now can be assumed to be \emph{invertible}.

As examples, one may again think of the dominative \(p\)-Laplacian, the Pucci operators and the linear operators \(L(X) = \tr(AX)\). But now we must restrict \(p\) to the open interval \((1,\infty)\) and \(A\) must be positive definite.

\section{Supersolutions and superharmonicity}
\begin{definition}
For a function \(u\) in \(\Omega\subseteq\mathbb{R}^n\) and a matrix \(0< Z\in S(n)\), we define the function \(u_Z\) as
\begin{equation}
u_Z(x) := u\left(\sqrt{Z}\,x\right)
\label{uzdef}
\end{equation}
in the corresponding domain
\[\Omega_Z := \left\{x\in\mathbb{R}^n\;|\; \sqrt{Z}\,x\in\Omega\right\}.\]
\end{definition}
By the square root, we mean the unique symmetric matrix \(\sqrt{Z}> 0\) satisfying \(\sqrt{Z}\sqrt{Z} = Z\).

\begin{theorem}\label{supsupthm}
Let \(F(\mathcal{H}u) = 0\) be a sublinear and uniformly elliptic equation with associated convex body \(\mathcal{K}\subseteq S(n)\). Let \(u\) be a function in \(\Omega\subseteq\mathbb{R}^n\). Then
\center
\(u\) is a supersolution in \(\Omega\)
\[\iff\]
\(u_Z\) is superharmonic in \(\Omega_Z\) for every \(Z\in\mathcal{K}\).
\end{theorem}

Note the tautology if \(F\) is the Laplacian: \(\mathcal{K} = \{I\}\).

\begin{proof}
We first observe that lower semi-continuity and finiteness on dense subsets are carried over in either direction.

Let \(u\) be a supersolution to \eqref{unifeq} in \(\Omega\) and let \(Z\in\mathcal{K}\).
Assume that \(\phi\) is a test function to \(u_Z\) from below a point \(x_0\in\Omega_Z\):
\[\phi(x_0) = u_Z(x_0) = u\left(\sqrt{Z}\,x_0\right),\qquad \phi(x) \leq u_Z(x) = u\left(\sqrt{Z}\,x\right)\]
for \(x\) near \(x_0\). Write \(y_0 := \sqrt{Z}\,x_0\in\Omega\). Then
\[\phi_{Z^{-1}}(y_0) = \phi\left(\sqrt{Z}^{-1}y_0\right) = \phi(x_0) = u(y_0),\]
and when \(y\) is near \(y_0\), then \(\sqrt{Z}^{-1}\,y\) is near \(\sqrt{Z}^{-1}\,y_0 = x_0\) and
\[\phi_{Z^{-1}}(y) = \phi\left(\sqrt{Z}^{-1}y\right) \leq u_Z\left(\sqrt{Z}^{-1}y\right) = u(y).\]
Thus \(\phi_{Z^{-1}}\) is a test function to \(u\) from below at \(y_0 := \sqrt{Z}\,x_0 \in\Omega\). This means that
\begin{equation}
F\left(\mathcal{H}\phi_{Z^{-1}}(y_0)\right) \leq 0.
\label{testres}
\end{equation}
Now, \(\phi(x) = \phi_{Z^{-1}}\left(\sqrt{Z}\,x\right)\) and
\[\nabla\phi(x_0) = \nabla\phi_{Z^{-1}}(y_0)\sqrt{Z},\qquad \mathcal{H}\phi(x_0) = \sqrt{Z}\,\mathcal{H}\phi_{Z^{-1}}(y_0)\sqrt{Z}.\]
It follows that
\begin{align*}
\Delta\phi(x_0) &= \tr\left(\sqrt{Z}\,\mathcal{H}\phi_{Z^{-1}}(y_0)\sqrt{Z}\right)\\
                &= \tr\left(Z\,\mathcal{H}\phi_{Z^{-1}}(y_0)\right)\\
								&\leq \max_{Y\in\mathcal{K}}\tr\left(Y\,\mathcal{H}\phi_{Z^{-1}}(y_0)\right)\\
								&= F\left(\mathcal{H}\phi_{Z^{-1}}(y_0)\right) \leq 0
\end{align*}
and we conclude that \(u_Z\) is superharmonic in \(\Omega_Z\).

For the other direction, assume that \(u\) is not a supersolution. Then there exists a test function \(\phi\) from below at some point \(y_0\in\Omega\) and a matrix \(Z\in\mathcal{K}\) such that
\[0 < F(\mathcal{H}\phi(y_0)) = \max_{Y\in\mathcal{K}}\tr(Y\,\mathcal{H}\phi(y_0)) = \tr(Z\,\mathcal{H}\phi(y_0)).\]
As above, \(\phi_Z\) will be a test function to \(u_Z\) from below at \(x_0 := \sqrt{Z}^{-1}y_0\in\Omega_Z\) and
\[\Delta\phi_Z(x_0) = \tr (Z\,\mathcal{H}\phi(y_0)) > 0.\]
Thus there exists a \(Z\in\mathcal{K}\) such that \(u_Z\) is \emph{not} superharmonic in \(\Omega_Z\).
\end{proof}

\begin{remark}
The proof shows that the Theorem can be strengthened: If \(u\) is a supersolution, then \(u_Z\) is superharmonic for every \(Z\in C_{\mathcal{K}}\setminus\{0\}\). For the other direction, it is sufficient to require that \(u_Z\) is superharmonic for every \(Z\in\ext\mathcal{K}\).
\end{remark}

Of course, Corollary \ref{unifsuppos} below follows from the general superposition principle (Theorem \ref{sublin_superpos}). But if one is willing to accept that a sum of superharmonic functions is superharmonic, it is perhaps interesting to see that superposition in the uniformly elliptic case is an immediate consequence of Theorem \ref{supsupthm}.

\begin{corollary}\label{unifsuppos}
Let \(F(\mathcal{H}w) = 0\) be a sublinear and uniformly elliptic equation and let \(\Omega\subseteq\mathbb{R}^n\). Then
\[\text{\(u\) and \(v\) are supersolutions in \(\Omega\)}\qquad\Rightarrow\qquad
\text{\(u+v\) is a supersolution in \(\Omega\).}\]
\end{corollary}

\begin{proof}
Let \(\mathcal{K}\subseteq S(n)\) be the associated convex body to the operator. Then \(u_Z\) and \(v_Z\) are superharmonic for all \(Z\in\mathcal{K}\) by Theorem \ref{supsupthm}. Since a sum of superharmonic functions is superharmonic, it follows that \((u+v)_Z = u_Z + v_Z\) is superharmonic for all \(Z\in\mathcal{K}\) and thus \(u+v\) is a supersolution.
\end{proof}

\section{Mean value}\label{sec:meanvalue}

Given a positive definite matrix \(Z = \sum_{i=1}^n\lambda_i\xi_i\xi_i^T\), a point \(y_0\in\mathbb{R}^n\) and a number \(\epsilon>0\), we define the solid ellipsoid
\begin{equation}
\begin{aligned}
\mathcal{E}_Z(y_0,\epsilon) &:= \left\{y\in\mathbb{R}^n\;|\; (y-y_0)^T Z^{-1}(y-y_0) < \epsilon^2\right\}\\
	& = \left\{y\in\mathbb{R}^n\;\colon\; |\sqrt{Z}^{-1}(y-y_0)| < \epsilon\right\}
\end{aligned}
\label{ezdef}
\end{equation}
in \(\mathbb{R}^n\).
The ellipsoid is centered at \(y_0\) and has semi axes of length \(\epsilon\sqrt{\lambda_i}>0\) in the corresponding directions \(\xi_i\).

The integrals below are Lebesgue integrals and notice that every semicontinuous function is measurable by part \ref{lsc_open} of Proposition \ref{lsceqapp}. We define \(\fint_D f\dd x := \frac{1}{|D|}\int_D f\dd x\) where \(|D|\) is the Lebesgue measure of the subset \(D\subseteq\mathbb{R}^n\). That is, \(\fint_D f\dd x\) is the average, or the \emph{mean value}, of \(f\) in \(D\). Superharmonic functions are exactly those functions that are larger than its mean value in any ball of its domain:

\begin{center}
\(u\) is superharmonic in \(\Omega\)
\[\iff\]
\(u\not\equiv\infty\) is l.s.c. and for each \(y_0\in\Omega\) and every \(\epsilon>0\) so that \(B(y_0,\epsilon)\subseteq\Omega\),
\[\fint_{B(y_0,\epsilon)} u(y)\dd y \leq u(y_0).\]
\end{center}
The lower part has often been used as a definition of superharmonicity, but it is not trivial that it actually is equivalent to the more modern notion of a viscosity supersolution.
See e.g. \cite{MR1871417}.

\begin{theorem}\label{mvthm}
Let \(F(\mathcal{H}w) = 0\) be a sublinear and uniformly elliptic equation with associated convex body \(\mathcal{K}\subseteq S(n)\). Then
\center
\(u\) is a supersolution in \(\Omega\)
\[\iff\]
\(u\not\equiv\infty\) is l.s.c. and for each \(y_0\in\Omega\) and every \(\epsilon>0\) so that \(\bigcup_{Z\in\mathcal{K}}\mathcal{E}_Z(y_0,\epsilon)\subseteq\Omega\), we have
\[\sup_{Z\in\mathcal{K}}\fint_{\mathcal{E}_Z(y_0,\epsilon)} u(y)\dd y \leq u(y_0).\]
\end{theorem}

\begin{proof}
For \(Z\in\mathcal{K}\) the substitution \(y(x) = \sqrt{Z}\,x\) yields
\[y\in\mathcal{E}_Z(y_0,\epsilon)\qquad\iff\qquad x\in B(x_0,\epsilon)\]
where \(y_0 = \sqrt{Z}\,x_0\). Also, the Jacobian determinant is the constant \(\sqrt{\det Z}\) and thus the volume of the ellipsoid is
\[|\mathcal{E}_Z(y_0,\epsilon)| = \sqrt{\det Z}|B(x_0,\epsilon)|.\]
Therefore, if \(u_Z\) is integrable over the ball,
\[\fint_{\mathcal{E}_Z(y_0,\epsilon)}u(y)\dd y = \frac{\sqrt{\det Z}}{|\mathcal{E}_Z(y_0,\epsilon)|}\int_{B(x_0,\epsilon)}u\left(\sqrt{Z}\,x\right)\dd x = \fint_{B(x_0,\epsilon)}u_Z(x)\dd x.\]
Since superharmonic functions are locally integrable, this shows that \emph{supersolutions to uniformly elliptic equations are locally integrable} as well.
See Theorem 1.25 in \cite{MR0435422}.

Assume that \(u\) is a supersolution to \(F(\mathcal{H}u)=0\) in \(\Omega\). Then \(u_Z\) is superharmonic in \(\Omega_Z\) for every \(Z\in\mathcal{K}\) by Theorem \ref{supsupthm}. Thus, for each \(y_0\in\Omega\) and \(\epsilon>0\) small we get that
\[\fint_{\mathcal{E}_Z(y_0,\epsilon)}u(y)\dd y = \fint_{B(x_0,\epsilon)}u_Z(x)\dd x \leq u_Z(x_0) = u(y_0).\]
Taking the supremum over \(Z\in\mathcal{K}\) yields the conclusion.

Now assume that for each \(y_0\in\Omega\) and every \(\epsilon>0\) sufficiently small,
\(\sup_{Z\in\mathcal{K}}\fint_{\mathcal{E}_Z(y_0,\epsilon)} u(y)\dd y \leq u(y_0).\)
Let \(Z\in\mathcal{K}\), \(x_0\in\Omega_Z\), and write \(y_0 := \sqrt{Z}\,x_0\in\Omega\). Then \(\fint_{\mathcal{E}_Z(y_0,\epsilon)} u(y)\dd y \leq u(y_0)\) and
\begin{align*}
\fint_{B(x_0,\epsilon)}u_Z(x)\dd x
	&= \fint_{\mathcal{E}_Z(y_0,\epsilon)}u(y)\dd y\\
	&\leq u(y_0)\\
	&= u_Z(x_0).
\end{align*}
We conclude that \(u_Z\) is superharmonic in \(\Omega_Z\) and \(u\) is a supersolution in \(\Omega\) by Theorem \ref{supsupthm}.
\end{proof}

For \(C^2\) functions \(\phi\), a Taylor expansion and the standard identity
\[\fint_{B(0,\epsilon)} x^TAx\dd x = \frac{\epsilon^2}{n+2}\tr A,\]
yields the estimate
\[\fint_{B(x_0,\epsilon)} \phi(x)\dd x = \phi(x_0) + \frac{\epsilon^2}{2(n+2)}\Delta\phi(x_0) + o(\epsilon^2)\]
as \(\epsilon\to 0\). The linear term in the expansion vanishes since it is an odd function in the ball.
It follows that
\begin{align*}
\sup_{Z\in\mathcal{K}}\fint_{\mathcal{E}_Z(y_0,\epsilon)} \phi(y)\dd y
					&= \sup_{Z\in\mathcal{K}}\fint_{B(x_0,\epsilon)} \phi_Z(x)\dd x\\
					&= \sup_{Z\in\mathcal{K}}\left(\phi_Z(x_0) + \frac{\epsilon^2}{2(n+2)}\Delta\phi_Z(x_0) + o(\epsilon^2)\right)\\
					&= \sup_{Z\in\mathcal{K}}\left(\phi(y_0) + \frac{\epsilon^2}{2(n+2)}\tr\left(\sqrt{Z}\mathcal{H}\phi(y_0)\sqrt{Z}\right) + o(\epsilon^2)\right)\\
					&= \phi(y_0) + \frac{\epsilon^2}{2(n+2)}F(\mathcal{H}\phi(y_0)) + o(\epsilon^2)
\end{align*}
and thus
\begin{equation}
\lim_{\epsilon\to 0}\frac{1}{\epsilon^2}\left[\sup_{Z\in\mathcal{K}}\fint_{\mathcal{E}_Z(y_0,\epsilon)} \phi(y)\dd y - \phi(y_0)\right]
 = \frac{1}{2(n+2)}F(\mathcal{H}\phi(y_0)).
\label{limc2}
\end{equation}

This is in fact the proof of Theorem \ref{mvthmsub} below when applied to the test functions. It is a stronger version of the upward direction of Theorem \ref{mvthm}, for which also the corresponding claim for subsolutions is true.

\begin{theorem}\label{mvthmsub}
Let \(F(\mathcal{H}w) = 0\) be a sublinear and uniformly elliptic equation with associated convex body \(\mathcal{K}\subseteq S(n)\). Let
\begin{itemize}
	\item \(u\colon\Omega\to(-\infty,\infty]\) be l.s.c.,
	\item \(v\colon\Omega\to[-\infty,\infty)\) be u.s.c.,
\end{itemize}
and let both be finite on a dense subset. Assume that, for each \(y_0\in\Omega\),
\[\sup_{Z\in\mathcal{K}}\fint_{\mathcal{E}_Z(y_0,\epsilon)} u(y)\dd y \leq u(y_0) + o(\epsilon^2),\quad\sup_{Z\in\ext\mathcal{K}}\fint_{\mathcal{E}_Z(y_0,\epsilon)} v(y)\dd y \geq v(y_0) + o(\epsilon^2)\]
as \(\epsilon\to 0\).
Then \(u\) is a supersolution and \(v\) is a subsolution in \(\Omega\).
\end{theorem}

%
%

Finally, we state Theorem \ref{mvthmeq}. It is true by Theorem \ref{mvthmsub} and the fact that solutions of uniformly elliptic sublinear equations are \(C^{2,\alpha}\).

\begin{theorem}\label{mvthmeq}
Let \(F(\mathcal{H}w) = 0\) be a sublinear and uniformly elliptic equation with associated convex body \(\mathcal{K}\). Assume that \(w\) is a continuous function in \(\Omega\). Then
\[\sup_{Z\in\mathcal{K}}\fint_{\mathcal{E}_Z(y_0,\epsilon)} w(y)\dd y = w(y_0) + o(\epsilon^2)\]
\center
as \(\epsilon\to 0\) for all \(y_0\in\Omega\)
\[\iff\]
\(w\) is a solution in \(\Omega\).
\end{theorem}

\section{Convolution with positive measures}\label{sec:convpos}
\begin{theorem}\label{convthm}
Let \(u\) be a supersolution to a sublinear and uniformly elliptic equation in \(\mathbb{R}^n\) and let \(\mu\geq 0\) be a Radon measure. Define the convolution
\[w(x) := (u*\mu)(x) = \int_{\mathbb{R}^n}u(x-y)\mu(\dd y)\]
and assume that one of the following conditions hold.
\begin{enumerate}
	\item \(w\not\equiv\infty\) and \(\forall R>0\) there exists a \(\mu\)-integrable function \(C_R\) such that \(u^-(x-y)\leq C_R(y)\) for all \(x\in B_R := B(0,R)\).
	\item \(w\not\equiv\infty\) and \(u\geq 0\).
	\item \(u\) is a fundamental function, \(u(x) = c|x|^\frac{p-n}{p-1}\), \(c\geq 0\), \(2\leq p< n\), and \(\int_{|y|\geq 1}|y|^\frac{p-n}{p-1}\mu(\dd y)<\infty\).
	\item \(w\not\equiv\infty\) and \(\mu\) has compact support. 
\end{enumerate}
Then \(w\) is a supersolution in \(\mathbb{R}^n\).
\end{theorem}

We denote the positive and negative part of a function \(f\) by \(f^+(x) := \max\{f(x),0\}\geq 0\) and \(f^-(x) := -\min\{f(x),0\}\geq 0\), respectively.

\begin{proof}[Proof]
%
Assume condition 1. Fix \(R>0\). For \(k = 1,2,\dots\) write \(B_k := B(0,k)\) and define the functions \(f_k\) in \(B_R\) as
\[f_k(x) := \int_{B_k} \left(u^+(x-y) - u^-(x-y) + C_R(y)\right)\mu(\dd y).\]
By assumption, the integrand is non-negative for all \(y\in\mathbb{R}^n\) when \(x\in B_R\).
Thus \((f_k)\) is an increasing sequence converging pointwise in \(B_R\) to
\[f(x) := \int_{\mathbb{R}^n} \left(u(x-y) + C_R(y)\right)\mu(\dd y) = w(x) + S_R\]
where \(S_R := \int_{\mathbb{R}^n}C_R(y)\mu(\dd y)<\infty\). Also, for each \(k\),
\begin{align*}
\liminf_{x\to x_0}f_k(x)
	&\geq \int_{B_k} \liminf_{x\to x_0}\left(u(x-y) + C_R(y)\right)\mu(\dd y)\\
	&\geq \int_{B_k} \left(u(x_0-y) + C_R(y)\right)\mu(\dd y) = f_k(x_0),
\end{align*}
by Fatou's Lemma (Theorem 1.24 \cite{MR0435422}) and since \(u\) is l.s.c. This means that \(f\) is l.s.c (Proposition \ref{semicontsup})  and that \(w|_{B_R}\) is as well.

Next, for \(Z\) in the associated convex body, \(\epsilon>0\), and \(x_0\in\mathbb{R}^n\), the mean value property and Fubini's Theorem yields
\begin{align*}
\fint_{\mathcal{E}_Z(x_0,\epsilon)} f(x)\dd x
	&= \fint_{\mathcal{E}_Z(x_0,\epsilon)} \int_{\mathbb{R}^n}\left(u(x-y) + C_R(y)\right)\mu(\dd y)\dd x\\
	&= \int_{\mathbb{R}^n}\fint_{\mathcal{E}_Z(x_0,\epsilon)} \left(u(x-y) + C_R(y)\right)\dd x\mu(\dd y)\\
	&\leq \int_{\mathbb{R}^n}\left(u(x_0-y) + C_R(y)\right)\mu(\dd y), && \text{by Thm \ref{mvthm},}\\
	&= f(x_0)
\end{align*}
and \(f\) is a supersolution in \(B_R\) by Theorem \ref{mvthm}. Since \(w(x) = f(x) - S_R\) and \(R>0\) was arbitrary, we conclude that \(w\) is a supersolution in \(\mathbb{R}^n\).

Condition 2 is sufficient by condition 1: Set \(C_R(y)\equiv 0\).

Assume condition 3. By part 2, we only need to show that \(w\not\equiv\infty\). But that is exactly what is insured by Theorem 3.4 in \cite{MR0435422}. 

Assume condition 4. Let \(K\) be a compact set containing the support of \(\mu\). i.e. \(\mu(\mathbb{R}^n) = \mu(K)<\infty\). Since \(u\) is l.s.c., it obtains its minimum on compact sets (Proposition \ref{lscmin}). Set therefore
\[C_R(y) := \max_{x\in\overline{B}_R}u^-(x-y) < \infty.\]
Then \(C_R(y) \geq u^-(x-y)\) for all \(x\in B_R\). It is \(\mu\)-integrable, since for \(y\in K\) we have
\[C_R(y) \leq \max_{z\in\overline{B}_R-K}u^-(z) =: M_R < \infty\]
and thus
\[\int_{\mathbb{R}^n}C_R(y)\mu(\dd y) = \int_K C_R(y)\mu(\dd y) \leq M_R\mu(K)<\infty.\]
By part 1, it follows that \(w\) is a supersolution in \(\mathbb{R}^n\).
\end{proof}

\begin{corollary}
Suppose that \(u\) is a supersolution to a sublinear and uniformly elliptic equation in \(\mathbb{R}^n\).
\begin{enumerate}
	\item Let \(Z\) be a matrix in the operators associated convex body. Then
	\[w_{Z,\epsilon}(x) := \fint_{\mathcal{E}_Z(x,\epsilon)} u(y)\dd y\]
	is a continuous supersolution in \(\mathbb{R}^n\) converging pointwise to \(u\) from below as \(\epsilon\to 0\).
	\item Let \(\delta_\epsilon\in C^\infty_0(\mathbb{R}^n)\) be the standard mollifier and set
	\[w_\epsilon(x) := \int_{\mathbb{R}^n} \delta_\epsilon(y) u(x-y)\dd y.\]
	Then \(w_\epsilon\) is a \(C^\infty\) supersolution in \(\mathbb{R}^n\) that converges almost everywhere to \(u\) as \(\epsilon\to 0\).
\end{enumerate}
\end{corollary}

\begin{proof}[Proof of 1]
By a change of variables, \(w_{Z,\epsilon}\) is continuous by Theorem 2.35 in \cite{MR0435422} since \(u\) is locally integrable. Also \(\lim_{\epsilon\to 0}w_{Z,\epsilon}(x) = u(x)\) by Theorem 2.30 in the same reference. Furthermore, \(w_{Z,\epsilon}(x)\leq u(x)\) by Theorem \ref{mvthm} and \(w_{Z,\epsilon}\) is a supersolution by part 4 of Theorem \ref{convthm} by setting
\[\mu_{Z,\epsilon}(E) := \frac{|E\cap\mathcal{E}_Z(0,\epsilon)|}{|\mathcal{E}_Z(0,\epsilon)|}.\]
\end{proof}

\begin{proof}[Proof of 2]
\(w_\epsilon\) is a supersolution by part 4 of Theorem \ref{convthm} by setting \(\mu_\epsilon(E) := \int_E \delta_\epsilon(y)\dd y\). The rest of the claims follow from Theorem 1.29 in \cite{MR0435422}. 
\end{proof}


\section{The strong comparison principle}\label{sec:strongcomp}

In Appendix \ref{ch:visc_ell} it is shown that an operator \(\D u = F(x,\mathcal{H}u)\) is uniformly elliptic if and only if there exists constants \(0<\lambda\leq \Lambda\) so that
\begin{equation}
F(x,A+B) - F(x,A) \leq P_{\lambda,\Lambda}(B)
\label{xunifdef}
\end{equation}
for all \(x\in\mathbb{R}^n\) and all \(A,B\in S(n)\).

Assume in addition that \(F(x,0)\equiv 0\).
Taking \(B=X, A=0\) and then \(B=-X, A=X\), \eqref{xunifdef} implies
\[-P_{\lambda,\Lambda}(-X)\leq F(x,X) \leq P_{\lambda,\Lambda}(X)\]
for all \(x\in\mathbb{R}^n\) and all \(X\in S(n)\). We have proved

\begin{proposition}\label{subsubpuc}
Let \(\D w = F(x,\mathcal{H}w) = 0\) be a uniformly elliptic equation with ellipticity constants \(0<\lambda\leq\Lambda\) and where \(F(x,0)\equiv 0\). Denote by \(\D_{\lambda,\Lambda} w := P_{\lambda,\Lambda}(\mathcal{H}w)\) the Pucci operator. Then the following hold in the viscosity sense in \(\Omega\).
\begin{enumerate}
	\item If \(\D v \geq 0\), then \(\D_{\lambda,\Lambda}v \geq 0\).
	\item If \(\D u \leq 0\), then \(\D_{\lambda,\Lambda}[-u] \geq 0\).
\end{enumerate}
\end{proposition}

Thus if \(w\) is a \emph{solution} to \(\D w = 0\), then both \(w\) and \(-w\) are subsolutions to \(\D_{\lambda,\Lambda}w=0\).

Given \(\Omega\subseteq\mathbb{R}^n\), we define the classes
\begin{equation}
\begin{aligned}
\uls_{\lambda,\Lambda} &:= \left\{v\;|\; \text{\(v\) is a subsolution to \(\D_{\lambda,\Lambda}w=0\) in \(\Omega\)}\right\},\\
\ols_{\lambda,\Lambda} &:= \left\{u\;|\; \text{\(-u\) is a subsolution to \(\D_{\lambda,\Lambda}w=0\) in \(\Omega\)}\right\},\\
S_{\lambda,\Lambda} &:= \uls_{\lambda,\Lambda}\cap \ols_{\lambda,\Lambda}.
\end{aligned}
\label{Sclass}
\end{equation}
This is the definition given in \cite{MR1351007}, except that they use some other constants in \eqref{xunifdef} and also require the functions to be continuous. \eqref{Sclass} are essential constructions for the derivation of the regularity results, as well as \emph{the strong maximum principle}, in non-linear uniformly elliptic equations.

\begin{figure}[ht]
\centering
\includegraphics{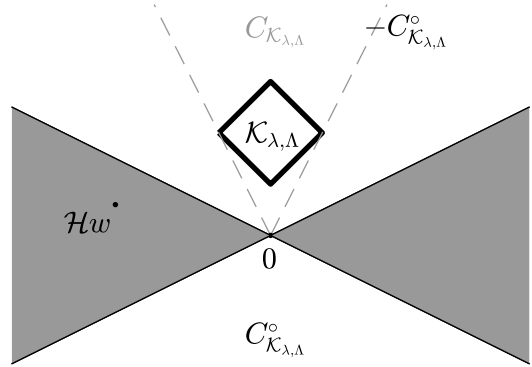}%
\caption{The Hessian of a solution \(w\) to a uniformly elliptic equation is confined to the region in \(S(n)\) marked in gray.}%
\label{fig:strong_comp}%
\end{figure}

There is a geometric interpretation of these classes, at least for \(C^2\) members: Recall the \emph{body cone} \(C_{\mathcal{K}_{\lambda,\Lambda}}\) and the dual \emph{solution cone} \(C_{\mathcal{K}_{\lambda,\Lambda}}^\circ\) in \(S(n)\) to the Pucci operator. We have
\[X\in C_{\mathcal{K}_{\lambda,\Lambda}}^\circ\qquad\iff\qquad P_{\lambda,\Lambda}(X) \leq 0\]
and thus
\[X\in -C_{\mathcal{K}_{\lambda,\Lambda}}^\circ\qquad\iff\qquad P_{\lambda,\Lambda}(-X) \leq 0.\]
This means that a \(C^2\) function \(v\) is a subsolution to the Pucci equation if and only if, at each point, its Hessian lies in the compliment \(\overline{S(n)\setminus C_{\mathcal{K}_{\lambda,\Lambda}}}\). Likewise, \(-v\) is a subsolution if and only if the Hessian lays in \(\overline{S(n)\setminus -C_{\mathcal{K}_{\lambda,\Lambda}}}\).
By Proposition \ref{subsubpuc}, the Hessian of a \(C^2\) solution to a uniformly elliptic equation \(F(x,\mathcal{H}w) = 0\) with \(F(x,0)\equiv 0\) is therefore trapped between the two solution cones \(C_{\mathcal{K}_{\lambda,\Lambda}}^\circ\) and \(-C_{\mathcal{K}_{\lambda,\Lambda}}^\circ\). See Figure \ref{fig:strong_comp}.

The picture literally complements our survey: As we have investigated supersolutions to sublinear equations \(F(\mathcal{H}w)=0\) (or equivalently, subsolutions to \emph{superlinear} equations \(-F(-\mathcal{H}w)=0\)) and thus have been concerned with the double cone \(C_{\mathcal{K}_F}^\circ\cup -C_{\mathcal{K}_F}^\circ\), Caffarelli and Cabré study its exterior.

We state the strong maximum principle valid for supersolutions to uniformly elliptic equations.
\begin{proposition}[Caffarelli, Cabré]\label{strongmax}
Let \(u\in\ols_{\lambda,\Lambda}\) in a connected open subset \(\Omega\subseteq\mathbb{R}^n\). Assume that \(u\geq 0\) and \(u(x_0) = 0\) for some \(x_0\in\Omega\). Then \(u\equiv 0\) in \(\Omega\).
\end{proposition}

\begin{proof}
This is Proposition 4.9 in \cite{MR1351007}, except that we now also allow the supersolution to be only lower semi-continuous.
We bridge the gap by sup-convolutions.

By the definition of \(\ols_{\lambda,\Lambda}\) and Lemma \ref{minlem},
\[v(x) := \max\{-u(x),-1\}\]
is a locally bounded subsolution to the Pucci equation in \(\Omega\). We have \(-1\leq v\leq 0\). Assume \(u(x_0) = 0\), i.e. \(v(x_0) = 0\), and choose a neighbourhood \(D\) of \(x_0\) and an \(\epsilon'>0\), according to Proposition \ref{supconvthm}, so that the restricted sup-convolution
\[v^\epsilon(x) := \sup_{y\in D}\left\{v(y) - \frac{|x-y|^2}{2\epsilon}\right\}\]
is a subsolution in \(D\) for \(0<\epsilon\leq\epsilon'\). We know that \(v^\epsilon\geq v\), but it is easily verified that we also have \(v^\epsilon\leq 0\) and \(v^\epsilon(x_0) = 0\).

Since \(-v^\epsilon\in\ols_{\lambda,\Lambda}\) and \(v^\epsilon\) is continuous, it follows from \cite{MR1351007} that \(v^\epsilon\equiv 0\) in \(D\) and  
\[\max\{-u(x),-1\} = v(x) = \lim_{\epsilon\to 0}v^\epsilon(x) = 0\]
as well.

This means that \(u\equiv 0\) in \(D\), and we have shown that the set
\[N:=\{x\in\Omega\;|\; u(x)=0\} = \{x\in\Omega\;|\; u(x)\leq 0\},\]
which is closed in \(\Omega\) by lower semi-continuity (Proposition \ref{lsceqapp}), also is open. That is, \(N=\Omega\).
\end{proof}

\begin{theorem}[The strong comparison principle]
Let \(u\) be a supersolution and let \(v\) be a subsolution to a uniformly elliptic equation \(F(\mathcal{H}w) = 0\) in a connected open subset \(\Omega\subseteq\mathbb{R}^n\). If \(u\geq v\) and \(u(x_0) = v(x_0)\) at some point \(x_0\in\Omega\) then \(u\equiv v\).
\end{theorem}

\begin{proof}
By Theorem \ref{subsupsub}, the function \(v-u\) is a subsolution to the Pucci equation \(P_{\lambda,\Lambda}(\mathcal{H}w) = 0\) for some ellipticity constants \(0<\lambda\leq \Lambda\). Thus
\[h := u-v = -(v-u) \in \ols_{\lambda,\Lambda}\]
and satisfies \(h\geq 0\) and \(h(x_0) = 0\). Therefore \(h\equiv 0\) and \(u\equiv v\) in \(\Omega\) by Proposition \ref{strongmax}.
\end{proof}

\appendix

\chapter{Appendix}

\section{Semi-continuity}
\begin{definition}
A function \(u\colon\Omega\to(-\infty,\infty]\) is \textbf{lower semi-continuous} (l.s.c.) at \(x_0\in\Omega\) if
\begin{itemize}
	\item \(u(x_0)<\infty\) and given \(\epsilon>0\) there exists a \(\delta>0\) such that
	\[u(x) \geq u(x_0) - \epsilon\qquad\text{whenever}\qquad |x-x_0|<\delta.\]
	\item \(u(x_0)=\infty\) and \(\lim_{x\to x_0}u(x) = \infty\).
\end{itemize}
We say that \(u\) is l.s.c. if it is l.s.c. for all \(x_0\in\Omega\).
\end{definition}

Likewise, a function \(v\colon\Omega\to[-\infty,\infty)\) is \textbf{upper semi-continuous} (u.s.c.) if \(-v\) is l.s.c.

\begin{proposition}\label{lsceqapp}
Let \(u\colon\Omega\to(-\infty,\infty]\). Then the following are equivalent.
\begin{enumerate}
	\item \(u\) is l.s.c.
	\item \begin{equation}
\liminf_{x\to x_0}u(x) := \lim_{r\to 0}\inf_{\substack{x\in B_r(x_0)\\ x\neq x_0}}u(x) \geq u(x_0)
\label{liminfcond}
\end{equation}
at each \(x_0\in\Omega\).
	\item \label{lsc_open}The suplevelset
	\[\Omega_\alpha := \left\{x\in\Omega\;|\; u(x) > \alpha\right\}\]
	is open for every \(\alpha\in\mathbb{R}\).
\end{enumerate}
\end{proposition}

\begin{proof}[Proof of 1\(\iff\) 2]
Let \(u\) be l.s.c. at \(x_0\) and assume first that \(u(x_0)<\infty\). Let \(\epsilon>0\) and choose \(r>0\) so that \(u(x)\geq u(x_0) -\epsilon\) when \(|x-x_0|<r\). Then
\[\inf_{\substack{x\in B_r(x_0)\\ x\neq x_0}}u(x) \geq \inf_{\substack{x\in B_r(x_0)\\ x\neq x_0}}u(x_0)-\epsilon = u(x_0)-\epsilon.\]
Thus \(\liminf_{x\to x_0}u(x) \geq u(x_0) -\epsilon\) for an arbitrary \(\epsilon\) and \eqref{liminfcond} follows.

Let \(M\) be arbitrarily big. Assuming l.s.c. and \(u(x_0) = \infty\), choose \(r>0\) so that \(u(x) \geq M\) when \(|x-x_0|<r\). Then
\[\inf_{\substack{x\in B_r(x_0)\\ x\neq x_0}}u(x) \geq \inf_{\substack{x\in B_r(x_0)\\ x\neq x_0}}M = M\]
and thus \(\liminf_{x\to x_0}u(x) = \infty\).

Now, assume \eqref{liminfcond}. Let \(\epsilon>0\) and \(x_0\in\Omega\) be given and assume first that \(u(x_0)<\infty\). By \eqref{liminfcond} we can choose an \(r>0\) so that \(\inf_{\substack{x\in B_r(x_0)\\ x\neq x_0}}u(x) \geq u(x_0) - \epsilon\).
It follows that
\[u(x_0) - \epsilon \leq \inf_{\substack{y\in B_r(x_0)\\ y\neq x_0}}u(y) \leq u(x)\]
whenever \(0<|x-x_0|<\delta := r\) and obviously also when \(x=x_0\).

Finally, if \(u(x_0) = \infty\) then \(\lim_{x\to x_0}u(x) \geq \lim\inf_{x\to x_0}u(x) \geq u(x_0) = \infty\).
\end{proof}

\begin{proof}[Proof of 1\(\iff\) 3]
Assume that \(u\) is l.s.c. Let \(\alpha\in\mathbb{R}\) and let \(x_0\in\Omega_\alpha\). That is, \(u(x_0)=\alpha+\epsilon\) for some \(\epsilon>0\) assuming first that \(u(x_0)<\infty\). Then we can choose \(\delta>0\) so that
\[u|_{B_\delta(x_0)} \geq u(x_0) - \epsilon/2 = \alpha + \epsilon/2 > \alpha.\]
Thus \(B_\delta(x_0)\subseteq\Omega_\alpha\) and \(\Omega_\alpha\) is open. The same reasoning applies when \(u(x_0)=\infty\).

Assume now that \(\Omega_\alpha\) is open. Let \(x_0\in\Omega\) where \(u(x_0)<\infty\) and let \(\epsilon>0\) be given. Set \(\alpha := u(x_0)-\epsilon\). Then \(x_0\in\Omega_\alpha\) and we can choose a ball \(B(x_0,\delta)\subseteq\Omega_\alpha\) and thus \(|x-x_0|<\delta\) implies \(u(x) > \alpha = u(x_0) - \epsilon\).

If \(u(x_0)=\infty\), let \(\alpha\) be arbitrarily big and choose \(B(x_0,\delta)\subseteq\Omega_\alpha\). Then \(u(x)\geq \alpha\) when \(x\) is close to \(x_0\). 
\end{proof}

\begin{proposition}\label{semicontsup}
Let \(u_i\colon\Omega\to(-\infty,\infty]\), \(i\in\mathcal{I}\) be a nonempty collection of l.s.c. functions. Then the pointwise supremum,
\[u(x) := \sup_{i\in\mathcal{I}}u_i(x),\]
is also l.s.c. in \(\Omega\).
\end{proposition}

\begin{proof}
Let \(x_0\in\Omega\). Then
\begin{align*}
\liminf_{x\to x_0}u(x) &= \liminf_{x\to x_0}\sup_{i\in\mathcal{I}}u_i(x)\\
                       &\geq \liminf_{x\to x_0}u_i(x)\\
											 &\geq u_i(x_0),\qquad by \eqref{liminfcond},
\end{align*}
for all \(i\in\mathcal{I}\). Thus \(\liminf_{x\to x_0}u(x) \geq \sup_{i\in\mathcal{I}}u_i(x_0) = u(x_0)\) and \(u\) is l.s.c. by \eqref{liminfcond}.
\end{proof}

A lower semi-continuous function obtains its minimum on compact sets:
\begin{proposition}\label{lscmin}
Let \(u\colon\Omega\to(-\infty,\infty]\) be l.s.c. and let \(\emptyset\neq D\subseteq\Omega\) be closed and bounded. Then \(u\) is bounded below in \(D\) and there exists a \(x^*\in D\) such that
\[\inf_{x\in D}u(x) = \min_{x\in D}u(x) = u(x^*).\]
\end{proposition}

\begin{proof}
Let \(m := \inf_{x\in D}u(x)\) which may be \(-\infty\). By the property of \(\inf\), there exists a sequence \((x_k)\) in \(D\) such that \(u(x_k)\to m\) as \(k\to\infty\). As \(D\) is compact, there is a subsequence -- which we again denote by \((x_k)\) -- converging to a point \(x^*\in D\). By \eqref{liminfcond},
\[u(x^*) \geq \inf_{x\in D}u(x) = m = \lim_{k\to\infty}u(x_k)\geq \liminf_{k\to\infty}u(x_k) \geq u(x^*).\]
Thus \(m = u(x^*)>-\infty\).
\end{proof}

\section{Stability}
The results of this section are versions of standard stability results. See for instance Lemma 5 in \cite{MR3289084}.
However, we have chosen to state the claims in terms of \emph{pointwise} convergence instead of the usual local uniform convergence.

\begin{lemma}\label{stablemma}
Let \(\Omega\subseteq\mathbb{R}^n\) be open and let \((u_k)\) be an increasing sequence of lower semi-continuous functions converging pointwise to a l.s.c. function \(u\). That is, for all \(x\in\Omega\)
\[u_k(x)\nearrow u(x)\qquad\text{as}\qquad k\to\infty\]
including the case \(u_k(x)\to\infty\) if \(u(x)=\infty\).

Assume further that \(\psi\) is a test-function touching \(u\) strictly from below at some point \(x_0\in\Omega\):
\begin{equation}
\psi(x_0) = u(x_0),\qquad \psi(x)<u(x)\qquad\text{near \(x_0\)}.
\label{psitest}
\end{equation}
Then there exist a sequence \((x_k)\) of points in \(\Omega\) and a sequence of real numbers \((a_k)\) such that
\[x_k \to x_0,\qquad a_k \to 0\]
as \(k\to\infty\) and such that, for all \(k\) big enough,
\begin{equation}
\psi_k(x_k) = u_k(x_k),\qquad \psi_k(x) \leq u_k(x)\qquad\text{near \(x_k\)}
\label{kineq}
\end{equation}
where \(\psi_k := \psi + a_k\).
\end{lemma}

\begin{proof}
Fix \(R>0\) such that \(B_R := B(x_0,R)\subset\subset\Omega\) and such that \(\psi(x)<u(x)\) when
\(0<|x-x_0|\leq R\).

Let \(k\in\mathbb{N}\). By semi-continuity and compactness, we can choose a \(x_k\in\overline{B}_R\) so that
\begin{equation}
a_k := u_k(x_k) - \psi(x_k) = \min_{\overline{B}_R}\{u_k-\psi\}.
\label{kmin}
\end{equation}
By compactness, there is a subsequence -- which we denote again by \((x_k)\) -- converging to a point \(x^*\in\overline{B}_R\). Now let \(k\leq l\). Then
\begin{align*}
0 &= u(x_0) - \psi(x_0) &&\text{by \eqref{psitest}}\\
  &\geq u_l(x_0) - \psi(x_0) &&\text{since \(u_l\nearrow u\)}\\
	&\geq u_l(x_l) - \psi(x_l) &&\text{by \eqref{kmin}}\\
	&\geq u_k(x_l) - \psi(x_l) &&\text{since \(k\leq l\)}.
\end{align*}
The semi-continuity and the pointwise convergence then yields
\begin{align*}
0 &\geq \lim_{k\to\infty}\liminf_{l\to\infty}u_k(x_l) - \psi(x_l)\\
  &\geq \lim_{k\to\infty}u_k(x^*) - \psi(x^*)\\
	&= u(x^*) - \psi(x^*)\\
	&\geq 0
\end{align*}
and we have that \(u(x^*) = \psi(x^*)\).
But \(x_0\) is the unique touching point of \(u\) and \(\psi\) in \(\overline{B}_R\), so \(x^* = x_0\) and thus \(x_k\to x_0\). By definition, \(\psi_k(x_k) = \psi(x_k) + a_k = u_k(x_k)\) and for \(k\) big enough, \(x_k\) is not on the boundary \(\partial B_R\) and hence \(\psi_k(x) \leq u_k(x)\) in a neighbourhood of \(x_k\) within \(B_R\) by \eqref{kmin}.

Finally, a similar calculation will show that \((a_k)\) is a nonpositive increasing sequence with \(\liminf_{l\to\infty}a_l\geq 0\). Therefore \(a_k\to 0\).
\end{proof}

As a tool that sometimes can simplify the situation,
we note that it is not necessary to consider the whole class of \(C^2\) test functions when dealing with sub- and supersolutions to a sublinear equation. The strictly touching quadratics suffices:

\begin{lemma}\label{quadstrict}
A function \(u\colon\Omega\to(-\infty,\infty]\) is a supersolution to the sublinear elliptic equation \(F(\mathcal{H}u) = 0\) if and only if
\begin{enumerate}
	\item \(u\) is lower semi-continuous.
	\item \(u<\infty\) in a dense subset of \(\Omega\)
	\item If \(x_0\in\Omega\) and a quadratic \(\phi\) touches \(u\) strictly from below at \(x_0\), i.e.
\[\phi(x_0) = u(x_0),\qquad \phi(x)< u(x)\qquad\text{for \(x\) near \(x_0\)},\]
then
\[F(\mathcal{H}\phi(x_0)) \leq 0.\]	
\end{enumerate}
The analogous result also holds for subsolutions.
\end{lemma}

\begin{proof}
The claim is obvious if \(u\) is a supersolution. Assume therefore that 1-3 holds and let \(\psi\) be a \(C^2\) test function to \(u\) from below at some point \(x_0\in\Omega\).

For \(\delta>0\) define the adjusted second order Taylor approximation
\[\phi_\delta(x) := \psi(x_0) + \nabla\psi(x_0)(x-x_0) + \frac{1}{2}(x-x_0)^T\mathcal{H}\psi(x_0)(x-x_0) - \frac{\delta}{2}|x-x_0|^2.\]
Then
\[\mathcal{H}\phi_\delta = \mathcal{H}\psi(x_0) - \delta I < \mathcal{H}\psi(x_0)\]
and \(\phi_\delta(x) < \psi(x) \leq u(x)\) for \(x\) near \(x_0\). Thus, by assumption, \(F(\mathcal{H}\phi_\delta)\leq 0\) and by sublinearity we get that
\[F(\mathcal{H}\psi(x_0)) = F(\mathcal{H}\phi_\delta + \delta I) \leq 0 + \delta F(I).\]
That is
\[F(\mathcal{H}\psi(x_0)) \leq \lim_{\delta\to 0}\delta F(I) = 0\]
and \(u\) is a supersolution.

A similar construction yields the result for subsolutions; we get \(F(\mathcal{H}\psi(x_0)) \geq -\lim_{\delta\to 0}\delta F(I) = 0\).
\end{proof}

\begin{lemma}\label{inclemma}
Let \(F(\mathcal{H}w) = 0\) be an elliptic equation. Assume that \((u_k)\) is an increasing sequence of supersolutions converging pointwise to a function \(u\) in \(\Omega\). Then \(u\) is also a supersolution provided it is finite on a dense subset.
\end{lemma}

\begin{proof}
It is well known that the limit of an increasing sequence of l.s.c. functions is l.s.c.

Let \(\phi\) be a test function from below to \(u\) at some point \(x_0\in\Omega\). By Lemma \ref{quadstrict}, we may assume that \(\phi\) is a quadratic and that the touching is strict. By Lemma \ref{stablemma} there is a \(k\in\mathbb{N}\) and a real number \(a_k\) so that
\[\phi_k := \phi + a_k\]
is a test function to \(u_k\) from below at some point \(x_k\in\Omega\). It follows that
\[F(\mathcal{H}\phi(x_0)) = F(\mathcal{H}\phi_k(x_k)) \leq 0\]
since \(\mathcal{H}\phi = \mathcal{H}\phi_k\) is constant.
\end{proof}

Again, the analogous result also holds for a decreasing sequence of subsolutions.

We conclude this section with a standard result. The proof is quite straightforward and omitted.

\begin{lemma}\label{minlem}
Let \(u_1,u_2\) be supersolutions and let \(v_1,v_2\) be subsolutions to an elliptic equation \(F(x,w,\nabla w,\mathcal{H}w) = 0\). Then
\begin{itemize}
	\item \(x\mapsto\min\{u_1(x),u_2(x)\}\) is a supersolution, and
	\item \(x\mapsto\max\{v_1(x),v_2(x)\}\) is a subsolution.
\end{itemize}
\end{lemma}


\section{sup/inf-convolutions}
\begin{definition}\label{recal}
Let \(D\subset\subset\mathbb{R}^n\) and let \(u\colon D\to(-\infty,\infty]\) be lower semi-continuous and finite on a dense subset. For \(\epsilon>0\) we define the \textbf{inf-convolution} \(u_\epsilon\) of \(u\) as
\[u_\epsilon(x) := \inf_{y\in D}\left\{u(y) + \frac{|x-y|^2}{2\epsilon}\right\}.\]
For an upper semi-continuous function \(v\colon D\to[-\infty,\infty)\) that is finite on a dense subset we define the \textbf{sup-convolution} \(v^\epsilon\) of \(v\) as
\[v^\epsilon(x) := \sup_{y\in D}\left\{v(y) - \frac{|x-y|^2}{2\epsilon}\right\}.\]
\end{definition}

\begin{proposition}\label{infthm}
Properties of the inf-convolution.
\begin{enumerate}
	\item \(u_\epsilon\) is semiconcave with concavity constant \(1/\epsilon\).
	\item \(u_\epsilon \nearrow u\) pointwise as \(\epsilon\to 0\).
	\item \(u_\epsilon\) is locally Lipschitz continuous.
	\item Let \(F(\nabla u,\mathcal{H}u) = 0\) be an elliptic equation and assume that \(u\) is a locally bounded   supersolution in \(\Omega\subseteq\mathbb{R}^n\). Then, for each \(x_0\in\Omega\) there are domains \(x_0\in D\subset D'\subset\subset\Omega\) and an \(\epsilon'>0\) so that the infimal convolution \((u|_{D'})_\epsilon\) is a supersolution in \(D\) whenever \(0<\epsilon<\epsilon'\).
\end{enumerate}
\end{proposition}

\begin{proof}[Proof of (1)]\label{infproof}
Let \(x_0\in D\) and let \( y_0\) be the obtained minimum of the inf-convolution in \(\overline{D}\). That is,
\[u_\epsilon(x_0) = \inf_{y\in D}\left\{u(y) + \frac{|x_0-y|^2}{2\epsilon}\right\} = u( y_0) + \frac{|x_0- y_0|^2}{2\epsilon}.\]
The quadratic \(\phi(x) := u( y_0) + \frac{|x- y_0|^2}{2\epsilon}\), \(\mathcal{H}\phi = \frac{1}{\epsilon}\), will then touch \(u_\epsilon\) from above at \(x_0\):
\begin{align*}
u_\epsilon(x_0) &= \phi(x_0).\\
u_\epsilon(x) &\leq u( y_0) + \frac{|x- y_0|^2}{2\epsilon}\\
	&= \phi(x)
\end{align*}
for all \(x\in \overline{D}\).
\end{proof}

\begin{proof}[Proof of (2)]
Fix \(x\in D\). It is obvious from the definition \eqref{recal} that \(u_\epsilon(x)\leq u(x)\) and that \(u_\epsilon\) is increasing when \(\epsilon\) is decreasing. For \(\epsilon>0\) denote the obtained infimum by \(x_\epsilon\). That is,
\begin{equation}
u_\epsilon(x) = u(x_\epsilon) + \frac{|x-x_\epsilon|^2}{2\epsilon} \geq u(x_\epsilon).
\label{obtmin}
\end{equation}
Set
\[m := \inf_{y\in D}u > -\infty.\]
Assume first that \(u(x)<\infty\). Then
\begin{equation}
u(x) \geq u_\epsilon(x) = u(x_\epsilon) + \frac{|x-x_\epsilon|^2}{2\epsilon} \geq m + \frac{|x-x_\epsilon|^2}{2\epsilon}.
\label{ebnd}
\end{equation}
So
\[|x-x_\epsilon|^2 \leq 2\epsilon(u(x) - m) \to 0\]
as \(\epsilon\to 0\). That is, \(x_\epsilon\to x\). By semi-continuity and \eqref{obtmin} we get
\[\liminf_{\epsilon\to 0}u_\epsilon(x) \geq \liminf_{\epsilon\to 0}u(x_\epsilon) \geq u(x)\]
and the claim follows.

Now assume \(u(x) = \infty\). Since \(u_\epsilon(x)\) is increasing as \(\epsilon\to 0\), it either goes to infinity or converges to some \(M\in\mathbb{R}\). We assume the latter and aim for a contradiction.

Similarly as above, the bounds \eqref{ebnd} produces the inequality
\begin{equation}
|x-x_\epsilon|^2 \leq 2\epsilon(M - m)
\label{xebnd}
\end{equation}
and again \(x_\epsilon\to x\). Thus
\[M = \lim_{\epsilon\to 0}u_\epsilon(x) \geq \lim_{\epsilon\to 0}u(x_\epsilon) = \infty\]
by \eqref{obtmin} and semi-continuity.
\end{proof}

\begin{proof}[Proof of (3)]
\(u_\epsilon\) is locally Lipschitz continuous simply because \(u_\epsilon(x) - \frac{1}{2\epsilon}|x|^2\) is concave.
\end{proof}

\begin{proof}[Proof of (4)]
Let \(x_0\in\Omega\). Since \(\Omega\) is open we can choose a radius \(R>0\) so that \(D := B(x_0,R)\), \(D' := B(x_0,2R)\) satisfies \(x_0\in D\subset\overline{D'}\subseteq\Omega\).

Denote again the restricted function \(u|_{D'}\) by \(u\). Then
\[u_\epsilon(x) = \min_{y\in \overline{D'}}\left\{u(y) + \frac{|x-y|^2}{2\epsilon}\right\}.\]
By semi-continuity and since \(u\) is locally bounded above, we can find two numbers \(m<M\) so that
\[m \leq u \leq M\qquad\text{in \(\overline{D'}\).}\]
Fix
\[0<\epsilon<\epsilon' := \frac{R^2}{2(M-m)}.\]
Then, as in \eqref{xebnd}, for any \(x\in D\) the obtained minimum \(\hx\in\overline{D'}\) satisfies
\[|\hx - x| \leq \sqrt{2\epsilon(M-m)} < R\]
and \(\hx\) lays in the \emph{open} ball \(D' = B(x_0,2R)\) since
\begin{equation}
|\hx - x_0| \leq |\hx - x| + |x - x_0| < R + R.
\label{hxbnd}
\end{equation}

Now assume that \(\phi\) is a test-function to \(u_\epsilon\) from below at some point \(x_1\in D\). Then
\begin{align*}
\phi(x_1) &= u_\epsilon(x_1) = u(\hx_1) + \frac{|x_1 - \hx_1|}{2\epsilon} &&\text{for some \(\hx_1\in D'\),}\\
\phi(x) &\leq u_\epsilon(x) &&\\
        &\leq u(y) + \frac{|x - y|}{2\epsilon}
\end{align*}
for all \(x\in D\) near \(x_1\) and for all \(y\in D'\).

Define
\[\psi(x) := \phi(x + x_1 - \hx_1) - \frac{|x_1 - \hx_1|^2}{2\epsilon}.\]
Then \(\psi(\hx_1) = u(\hx_1)\). When \(x\in D'\) is near \(\hx_1\) then \(x+x_1-\hx_1\) is near \(x_1\) and
\begin{align*}
\psi(x) &= \phi(x + x_1 - \hx_1) - \frac{|x_1 - \hx_1|^2}{2\epsilon}\\
        &\leq u(y) + \frac{|x + x_1 - \hx_1 - y|}{2\epsilon} - \frac{|x_1 - \hx_1|^2}{2\epsilon}.
\end{align*}
In particular, for \(y = x\in D'\), \(\psi(x) \leq u(x)\) and \(\psi\) is a test-function to \(u\) at \(\hx_1\in D'\). Thus
\[0 \geq F(\nabla\psi(\hx_1),\mathcal{H}\psi(\hx_1)) = F(\nabla\phi(x_1),\mathcal{H}\phi(x_1)).\]
\end{proof}

Since \(v^\epsilon = -(-v)_\epsilon\), the corresponding properties of sup-convolutions also hold.
\begin{proposition}\label{supconvthm}
Properties of the sup-convolution.
\begin{enumerate}
	\item \(v^\epsilon\) is semiconvex with convexity constant \(1/\epsilon\).
	\item \(v^\epsilon \searrow u\) pointwise as \(\epsilon\to 0\).
	\item \(v^\epsilon\) is Lipschitz continuous.
	\item Let \(F(\nabla u,\mathcal{H}u) = 0\) be an elliptic equation and assume that \(v\) is a locally bounded   subsolution in \(\Omega\subseteq\mathbb{R}^n\). Then, for each \(x_0\in\Omega\) there are domains \(x_0\in D\subset D'\subset\subset\Omega\) and an \(\epsilon'>0\) so that the sup-convolution \((v|_{D'})^\epsilon\) is a subsolution in \(D\) whenever \(0<\epsilon<\epsilon'\).
\end{enumerate}
\end{proposition}

\section{Viscosity, ellipticity, and Pucci}\label{ch:visc_ell}
\begin{definition}
An operator \(\D u = F(x,u,\nabla u,\mathcal{H} u)\) is \textbf{degenerate elliptic} in \(\Omega\subseteq\mathbb{R}^n\) if \(F\) is an increasing function with respect to the matrix argument. That is,
\[X\leq Y\qquad\Rightarrow\qquad F(x,s,q,X) \leq F(x,s,q,Y)\]
for all \((x,s,q) \in \Omega\times\mathbb{R}\times\mathbb{R}^n\).
\end{definition}

The inequality \(X\leq Y\) is a partial ordering in \(S(n)\) defined by \(\xi^TX\xi\leq \xi^TY\xi\) for all \(\xi\in\mathbb{R}^n\).

Degenerate ellipticity is a necessary condition in order to make the concept of viscosity solutions consistent with smooth solutions:
Consider a degenerate elliptic equation
\begin{equation}
F(x,w,\nabla w,\mathcal{H}w) = 0.
\label{pdeapp}
\end{equation}

\begin{definition}\label{visdefapp}
We say that \(u\colon\Omega\to(-\infty,\infty]\) is a \textbf{viscosity supersolution} of the PDE \eqref{pdeapp} if
\begin{enumerate}
	\item \(u\) is lower semi-continuous (l.s.c)
	\item \(u<\infty\) in a dense subset of \(\Omega\)
	\item If \(x_0\in\Omega\) and \(\phi\in C^2\) touches \(u\) from below at \(x_0\), i.e.
\[\phi(x_0) = u(x_0),\qquad \phi(x)\leq u(x)\qquad\text{for \(x\) near \(x_0\)},\]
we require that
\begin{equation}
F(x_0,\phi(x_0),\nabla\phi(x_0),\mathcal{H}\phi(x_0)) \leq 0.
\label{Fapp}
\end{equation}
\end{enumerate}
\end{definition}

The viscosity \emph{sub}solutions \(v\colon\Omega\to[-\infty,\infty)\) are defined in a similar way: they are upper semicontinuous, the test functions touch from above, and the inequality \eqref{Fapp} is reversed. Finally, a function \(w\colon\Omega\to\mathbb{R}\) is a \textbf{viscosity solution} if it is both a viscosity supersolution and a viscosity subsolution. Necessarily, \(w\in C(\Omega)\).

For operators of the form \(\D u = F(x,\mathcal{H} u)\), we also have the stronger notion of \emph{uniform} ellipticity:

\begin{definition}
An operator \(\D u = F(x,\mathcal{H} u)\) is \textbf{uniformly elliptic} in \(\Omega\subseteq\mathbb{R}^n\) if there exist constants \(0<\lambda\leq\Lambda\) such that for all \(X,A\in S(n)\) with \(A\geq 0\), and all \(x\in\Omega\) we have
\[\lambda\tr A \leq F(x,X+A) - F(x,X) \leq \Lambda\tr A.\]
\end{definition}

The inequality on the left shows that uniformly elliptic operators indeed are degenerate elliptic.

\begin{definition}
Let \(\lambda\) and \(\Lambda\) be two constants such that \(0<\lambda\leq\Lambda\). The \textbf{Pucci operator} \(\D_{\lambda,\Lambda} u = P_{\lambda,\Lambda}(\mathcal{H} u)\) is the rotationally invariant sublinear operator given by
\[P_{\lambda,\Lambda}(X) := \max_{Y\in\mathcal{K}_{\lambda,\Lambda}}\langle Y,X\rangle\]
where
\[\mathcal{K}_{\lambda,\Lambda} := \left\{Y\in S(n)\;|\;\lambda I\leq Y\leq \Lambda I\right\}.\]
We also define the \(\infty\)-\textbf{Pucci operator} \(\D_{0,1} u = P_{0,1}(\mathcal{H} u)\) given by
\[P_{0,1}(X) := \max_{Y\in\mathcal{K}_{0,1}}\langle Y,X\rangle\]
where
\[\mathcal{K}_{0,1} := \left\{Y\in S(n)\;|\;0\leq Y\leq I\right\}.\]
\end{definition}

As suggested in the introduction, the corresponding convex body of the Pucci operator is a cube in \(\mathbb{R}^n\).
The extreme points (Definition \ref{extremedef}) should then correspond to the \(2^n\) verticies of the cube.
This can be shown directly, but we have chosen to take a detour in order to shed some light on the underlying features.


Let \(Pr(n)\) denote the set of \(n\times n\) symmetric projection matrices. That is,
\[Pr(n) = \left\{P\in S(n)\;|\; PP = P\right\}.\]
One can easily show that every \(P\in Pr(n)\) is on the form
\[P = \sum_{i=1}^k \eta_i\eta_i^T\]
for some \(k = 0,\dots,n\) and where \(\eta_i\in\mathbb{R}^n\) satisfy \(\eta_i^T\eta_j = \delta_{ij}\).
Notice that
\begin{equation}
\oll(P) = (0,\dots,0,\underbrace{1,\dots,1}_k)^T.
\label{ollP}
\end{equation}

Given a matrix \(X = \sum_{i=1}^n\lambda_i\xi_i\xi_i^T\in S(n)\) we define the projection matrix \(P_X\in Pr(n)\) as
\[P_X := \sum_{\substack{i=1,\\ \lambda_i\geq 0}}^n\xi_i\xi_i^T.\]
When applied to \(X\), the projection \(P_X\) decomposes \(X\) into a difference of two non-negative matrices:
\[X = P_X X + (I-P_X)X = \sum_{\lambda_i\geq 0}\lambda_i\xi_i\xi_i^T - \sum_{\lambda_i< 0}-\lambda_i\xi_i\xi_i^T =: X^+ - X^-.\]
Observe that \((-X)^+ = X^-\) and \((-X)^- = X^+\).

\begin{lemma}\label{projlem}
For each \(X\in S(n)\), we have
\[\max_{P\in Pr(n)}\langle P,X\rangle = \tr X^+\qquad\text{and}\qquad \min_{P\in Pr(n)}\langle P,X\rangle = -\tr X^-.\]
\end{lemma}

\begin{proof}
First we note that if \(Y\geq 0\) and \(A\leq 0\), then \(\langle Y,A\rangle\leq\oll(Y)^T\oll(A)\leq 0\) by Corollary \ref{vonNeu}.
Thus \(\langle P,X\rangle \leq \langle P,X^+\rangle\) for all \(P\in Pr(n)\).
Also,
\[\langle P,X^+\rangle \leq \oll(P)^T\oll(X^+) \leq \mathbbm{1}^T\oll(X^+) = \tr X^+\]
since \(X^+\geq 0\).
This means that
\[\max_{P\in Pr(n)}\langle P,X\rangle \leq \tr X^+.\]
On the other hand, \(P_X\in Pr(n)\) and we therefore also have
\[\max_{P\in Pr(n)}\langle P,X\rangle \geq \tr (P_X X) = \tr X^+.\]
The last claim follows from considering the matrix \(-X\).
\end{proof}

\begin{proposition}\label{extpucciprop}
Let \(0\leq \lambda\leq\Lambda\). The extreme points of \(\mathcal{K}_{\lambda,\Lambda}\) is
\[\mathcal{E}_{\lambda,\Lambda} := \left\{\lambda I + (\Lambda-\lambda)P\;|\; P\in Pr(n)\right\}\subseteq S(n).\]
\end{proposition}

\begin{proof}
First we show that \(\ext\mathcal{K}_{0,1} = Pr(n)\). As suspected,
\begin{align*}
\Phi(\mathcal{K}_{0,1})
	&= \left\{P\oll(Y)\;|\; P\in\mathcal{P}(n),\, Y\in\mathcal{K}_{0,1}\right\}\\
	&= \left\{P(y_1,\dots,y_n)^T\;|\; P\in\mathcal{P}(n),\, 0\leq y_1\leq\cdots\leq y_n\leq 1\right\}\\
	&= \left\{(y_1,\dots,y_n)^T\in\mathbb{R}^n\;|\; 0\leq y_i\leq 1\right\}\\
	&= [0,1]^n
\end{align*}
is the unit cube in \(\mathbb{R}^n\).

The set of orthogonal projection matrices \(Pr(n)\) is clearly symmetric in \(S(n)\) (Definition \ref{symbodies}). By \eqref{ollP},
\begin{align*}
\Phi(Pr(n))
	&= \left\{P\oll(Y)\;|\; P\in\mathcal{P}(n),\, Y\in Pr(n)\right\}\\
	&= \left\{P(0,\dots,0,\underbrace{1,\dots,1}_k)^T\;|\; P\in\mathcal{P}(n),\, 0\leq k\leq n\right\}\\
	&= \left\{(\delta_1,\dots,\delta_n)^T\in\mathbb{R}^n\;|\; \text{\(\delta_i = 0\) or 1}\right\}
\end{align*}
which are the vertices of the cube. That is, \(\Phi(Pr(n)) = \ext\Phi(\mathcal{K}_{0,1}) = \Phi(\ext\mathcal{K}_{0,1})\) by part 2 of Theorem \ref{Phithm}. It follows that \(Pr(n) = \ext\mathcal{K}_{0,1}\) by the injectivity of \(\Phi\).

Finally, since the extreme points are invariant under scalings and shifts (Proposition \ref{extscaleshiftinv}), we get that
\begin{align*}
\mathcal{E}_{\lambda,\Lambda} &= \lambda\{I\} + (\Lambda-\lambda)Pr(n)\\
  &= \lambda\{I\} + (\Lambda-\lambda)\ext\mathcal{K}_{0,1}\\
	&= \ext\Big(\lambda\{I\} + (\Lambda-\lambda)\mathcal{K}_{0,1}\Big)\\
	&= \ext\mathcal{K}_{\lambda,\Lambda}, &&\eqref{pucscaleshift}.
\end{align*}
\end{proof}

We are now ready to give the ``extreme point-representation'' of the Pucci operator. It turns out to be a well-known formula. It becomes useful in the proof of the next proposition.

\begin{equation}
\begin{aligned}
P_{\lambda,\Lambda}(X) &= \max_{Y\in\mathcal{K}_{\lambda,\Lambda}}\langle Y,X\rangle\\
	&= \max_{Y\in\mathcal{E}_{\lambda,\Lambda}}\langle Y,X\rangle\\
	&= \max_{P\in Pr(n)}\tr\left(\left(\lambda I + (\Lambda-\lambda)P\right)X\right)\\
	&= \lambda\tr X + (\Lambda-\lambda)\max_{P\in Pr(n)}\tr (PX)\\
	&= \lambda\tr X + (\Lambda-\lambda)\tr X^+\\
	&= \Lambda\tr X^+ - \lambda\tr X^-.
\end{aligned}
\label{puccialtdef}
\end{equation}

\begin{proposition}[Alternative definition of uniform ellipticity]\label{altelldef}
An operator \(\D u = F(x,\mathcal{H} u)\) is uniformly elliptic in \(\Omega\subseteq\mathbb{R}^n\) with ellipticity constants \(0<\lambda\leq\Lambda\) if and only if
\begin{equation}
F(x,X+Y) \leq F(x,X) + P_{\lambda,\Lambda}(Y)
\label{altunifdef}
\end{equation}
for all \(X,Y\in S(n)\) and all \(x\in\Omega\).
\end{proposition}

\begin{proof}
Assume first that \eqref{altunifdef} holds and let \(X,A\in S(n)\), \(A\geq 0\). Then
\[F(x,X+A) - F(x,X) \leq P_{\lambda,\Lambda}(A) = \Lambda\tr A\]
and
\[\lambda\tr A = -P_{\lambda,\Lambda}(-A) \leq F(x,X+A) - F(x,X)\]
by \eqref{puccialtdef} and \(\D u\) is therefore uniformly elliptic.

Next, assume that \(\D u\) is uniformly elliptic and let \(X,Y\in S(n)\).
Then
\[F(x,X+Y) = F(x,X - Y^- + Y^+) \leq F(x,X-Y^-) + \Lambda\tr Y^+\]
and
\[F(x,X) = F(x,X - Y^- + Y^-) \geq F(x,X-Y^-) + \lambda\tr Y^-.\]
Putting it together yields
\[F(x,X+Y) \leq F(X) + \Lambda\tr Y^+ - \lambda\tr Y^-\]
which is \eqref{altunifdef} by \eqref{puccialtdef}.
\end{proof}

Proof of Proposition \ref{definite} from Section \ref{sec:conv_body}:

\begin{proof}
Let \(F(X) = \max_{Y\in\mathcal{K}}\langle Y,X\rangle\) be degenerate elliptic and assume for the sake of contradiction that there is a matrix \(Y_0 =\sum_{i=1}^n\lambda_i\xi_i\xi_i^T\in\mathcal{K}\) with a negative eigenvalue, \(\lambda_1<0\). But \(0\geq -\xi_1\xi_1^T\), so
\[0\geq F(-\xi_1\xi_1^T) = \max_{Y\in\mathcal{K}}\langle Y,-\xi_1\xi_1^T\rangle = \max_{Y\in\mathcal{K}} -\xi_1^T Y\xi_1 \geq -\xi_1^T Y_0\xi_1 = -\lambda_1 > 0.\]

Assume now that every matrix in \(\mathcal{K}\) is positive semidefinite. Let \(X\leq Z\) and write \(A := X-Z\leq 0\). By corollary \ref{vonNeu} \(\langle Y,A\rangle \leq 0\) whenever \(Y\geq 0\). It follows that
\[F(X) = F(Z+A) \leq F(Z) + F(A) \leq F(Z)\]
and \(F\) is degenerate elliptic.

By Proposition \ref{altelldef}, and since \(F(0)=0\), an operator \(F(X) = \max_{Y\in\mathcal{K}}\langle Y,X\rangle\) is uniformly elliptic with ellipticity constants \(0<\lambda\leq\Lambda\) if and only if \(F(X) \leq P_{\lambda,\Lambda}(X)\) for all \(X\in S(n)\). But that is again equivalent to the inclusion
\begin{align*}
\mathcal{K} &= \left\{Z\in S(n)\;|\; \langle Z,X\rangle  \leq F(X)\;\forall X\in S(n)\right\}\\
            &\subseteq \left\{Z\in S(n)\;|\; \langle Z,X\rangle  \leq P_{\lambda,\Lambda}(X)\;\forall X\in S(n)\right\}\\
						&= \mathcal{K}_{\lambda,\Lambda}.
\end{align*}

\end{proof}

\backmatter
\bibliographystyle{alpha}
\bibliography{references}

\begin{thebibliography}{JLM01}

\bibitem[Bha97]{MR1477662}
Rajendra Bhatia.
\newblock {\em Matrix analysis}, volume 169 of {\em Graduate Texts in
  Mathematics}.
\newblock Springer-Verlag, New York, 1997.

\bibitem[Bru18]{bru18}
Karl~K. Brustad.
\newblock Superposition of \(p\)-superharmonic functions.
\newblock {\em Advances in Calculus of Variations}, 2018.

\bibitem[CC95]{MR1351007}
Luis~A. Caffarelli and Xavier Cabr\'e.
\newblock {\em Fully nonlinear elliptic equations}, volume~43 of {\em American
  Mathematical Society Colloquium Publications}.
\newblock American Mathematical Society, Providence, RI, 1995.

\bibitem[dP70]{MR0435422}
Nicolaas du~Plessis.
\newblock {\em An introduction to potential theory}.
\newblock Hafner Publishing Co., Darien, Conn.; Oliver and Boyd, Edinburgh,
  1970.
\newblock University Mathematical Monographs, No. 7.

\bibitem[JLM01]{MR1871417}
Petri Juutinen, Peter Lindqvist, and Juan~J. Manfredi.
\newblock On the equivalence of viscosity solutions and weak solutions for a
  quasi-linear equation.
\newblock {\em SIAM J. Math. Anal.}, 33(3):699--717, 2001.

\bibitem[Kat15]{MR3289084}
Nikos Katzourakis.
\newblock {\em An introduction to viscosity solutions for fully nonlinear {PDE}
  with applications to calculus of variations in {$L^\infty$}}.
\newblock SpringerBriefs in Mathematics. Springer, Cham, 2015.

\bibitem[Koi04]{MR2084272}
Shigeaki Koike.
\newblock {\em A beginner's guide to the theory of viscosity solutions},
  volume~13 of {\em MSJ Memoirs}.
\newblock Mathematical Society of Japan, Tokyo, 2004.

\bibitem[Lin12]{MR2953377}
Peter Lindqvist.
\newblock Regularity of supersolutions.
\newblock In {\em Regularity estimates for nonlinear elliptic and parabolic
  problems}, volume 2045 of {\em Lecture Notes in Math.}, pages 73--131.
  Springer, Heidelberg, 2012.

\bibitem[Sch14]{MR3155183}
Rolf Schneider.
\newblock {\em Convex bodies: the {B}runn-{M}inkowski theory}, volume 151 of
  {\em Encyclopedia of Mathematics and its Applications}.
\newblock Cambridge University Press, Cambridge, expanded edition, 2014.

\end{thebibliography}

\end{document}